\documentclass[11pt, a4paper,reqno]{amsart}
\pdfoutput=1
\usepackage{tikz,genyoungtabtikz,enumitem,rotating,latexsym,bm,stmaryrd}
\usetikzlibrary{positioning,intersections}
\usepackage{amsmath,amsthm,amsfonts,amssymb,mathrsfs,pb-diagram}
\usepackage[pagebackref=false,bookmarks=true,colorlinks=true,linktoc=page,citecolor=darkgreen,linkcolor=blue,urlcolor=mediumblue]{hyperref}
\usepackage{caption}
\usepackage{lipsum,wasysym}
\usepackage{mathtools}
\usepackage{cleveref,fullpage}
\usepackage{tikz,enumitem,rotating,latexsym,bm,stmaryrd}
\usepackage{mathdots}
\usetikzlibrary{positioning,intersections,shapes.geometric,calc,decorations.pathreplacing}

\usepackage{xcolor}
\definecolor{mediumblue}{rgb}{0.0, 0.0, 0.8}
\colorlet{darkgreen}{green!50!black}
\newcommand{\stirling}[2]{\genfrac{[}{]}{0pt}{}{#1}{#2}}

\newcommand\ptn[1]{\mathscr{P}_{#1}}
\newcommand\regptn[2]{\mathscr{R}^{#1}_{#2}}
\renewcommand{\geq}{\geqslant}

\renewcommand{\leq}{\leqslant}
\renewcommand{\trianglerighteq}{\trianglerighteqslant}
\renewcommand{\trianglelefteq}{\trianglelefteqslant}

\tikzset{wei/.style=%{red,double=pink,thick,doubledistance=1.5pt}}
{red,double=red,double
distance=1pt}}

\usepackage{ltxtable}
\usepackage{tabularx}
\newcommand\gr{\Yfillcolour{magenta!30}}
\newcommand\dgr{\Yfillcolour{blue!45!cyan!30}}
\newcommand\wh{\Yfillcolour{white}}

\tikzset{wei2/.style={red,double=red,double
distance=1pt}}

\allowdisplaybreaks
\numberwithin{equation}{section}
\usepackage{scalefnt}

\newtheorem{thm}{Theorem}[section]
\newtheorem{cor}[thm]{Corollary}
\newtheorem*{saxl}{Saxl's conjecture}

\newtheorem*{ack}{Acknowledgements}

\newtheorem{conj}[thm]{Conjecture}

\newtheorem*{conjA}{Generalised Saxl Conjecture}
\newtheorem*{conjB}{Strengthened Saxl Conjecture}

\newtheorem{lem}[thm]{Lemma}
\newtheorem{prop}[thm]{Proposition}
\newtheorem*{prop*}{Proposition}
\newtheorem*{thmA*}{Theorem A}
\newtheorem*{thmB*}{Theorem B}
\newtheorem*{thmC*}{Theorem C}

\newtheorem*{cor*}{Corollary}

\newtheorem*{conj*}{Semisimplicity  Conjecture}

\theoremstyle{remark}

\newtheorem{rmk}[thm]{Remark}  %numbering?

\newtheorem*{Acknowledgements*}{Acknowledgements}

\theoremstyle{definition}
\newtheorem{defn}[thm]{Definition}
\newtheorem{eg}[thm]{Example}

\newcommand{\boxla}{\lambda}

\newcommand{\rad}{\mathrm{rad}}
\newcommand{\res}{\mathrm{res}}

\newcommand{\Std}{{\rm Std}}
\newcommand{\SStd}{{\rm SStd}}
\newcommand{\RStd}{{\rm RStd}}

\newcommand{\Shape}{\operatorname{Shape}}

\newcommand{\la}{\lambda}

\newcommand{\N}{{\bf N}}

 \newcommand{\SSTS}{\mathsf{S}}

\newcommand{\SSTA}{\mathsf{A}}
\newcommand{\SSTT}{\mathsf{T}}  %semistandard index
\newcommand{\SSTU}{\mathsf{U}}  %semistandard index
  \newcommand{\SSTR}{\mathsf{R}}  %semistandard index
\newcommand{\sts}{\mathsf{s}}  %standard index
\newcommand{\stt}{\mathsf{t}}  %standard index
\newcommand{\Lad}{{\mathsf{L}}}  %standard index
\newcommand{\ZZ}{{\mathbb Z}}
\newcommand{\NN}{{\mathbb N}}

\newcommand{\CC}{\mathbb{C}}

\DeclareMathOperator{\Hom}{Hom}

\let\<=\langle
\let\>=\rangle

 \newcommand\releven{11}
  \newcommand\rsix{16}
 \newcommand\rten{10} \newcommand\rtwelve{12}
 \newcommand\rthirt{13}
 \newcommand\rfour{14}
 \newcommand\rfive{15}

\tikzset{
    ultra thin/.style= {line width=0.05pt},
    very thin/.style=  {line width=0.2pt},
    thin/.style=       {line width=0.1pt},
    semithick/.style=  {line width=0.6pt},
    thick/.style=      {line width=0.8pt},
    very thick/.style= {line width=1.2pt},
    ultra thick/.style={line width=1.6pt}
}

\crefname{defn}{Definition}{Definitions}
\crefname{thm}{Theorem}{Theorems}
\crefname{prop}{Proposition}{Propositions}
\crefname{lem}{Lemma}{Lemmas}
\crefname{cor}{Corollary}{Corollaries}
\crefname{conj}{Conjecture}{Conjectures}
\crefname{section}{Section}{Sections}
\crefname{subsection}{Subsection}{Subsections}
\crefname{eg}{Example}{Examples}
\crefname{figure}{Figure}{Figures}
\crefname{rem}{Remark}{Remarks}
\crefname{rmk}{Remark}{Remarks}
\crefname{equation}{equation}{equation}

\Crefname{defn}{Definition}{Definitions}
\Crefname{thm}{Theorem}{Theorems}
\Crefname{prop}{Proposition}{Propositions}
\Crefname{lem}{Lemma}{Lemmas}
\Crefname{cor}{Corollary}{Corollaries}
\Crefname{conj}{Conjecture}{Conjectures}
\Crefname{section}{Section}{Sections}
\Crefname{subsection}{Subsection}{Subsections}
\Crefname{eg}{Example}{Examples}
\Crefname{figure}{Figure}{Figures}
\Crefname{rem}{Remark}{Remarks}
\Crefname{rmk}{Remark}{Remarks}

 \newlength{\mylen}
\setbox1=\hbox{$\bullet$}\setbox2=\hbox{\tiny$\bullet$}
\newcommand{\ten}{10}
\newcommand{\eleven}{11}

\newcommand\Dim[2][t]{\text{\rm Dim}_{#1}#2}

\usepackage{amsmath}
\newcommand\Item[1][]{%
  \ifx\relax#1\relax  \item \else \item[#1] \fi
  \abovedisplayskip=0pt\abovedisplayshortskip=0pt~\vspace*{-\baselineskip}}

\def\Item{\item\abovedisplayskip=0pt\abovedisplayshortskip=5pt~\vspace*{-\baselineskip}}

\begin{document}

\title[Kronecker positivity  and 2-modular representation theory]
{Kronecker positivity  \\  and 2-modular representation theory}
\author{C. Bessenrodt}
\address{Institut f\"ur Algebra, Zahlentheorie und Diskrete Mathematik, Leibniz Universit\"at Hannover, D-30167 Hannover, Germany}
\email{bessen@math.uni-hannover.de}

\author[C. Bowman]{C. Bowman}
\address{School of Mathematics, Statistics and Actuarial Science University of Kent Canterbury
CT2 7NF, UK}
\email{C.D.Bowman@kent.ac.uk}

\author{L. Sutton}
\address{National University of Singapore, 10 Lower Kent Ridge Road, Singapore 119076}
\email{matloui@nus.edu.sg}

\maketitle

\!\!\!\!\!\!\!\!\!
  \begin{abstract}
 This paper consists of  two prongs.
Firstly, we prove that any Specht module labelled by a 2-separated partition   is semisimple   and
 we completely determine its decomposition as a direct sum of graded simple modules.
   Secondly, we apply these results and other modular representation theoretic techniques on the study of Kronecker coefficients and hence verify  Saxl's conjecture for several  large new families  of partitions.
   In particular, we verify Saxl's conjecture for all height zero characters of $\mathfrak{S}_n$ for $p=2$.
  \end{abstract}

 \section*{Introduction}

This paper brings together, for the first time, the two   oldest open problems in the representation theory of the symmetric groups and their  quiver  Hecke algebras.
The first   problem is to  understand    the
   structure  of Specht modules
and the second is to describe  the decomposition of a tensor product of two Specht modules --- the infamous {\em Kronecker problem}.

 \medskip
\noindent{\bf Kronecker positivity:}
The Kronecker problem  is not only one of the central open problems in
the classical representation theory of the symmetric groups, but it is also one of  the  definitive open problems in algebraic combinatorics
 as   identified by Richard Stanley in \cite{Sta00}.  The problem of deciding the positivity of Kronecker coefficients arose in recent times also in quantum information theory  \cite{ky,MR2197548,MR2276458,MR3186652}  and Kronecker coefficients have subsequently been used to study entanglement entropy  \cite{entropy}.

    A new benchmark for the Kronecker positivity problem  is a conjecture of
 Heide, Saxl, Tiep and Zalesskii \cite{MR3056296} that was inspired by their investigation of the square of the Steinberg character for simple groups of Lie type.
It says that for any $n\neq  2, 4, 9$ there is always a
 complex irreducible character of $\mathfrak{S}_n$ whose square contains all irreducible characters of $\mathfrak{S}_n$ as constituents.
For $n $ a triangular number,   an explicit candidate was suggested by Saxl in 2012:
 Let $\rho:=\rho(k)=(k,k-1,k-2,k-3,\dots,2,1)$ denote the $k$th staircase partition.
Phrased in terms of modules,
Saxl's conjecture states that all simple modules appear in  the tensor square of
the simple $\CC\mathfrak{S}_n$-module ${\bf D}^\CC(\rho)$.  In other words,   we have that
$${\bf D}^\CC(\rho)\otimes {\bf D}^\CC(\rho) = \bigoplus _{\lambda } g(\rho,\rho,\lambda  ) {\bf D}^\CC(\lambda)$$
with $g(\rho,\rho,\lambda  ) \neq 0$ for all partitions $\lambda$ of $n$.
This conjecture has been attacked by algebraists, probabilists, and  complexity theorists
  \cite{MR3856528,I15,LS,PP16} but remains stubbornly elusive.
 Positivity of the Kronecker coefficient $g(\rho,\rho,\lambda  )$ has  been verified
 for hooks and two-row partitions when $n$ is sufficiently large in \cite{PP16},
 and then for arbitrary $n$ and
 $\lambda$ a hook  in \cite{I15,MR3856528} or a double-hook partition (i.e., when the Durfee size is 2) in \cite{MR3856528},´
   and   for any $\lambda$ comparable to $\rho$ in dominance order in \cite{I15}.

This paper begins with the  observation that the $\Bbbk \mathfrak{S}_n$-module  ${\bf D}^\Bbbk(\rho)$
is   projective   over a field $\Bbbk$ of  characteristic $p=2$,
or equivalently, that the character to  the Specht module
${\bf D}^\CC(\rho) = {\bf S}^\CC(\rho)$ is the character $\xi^\rho$
associated to a projective indecomposable $\Bbbk \mathfrak{S}_n$-module
(via its integral lift to characteristic 0).
Therefore, the tensor square of ${\bf D}^\Bbbk(\rho)$ is again a projective module,
and the square of $\xi^\rho$ is the character to a projective module.
This allows us
  to {bring to bear the tools of modular and graded representation theory}  on the study of the Kronecker coefficients.
In particular, we deduce that if ${\bf D}^\Bbbk(\la)=  {\bf S}^\Bbbk(\la)$  is a  simple Specht module,   then
  all constituents of the projective cover of   ${\bf D}^\Bbbk(\la)$ must also appear in  Saxl's tensor-square.
 For example, using this property for the trivial simple module ${\bf D}^\Bbbk((n))$
of $\mathfrak{S}_n$ at characteristic~2
gives all characters of odd degree as constituents in the Saxl square;
more generally, we will detect all irreducible characters of height~0
as constituents.
Our aim is to understand the columns of the   2-modular graded decomposition matrix  which are labelled by   {\em simple Specht modules}   and to utilise these results towards Saxl's conjecture.

 \medskip
\noindent{\bf Modular representation theory:}
The   classification of  simple Specht modules for symmetric groups and their
Hecke algebras was a massive undertaking involving over 30 years work \cite{James,MR2207757,MR1687552,MR2137291,MR1402572,MR2339470,MR1425323,MR2089249,MR2488560,MR3011340,MR2588144}.
 The pursuit of a description of  semisimple and decomposable Specht modules  is similarly old \cite{James} and yet has proven a much more difficult nut to crack.
  The decomposable  Specht modules
   labelled by hook partitions  were characterised  by  Murphy   and Speyer   \cite{MR591250,MR3250450};
      the graded decomposition numbers of these Specht modules were calculated by Chuang, Miyachi, and Tan \cite{MR2081930};
  the first examples of   decomposable Specht modules labelled by non-hook partitions were given by Dodge and Fayers   \cite{MR2905253};
 Donkin and Geranios  very recently   unified and extended these results to certain ``framed staircase" partitions \cite{DG} which we will  discuss (within the wider context of ``2-separated" partitions) below.

 We show that any Specht module labelled by a 2-separated partition   is semisimple   and
 we completely determine its decomposition as a sum of graded simple modules.
  Our proof makes heavy use of recent results in the {\em graded} representation theory of   Hecke and rational Cherednik algebras.
 We shall denote the quantisations of the Specht and simple modules by
$ {\bf S}^\Bbbk_{q}(\la)$ and  ${\bf D}^\Bbbk_{q}(\la)$ respectively over $\Bbbk$.   We  completely determine the rows  of the graded decomposition matrix of $H^\CC_{q}(n)$ labelled by 2-separated partitions;  this serves as a first approximation to our goal and subsumes and generalises the results on decomposability and decomposition numbers of   Specht modules for hook partitions
 (belonging to blocks of small $2$-core)
 \cite{MR3250450,MR2081930}, and
  results on  decomposition numbers of Specht modules in blocks of enormous 2-cores  \cite{MR1402572}.

 \bigskip
\noindent
{\bf Graded decomposition numbers of semisimple Specht modules:}
The partitions of interest to us (for both Saxl's conjecture and our   decomposability classification)  are the  2-separated partitions.
    Such partitions are obtained by  taking a staircase partition, $\tau$, and
   adding 2 copies of a partition $\lambda$   to the right of  $\tau$ and 2 copies of  a partition   $\mu$ to the bottom of $\tau$ in such a way that $\lambda$ and $\mu$ do not touch (except perhaps diagonally).
   Such partitions, denoted $\tau^\la_\mu$
    can be pictured as in \cref{figure2quosep}.

\!\!\!\!\!\begin{figure}[ht!]$$
\begin{minipage}{3cm}\begin{tikzpicture}[scale=0.3]
 \draw[fill=cyan!35](-0.5,-0.5)--(5.5,-0.5)--++(-90:0.5)--++(-180:0.5)--++(-90:0.5)--++(-180:0.5)--++(-90:0.5)--++(-180:0.5)--++(-90:0.5)--++(-180:0.5)--++(-90:0.5)--++(-180:0.5)--++(-90:0.5)--++(-180:0.5)--++(-90:0.5)--++(-180:0.5)--++(-90:0.5)--++(-180:0.5)--++(-90:0.5)--++(-180:0.5)--++(-90:0.5)--++(-180:0.5)--++(-90:0.5)--++(-180:0.5)--++(-90:0.5)--++(-180:0.5)--(-0.5,-0.5);

\draw[dashed] (3.5,0)--(3.5,-9.5);
\draw[dashed] (-1,-3)--(9.5,-3);
\draw(7.5,-1.5) node {$+2\lambda$};
\draw(1,-8) node [rotate=-90] {$+2\mu$};
\end{tikzpicture}
\end{minipage}
\quad\quad\quad =\quad \quad
\begin{minipage}{3cm}
\begin{tikzpicture}[scale=0.3]
%\draw(0,0)--(5.5,0)--(5.5,-0.5)--(0.5,-0.5)--(0.5,-6)--(0,-6)--(0,0);
\draw[fill=cyan!35](-0.5,-0.5)--(5.5,-0.5)--++(-90:0.5)--++(-180:0.5)--++(-90:0.5)--++(-180:0.5)--++(-90:0.5)--++(-180:0.5)--++(-90:0.5)--++(-180:0.5)--++(-90:0.5)--++(-180:0.5)--++(-90:0.5)--++(-180:0.5)--++(-90:0.5)--++(-180:0.5)--++(-90:0.5)--++(-180:0.5)--++(-90:0.5)--++(-180:0.5)--++(-90:0.5)--++(-180:0.5)--++(-90:0.5)--++(-180:0.5)--++(-90:0.5)--++(-180:0.5)--(-0.5,-0.5);
\draw[dashed] (3.5,0)--(3.5,-9.5);
\draw[dashed] (-1,-3)--(9.5,-3);

\draw(5.5,-0.5)--(10.5,-0.5)--++(-90:0.5)--++(0:-0.5)--++(-90:0.5)--++(0:-1.5)--++(-90:0.5)--++(0:-2.5)--++(-90:0.5)  --++(-0:-0.5)--++(-90:0.5)--++(-0:-2)
--++(-90:1.5)--++(0:-0.5)--++(-90:0.5)--++(0:-0.5)--++(-90:0.5)--++(0:-0.5)--++(-90:0.5)--++(0:-0.5)--++(-90:0.5)--++(0:-0.5)--++(-90:0.5)
--++(0:-0.5)--++(-90:1.5)--++(0:-0.5)--++(-90:0.5)--++(0:-0.5)--++(90:3.5);
\end{tikzpicture}\end{minipage}
$$

\!\!\!\!\!
\caption{A 2-separated partition $\tau^\la_\mu$ (see \cref{2sepdef}).}
\label{figure2quosep}
\end{figure}
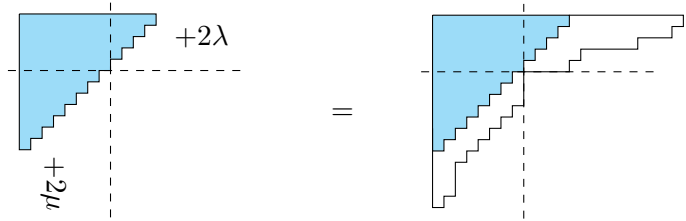

 Notice that if the weight of a block is small compared to the size of the core, then all partitions in that block are 2-separated.
 %We are now able to state our first  main theorem as follows.
   We emphasise that the size of the staircase $\rho(k)$ in the following statement is immaterial (provided  that  $k\geq \ell(\lambda)+\ell(\mu^T)$, where $\ell(\la)$ denotes the length of the partition $\la$),
 and so we simply write $\tau:=\rho(k)$.
   For those interested in  the extra graded structure, we refer the reader to the full statement in   \cref{gradedresult}.

\begin{thmA*}
Let $e=2$ and let
$\tau^\la_\mu$ denote a 2-separated partition of $n$.  We have that the $H_{-1} ^\mathbb{C} (n)$-module
$ {\bf S}^\CC_{-1}({
 {\tau^{\la }
_
{\mu}
}
})  $ is semisimple and decomposes as a direct sum of simples as follows
$$
 {\bf S}^\CC_{-1}({
 {\tau^{\la }
_
{\mu}
}
})   = \bigoplus_\nu     c(\nu^T, \lambda^T,\mu)  {\bf D}^\CC_{-1}({
 {\tau^{\nu }
_
{\varnothing}
}
})
$$
 where $   c(\nu^T, \lambda^T,\mu) $ is the Littlewood--Richardson coefficient labelled by this triple of partitions.
\end{thmA*}

In particular, there exist many blocks of $H^\CC_{-1}(n)$ (those with large cores) for which all Specht modules in the block  are semisimple.
  In \cite{MR2905253} Dodge and Fayers remark that {\em``every known example of a decomposable Specht module is labelled by a 2-separated partition"} and
  {\em ``it is interesting to speculate whether the 2-separated condition is necessary for a Specht module to be decomposable"}.
  In fact in Section 6  we show that their speculation  is {\em not} {true} by exhibiting two infinite  families of decomposable Specht modules  obtained by ``inflating" the smallest  decomposable Specht module (indexed by $(3,1^2)$).

Theorem A implies that all known examples of decomposable Specht modules for $\mathfrak{S}_n$ are obtained by reduction modulo $p=2$ from decomposable semisimple Specht modules for $H_{-1}^\CC(n)$.

 \bigskip
\noindent{\bf Applications to Kronecker coefficients:}
We now discuss the  results and insights  which 2-modular representation theory affords us in the study of Kronecker coefficients.  We verify
the positivity of the Kronecker coefficients in
Saxl's conjecture for a large new class of partitions,
 and
propose conjectural strengthened and generalised versions of Saxl's original conjecture.
  Our first main theorem on Kronecker coefficients is as follows:

\begin{thmB*}
Let $\lambda \vdash n=k(k+1)/2$ such that $\chi^\lambda$ is of height~0.
Then $g(\rho,\rho,\lambda)>0$.
 In particular, all $\chi^\lambda$ of odd degree are constituents of
the Saxl square.
 \end{thmB*}

\begin{figure}[ht!]
$$ \begin{minipage}{2.8cm}\begin{tikzpicture} [scale=0.33]

 \draw[very thick](0,0)--++(0:9)--++(-90:1)--++(180:1)--++(-90:1)
 --++(180:5)--++(-90:1)
 --++(180:1)--++(-90:3)
  --++(180:1)--++(-90:2)
    --++(180:1)--++(90:8)
  ;

  \draw[very thick,fill=cyan!35](0,0)--++(0:3)
--++(-90:1)--++(180:1)
--++(-90:1)--++(180:1)
--++(-90:1)--++(180:1)
--++(90:3);
  \draw[very thick](0,0)--++(0:7)--++(-90:1)--++(180:4);

 \draw[very thick](0,0)--++(-90:5)--++(0:1)--++(90:2);

  \clip((0,0)--++(0:9)--++(-90:1)--++(180:1)--++(-90:1)
 --++(180:5)--++(-90:1)
 --++(180:1)--++(-90:3)
  --++(180:1)--++(-90:2)
    --++(180:1)--++(90:8)
  ;

  \path(0,0) coordinate (origin);
   \foreach \i in {1,...,19}
  {
    \path (origin)++(0:1*\i cm)  coordinate (a\i);
    \path (origin)++(-90:1*\i cm)  coordinate (b\i);
    \path (a\i)++(-90:10cm) coordinate (ca\i);
    \path (b\i)++(0:10cm) coordinate (cb\i);
    \draw[densely dotted] (a\i) -- (ca\i)  (b\i) -- (cb\i); }

 \path(0.5,-0.5) coordinate (origin);
 \foreach \i in {1,...,19}
  {
    \path (origin)++(0:1*\i cm)  coordinate (a\i);
    \path (origin)++(-90:1*\i cm)  coordinate (b\i);
    \path (a\i)++(-90:1cm) coordinate (ca\i);
        \path (ca\i)++(-90:1cm) coordinate (cca\i);
    \path (b\i)++(0:1cm) coordinate (cb\i);
    \path (cb\i)++(0:1cm) coordinate (ccb\i);
  }
 \draw[very thick,fill=cyan!35](0,0)--++(0:3)
--++(-90:1)--++(180:1)
--++(-90:1)--++(180:1)
--++(-90:1)--++(180:1)
--++(90:3);
   \end{tikzpicture}\end{minipage} \qquad
   %%%
   %%%
   \begin{minipage}{2.8cm}\begin{tikzpicture} [scale=0.33]

 \draw[very thick](0,0)--++(0:9)--++(-90:1)--++(180:3)--++(-90:1)
 --++(180:1)--++(-90:1)
 --++(180:3)--++(-90:3)
  --++(180:1)--++(-90:2)
    --++(180:1)--++(90:7)
  ;

  \draw[very thick,fill=cyan!35](0,0)--++(0:3)
--++(-90:1)--++(180:1)
--++(-90:1)--++(180:1)
--++(-90:1)--++(180:1)
--++(90:3);
  \draw[very thick](0,0)--++(0:5)--++(-90:1)--++(180:1)--++(-90:1)--++(180:2);

 \draw[very thick](0,0)--++(-90:5)--++(0:1)--++(90:2);

  \clip(0,0)--++(0:9)--++(-90:1)--++(180:3)--++(-90:1)
 --++(180:1)--++(-90:1)
 --++(180:3)--++(-90:3)
  --++(180:1)--++(-90:2)
    --++(180:1)--++(90:7)
  ;

  \path(0,0) coordinate (origin);
   \foreach \i in {1,...,19}
  {
    \path (origin)++(0:1*\i cm)  coordinate (a\i);
    \path (origin)++(-90:1*\i cm)  coordinate (b\i);
    \path (a\i)++(-90:10cm) coordinate (ca\i);
    \path (b\i)++(0:10cm) coordinate (cb\i);
    \draw[densely dotted] (a\i) -- (ca\i)  (b\i) -- (cb\i); }

 \path(0.5,-0.5) coordinate (origin);
 \foreach \i in {1,...,19}
  {
    \path (origin)++(0:1*\i cm)  coordinate (a\i);
    \path (origin)++(-90:1*\i cm)  coordinate (b\i);
    \path (a\i)++(-90:1cm) coordinate (ca\i);
        \path (ca\i)++(-90:1cm) coordinate (cca\i);
    \path (b\i)++(0:1cm) coordinate (cb\i);
    \path (cb\i)++(0:1cm) coordinate (ccb\i);
  }
 \draw[very thick,fill=cyan!35](0,0)--++(0:3)
--++(-90:1)--++(180:1)
--++(-90:1)--++(180:1)
--++(-90:1)--++(180:1)
--++(90:3);
   \end{tikzpicture}\end{minipage} \qquad
   %%%
   %%%
   \begin{minipage}{3cm}\begin{tikzpicture} [scale=0.33]

 \draw[very thick](0,0)--++(0:9)--++(-90:1)--++(180:3)--++(-90:1)
 --++(180:1)--++(-90:3)
 --++(180:1)--++(-90:1)
  --++(180:3)--++(-90:2)
    --++(180:1)--++(90:8)
  ;

  \draw[very thick,fill=cyan!35](0,0)--++(0:3)
--++(-90:1)--++(180:1)
--++(-90:1)--++(180:1)
--++(-90:1)--++(180:1)
--++(90:3);
%  \draw[very thick](0,0)--++(0:5)--++(-90:1)--++(180:1)--++(-90:1)--++(180:2);

 \draw[very thick](3,0)--++(-90:3)--++(180:1)--++(-90:1)--++(180:1);

 \draw[very thick](5,0) --++(-90:1)--++(180:1)
  --++(-90:3)--++(180:1)
   --++(-90:1)--++(180:2);

 \draw[very thick](0,0)--++(-90:5)--++(0:1)--++(90:2);

  \clip(0,0)--++(0:9)--++(-90:1)--++(180:3)--++(-90:1)
 --++(180:1)--++(-90:3)
 --++(180:1)--++(-90:1)
  --++(180:3)--++(-90:2)
    --++(180:1)--++(90:8)
  ;

  \path(0,0) coordinate (origin);
   \foreach \i in {1,...,19}
  {
    \path (origin)++(0:1*\i cm)  coordinate (a\i);
    \path (origin)++(-90:1*\i cm)  coordinate (b\i);
    \path (a\i)++(-90:10cm) coordinate (ca\i);
    \path (b\i)++(0:10cm) coordinate (cb\i);
    \draw[densely dotted] (a\i) -- (ca\i)  (b\i) -- (cb\i); }

 \path(0.5,-0.5) coordinate (origin);
 \foreach \i in {1,...,19}
  {
    \path (origin)++(0:1*\i cm)  coordinate (a\i);
    \path (origin)++(-90:1*\i cm)  coordinate (b\i);
    \path (a\i)++(-90:1cm) coordinate (ca\i);
        \path (ca\i)++(-90:1cm) coordinate (cca\i);
    \path (b\i)++(0:1cm) coordinate (cb\i);
    \path (cb\i)++(0:1cm) coordinate (ccb\i);
  }
 \draw[very thick,fill=cyan!35](0,0)--++(0:3)
--++(-90:1)--++(180:1)
--++(-90:1)--++(180:1)
--++(-90:1)--++(180:1)
--++(90:3);
   \end{tikzpicture}\end{minipage}\quad
   %%%%
   %%%%
   %%%
     \begin{minipage}{2.5cm}\begin{tikzpicture} [scale=0.33]
  \draw[very thick,fill=cyan!35](0,0)--++(0:4)
--++(-90:1)--++(180:1)
--++(-90:1)--++(180:1)--++(-90:1)--++(180:1)
--++(-90:1)--++(180:1)
--++(90:3);
 \draw[very thick](0,0)--++(0:10)--++(-90:1)--++(180:3)--++(-90:1)
 --++(180:3)--++(-90:3)
 --++(180:1)--++(-90:1)
  --++(180:1)--++(-90:2)
    --++(180:2)--++(90:8)
  ;
  \draw[very thick] (6,0)--++(90:-1)--++(180:2);
  \draw[very thick] (0,-8)--++(0:1)--++(90:3)
  --++(0:1)--++(90:1)
  --++(0:1)%--++(90:1)
  %--++(0:1)
  --++(90:2);
    %--++(0:2)--++(90:1);

  \clip(0,0)--++(0:10)--++(-90:1)--++(180:3)--++(-90:1)
 --++(180:3)--++(-90:3)
 --++(180:1)--++(-90:1)
  --++(180:1)--++(-90:2)
    --++(180:2)--++(90:8)
  ;

  \path(0,0) coordinate (origin);
   \foreach \i in {1,...,19}
  {
    \path (origin)++(0:1*\i cm)  coordinate (a\i);
    \path (origin)++(-90:1*\i cm)  coordinate (b\i);
    \path (a\i)++(-90:10cm) coordinate (ca\i);
    \path (b\i)++(0:10cm) coordinate (cb\i);
    \draw[densely dotted] (a\i) -- (ca\i)  (b\i) -- (cb\i); }

 \path(0.5,-0.5) coordinate (origin);
 \foreach \i in {1,...,19}
  {
    \path (origin)++(0:1*\i cm)  coordinate (a\i);
    \path (origin)++(-90:1*\i cm)  coordinate (b\i);
    \path (a\i)++(-90:1cm) coordinate (ca\i);
        \path (ca\i)++(-90:1cm) coordinate (cca\i);
    \path (b\i)++(0:1cm) coordinate (cb\i);
    \path (cb\i)++(0:1cm) coordinate (ccb\i);
  }
    \draw[very thick,fill=cyan!35](0,0)--++(0:4)
--++(-90:1)--++(180:1)
--++(-90:1)--++(180:1)--++(-90:1)--++(180:1)
--++(-90:1)--++(180:1)
--++(90:4);
     \end{tikzpicture}\end{minipage}$$
\caption{Examples of partitions which label height 0 characters for $\mathfrak{S}_{28}$ and  $\mathfrak{S}_{36}$ (and therefore label constituents of  Saxl's tensor square by Theorem B).    There are
672 %1344
and
1417 such characters for these groups, respectively.  A combinatorial construction of all such partitions (for arbitrary $n\in \NN$) is given in \cref{subsec:height0}.
}
\label{2adicpic2836}
\end{figure}
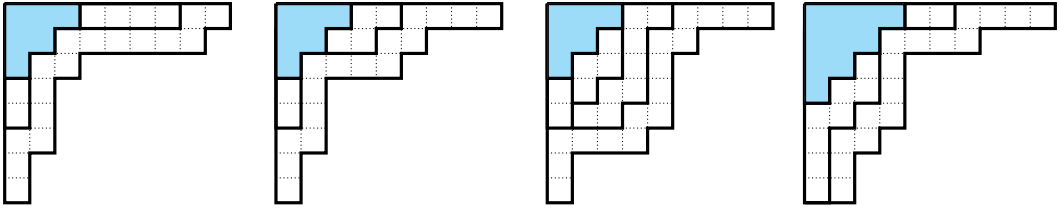

We   now shift  focus  to  the Kronecker coefficients labelled by  2-separated partitions.
 In what follows, we shall write  $g(\rho,\rho,\tau^\la_\mu)$ for the Kronecker coefficient labelled by
 a staircase
 $\rho$ of size  $n= k(k+1)/2$ for
 some $k\in \NN$ and some   2-separated partition  $\tau^\la_\mu $ of $n$;
  in other words, we do not encumber the notation by explicitly recording the size of the
    staircases involved.

\begin{thmC*}
 For $(\alpha,\beta)$ a $k$-Carter--Saxl pair (as in   \cref{ineedalabel}) we have that $g(\rho,\rho,\alpha)\geq k$.  In particular, all framed staircase partitions $\tau^{(a)}_{(1^b)}$ appear in the Saxl square.
 \end{thmC*}

We do not recall the definition of a $k$-Carter--Saxl pair here, but rather discuss some examples and consequences of Theorem B.
In particular, Theorem B implies that  every 2-block contains a wealth of constituents of the Saxl square
${\bf S}^\CC(\rho)\otimes {\bf S}^\CC(\rho)$ which can be deduced using our techniques.
 Carter--Saxl pairs  cut  across
 hook partitions,   partitions of arbitrarily large   Durfee size,
 symmetric and non-symmetric partitions, partitions from arbitrary blocks, and   across the full range of the dominance order.  (In fact, the only common trait of these partitions is that they label semisimple Specht modules for $H_{-1}^\mathbb{C}(n)$.)
We shall illustrate  below that the property of being a Carter--Saxl pair %it
is actually very easy to  work with diagrammatically.
 For example, the above theorem includes the infinite family of  ``framed staircases" as some of the simplest examples: these are partitions which
interpolate hooks and staircases.  More explicitly, these are  the partitions of the  form $\alpha=\tau^{(a)}_{(1^b)}$.  These can be pictured as in \cref{staircasepart} below.

\begin{figure}[ht!]
$$
\begin{minipage}{3.4cm}\begin{tikzpicture}[scale=0.3]
 \draw (1,0)--++(0:13)--++(-90:0.5)--++(180:13 );
   \draw(1,0)--(1,-0.5);
    \draw(2,0)--(2,-0.5);
      \draw(3,0)--(3,-0.5);
            \draw(4,0)--(4,-0.5);
                        \draw(5,0)--++(90:-0.5);
                        \draw(6,0)--++(90:-0.5);\draw(7,0)--++(90:-0.5);
  \draw(8,0)--++(90:-0.5);
    \draw(9,0)--++(90:-0.5);
      \draw(10,0)--++(90:-0.5);
        \draw(11,0)--++(90:-0.5);
          \draw(12,0)--++(90:-0.5);
          \draw(13,0)--++(90:-0.5);
   \draw (0.5,-1)--++(-90:8)--++(180:0.5)--++(90:8 );
          \draw(0,-2)--++(0:0.5);
          \draw(0,-3)--++(0:0.5);
                    \draw(0,-4)--++(0:0.5);
                              \draw(0,-5)--++(0:0.5);
                                        \draw(0,-6)--++(0:0.5);
              \draw(0,-7)--++(0:0.5);          \draw(0,-8)--++(0:0.5);

  \draw (0,0)--(1,0)--(1,-0.5) --(0.5,-0.5)  --(0.5,-1)  --(0,-1)  --(0,0);

  \draw[fill=cyan!30] (0,0)--(1,0)--(1,-0.5)--(0.5,-0.5)--(0.5,-1)--(0,-1)--(0,0);

 \end{tikzpicture}
\end{minipage}
 \ \ \quad\quad\quad
\begin{minipage}{2.8cm}\begin{tikzpicture}[scale=0.3]
 \draw (2.5,0)--++(0:9)--++(-90:0.5)--++(180:11 );
   \draw(0.5+1,0)--(1.5,-0.5);
    \draw(0.5+2,0)--(2.5,-0.5);
      \draw(0.5+3,0)--(3.5,-0.5);
            \draw(0.5+4,0)--(4.5,-0.5);
                        \draw(0.5+5,0)--++(90:-0.5);
                        \draw(0.5+6,0)--++(90:-0.5);\draw(7.5,0)--++(90:-0.5);
  \draw(0.5+8,0)--++(90:-0.5);
    \draw(0.5+9,0)--++(90:-0.5);
      \draw(0.5+10,0)--++(90:-0.5);
        \draw(0.5+11,0)--++(90:-0.5);

   \draw (0.5,-2.5 )--++(-90:6 )--++(180:0.5)--++(90:6  );
          \draw(0,-2-0.5)--++(0:0.5);
          \draw(0,-3-0.5)--++(0:0.5);
                    \draw(0,-4-0.5)--++(0:0.5);
                              \draw(0,-5-0.5)--++(0:0.5);
                                        \draw(0,-6-0.5)--++(0:0.5);
              \draw(0,-7-0.5)--++(0:0.5);

   \draw[fill=cyan!35] (0,0)--(2.5,0)--++(-90:0.5)--++(180:0.5)--++(-90:0.5)--++(180:0.5)--++(-90:0.5)--++(180:0.5)--++(-90:0.5)--++(180:0.5)--++(-90:0.5)--++(180:0.5)--(0,0);

 \end{tikzpicture}

\end{minipage}
\quad
 \ \ \quad\quad\quad
\begin{minipage}{2.8cm}\begin{tikzpicture}[scale=0.3]
   \draw (3,0)--(10,-0)--(10,-0.5)--(0,-0.5);
  \draw (0.5,-3)--(0.5,-8)--(0,-8)--(0,0);

   \draw(1,0)--(1,-0.5);
    \draw(2,0)--(2,-0.5);
      \draw(3,0)--(3,-0.5);
            \draw(4,0)--(4,-0.5);
                        \draw(5,0)--++(90:-0.5);
                        \draw(6,0)--++(90:-0.5);\draw(7,0)--++(90:-0.5);
  \draw(8,0)--++(90:-0.5);
    \draw(9,0)--++(90:-0.5);
      \draw(10,0)--++(90:-0.5);

            \draw(0,-2)--++(0:0.5);

          \draw(0,-3)--++(0:0.5);
                    \draw(0,-4)--++(0:0.5);
                              \draw(0,-5)--++(0:0.5);
                                        \draw(0,-6)--++(0:0.5);
              \draw(0,-7)--++(0:0.5);          \draw(0,-8)--++(0:0.5);

   \draw[fill=cyan!35] (0,0)--(3,0)--++(-90:0.5)--++(180:0.5)--++(-90:0.5)--++(180:0.5)--++(-90:0.5) --++(180:0.5)--++(-90:0.5)--++(180:0.5)--++(-90:0.5)--++(180:0.5)--++(-90:0.5)--++(180:0.5)--(0,0);

 \end{tikzpicture}

\end{minipage}
\quad\quad\quad   \begin{minipage}{1.4cm}\begin{tikzpicture}[scale=0.3]
  \draw[fill=cyan!35] (0,0)--(3.5,0)--++(-90:0.5)--++(180:0.5)--++(-90:0.5)--++(180:0.5)--++(-90:0.5)--++(180:0.5)--++(-90:0.5)--++(180:0.5)--++(-90:0.5)--++(180:0.5)--++(-90:0.5)--++(180:0.5)--++(-90:0.5)--++(180:0.5)--(0,0);

  \end{tikzpicture}
\end{minipage}
$$%\captionsetup{width=1\linewidth}
\caption{Some examples of  framed staircases:  $\alpha=\tau^{(13)}_{(1^{8})}$,
 $\tau^{(9)}_{(1^{6})}$,  $\tau^{(7)}_{(1^{5})}$, and $\tau=\rho(9)$  are all partitions of  $n=45$.      Up to conjugation, there are 35   framed staircase partitions of   $n=45$.
The classification of  decomposable  Specht modules labelled by framed staircases  is the main result of  \cite{DG}. In \cref{stairacse} we prove Saxl's conjecture for all framed staircase partitions.  The key ingredient in our proof is  that
 $(\tau^{(a)}_{(1^b)}, \tau^{(a+b)}_\varnothing)$ is a 1-Carter--Saxl pair	for	 $a,b\in \mathbb N$. }
\label{staircasepart}
\end{figure}
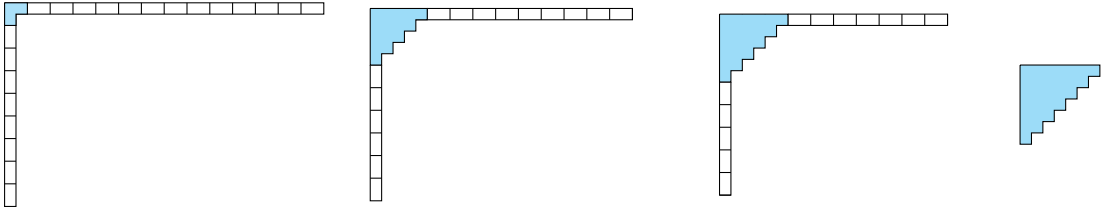

    We   wish   to provide bounds on the
 Kronecker coefficients: the maximal possible values obtained  by Kronecker products are studied in \cite{PP16}, and the Kronecker products whose coefficients
 are all as small as possible (namely all 0 or 1) are classified in \cite{cbcb}.
 For constituents to partitions of depth at most~4,
 explicit formulae for their multiplicity in squares
 were provided by Saxl in 1987,
   and later work by Zisser and Vallejo, respectively.
  For the Kronecker coefficients studied here, the easiest (and well known) non-trivial case
 is $g(\rho(k),\rho(k),(n-1,1))=k-1$, so the Kronecker coefficients are
 even unbounded; this also holds for the other families corresponding to
 partitions of small depth.
 Lower bounds coming from character values on a specific class were obtained
 by Pak and Panova in \cite{PP17},
  where also the asymptotic behaviour of the multiplicity of special constituents
 is studied.
 Theorem B allows us to
 provide explicit lower bounds on the Kronecker coefficients $g(\rho(k),\rho(k),\lambda)$
 for new infinite families of  Saxl constituents, where again the multiplicities are unbounded.

We now provide some examples of more complicated Carter--Saxl pairs.
For $n=78$, if we first focus on
 the (unique) block of weight $w=6$ we find 6 constituents in this block labelled by framed staircases  as well as the   Carter--Saxl pairs  given (up to conjugation) in \cref{afiguregofigure} below.
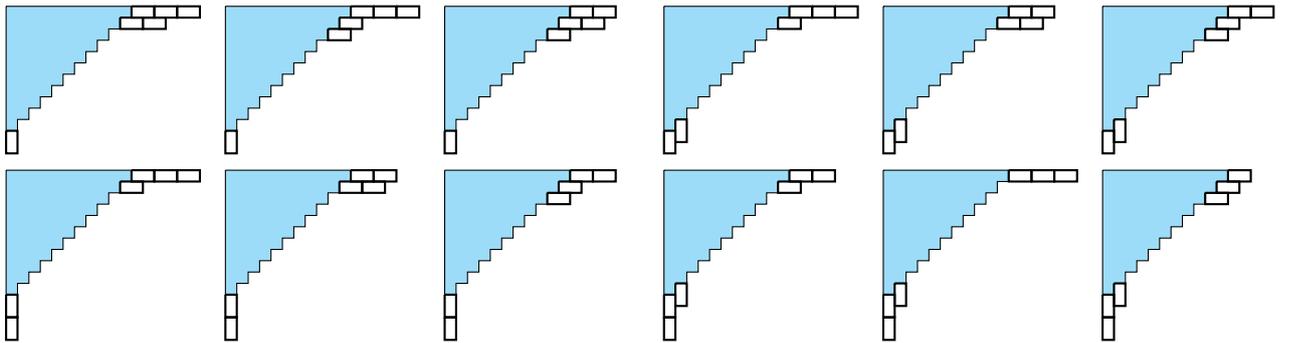
\begin{figure}[ht!] $$
 \begin{minipage}{2.8cm}\begin{tikzpicture}[scale=0.3]
 \draw[fill=cyan!35](0,-0.5)--(5.5,-0.5)--++(-90:0.5)--++(-180:0.5)--++(-90:0.5)--++(-180:0.5)--++(-90:0.5)--++(-180:0.5)--++(-90:0.5) --++(-180:0.5)--++(-90:0.5)--++(-180:0.5)--++(-90:0.5)--++(-180:0.5)--++(-90:0.5)--++(-180:0.5)--++(-90:0.5)--++(-180:0.5)--++(-90:0.5)--++(-180:0.5)--++(-90:0.5)--++(-180:0.5)--++(-90:0.5)--++(-180:0.5)--(0,-0.5);

\draw[thick](5.5,-0.5)--(6.5,-0.5)--(6.5,-1)--(5.5,-1)--(5.5,-0.5);
\draw[thick](6.5,-0.5)--++(0:1)--++(-90:0.5)--++(180:1)--++(90:0.5);
\draw[thick](7.5,-0.5)--++(0:1)--++(-90:0.5)--++(180:1)--++(90:0.5);
\draw[thick](6,-1)--++(0:1)--++(-90:0.5)--++(180:1)--++(90:0.5);
\draw[thick](5,-1)--++(0:1)--++(-90:0.5)--++(180:1)--++(90:0.5);

\draw[thick](0,-6)--++(0:0.5)--++(-90:1)--++(180:0.5)--++(90:1);
   \end{tikzpicture}
\end{minipage}
 \
 \begin{minipage}{2.8cm}\begin{tikzpicture}[scale=0.3]
%\draw(0,0)--(5.5,0)--(5.5,-0.5)--(0.5,-0.5)--(0.5,-6)--(0,-6)--(0,0);
\draw[fill=cyan!35](0,-0.5)--(5.5,-0.5)--++(-90:0.5)--++(-180:0.5)--++(-90:0.5)--++(-180:0.5)--++(-90:0.5)--++(-180:0.5)--++(-90:0.5) --++(-180:0.5)--++(-90:0.5)--++(-180:0.5)--++(-90:0.5)--++(-180:0.5)--++(-90:0.5)--++(-180:0.5)--++(-90:0.5)--++(-180:0.5)--++(-90:0.5)--++(-180:0.5)--++(-90:0.5)--++(-180:0.5)--++(-90:0.5)--++(-180:0.5)--(0,-0.5);

\draw[thick](5.5,-0.5)--(6.5,-0.5)--(6.5,-1)--(5.5,-1)--(5.5,-0.5);
\draw[thick](6.5,-0.5)--++(0:1)--++(-90:0.5)--++(180:1)--++(90:0.5);
\draw[thick](7.5,-0.5)--++(0:1)--++(-90:0.5)--++(180:1)--++(90:0.5);
\draw[thick](5,-1)--++(0:1)--++(-90:0.5)--++(180:1)--++(90:0.5);
\draw[thick](4.5,-1.5)--++(0:1)--++(-90:0.5)--++(180:1)--++(90:0.5);

\draw[thick](0,-6)--++(0:0.5)--++(-90:1)--++(180:0.5)--++(90:1);
   \end{tikzpicture}
\end{minipage}
 \
 \begin{minipage}{2.8cm}\begin{tikzpicture}[scale=0.3]
%\draw(0,0)--(5.5,0)--(5.5,-0.5)--(0.5,-0.5)--(0.5,-6)--(0,-6)--(0,0);
\draw[fill=cyan!35](0,-0.5)--(5.5,-0.5)--++(-90:0.5)--++(-180:0.5)--++(-90:0.5)--++(-180:0.5)--++(-90:0.5)--++(-180:0.5)--++(-90:0.5) --++(-180:0.5)--++(-90:0.5)--++(-180:0.5)--++(-90:0.5)--++(-180:0.5)--++(-90:0.5)--++(-180:0.5)--++(-90:0.5)--++(-180:0.5)--++(-90:0.5)--++(-180:0.5)--++(-90:0.5)--++(-180:0.5)--++(-90:0.5)--++(-180:0.5)--(0,-0.5);

\draw[thick](5.5,-0.5)--(6.5,-0.5)--(6.5,-1)--(5.5,-1)--(5.5,-0.5);
\draw[thick](6.5,-0.5)--++(0:1)--++(-90:0.5)--++(180:1)--++(90:0.5);
\draw[thick](6,-1)--++(0:1)--++(-90:0.5)--++(180:1)--++(90:0.5);
\draw[thick](5,-1)--++(0:1)--++(-90:0.5)--++(180:1)--++(90:0.5);
\draw[thick](4.5,-1.5)--++(0:1)--++(-90:0.5)--++(180:1)--++(90:0.5);

\draw[thick](0,-6)--++(0:0.5)--++(-90:1)--++(180:0.5)--++(90:1);
   \end{tikzpicture}
\end{minipage}
 \
\begin{minipage}{2.8cm}\begin{tikzpicture}[scale=0.3]
%\draw(0,0)--(5.5,0)--(5.5,-0.5)--(0.5,-0.5)--(0.5,-6)--(0,-6)--(0,0);
\draw[fill=cyan!35](0,-0.5)--(5.5,-0.5)--++(-90:0.5)--++(-180:0.5)--++(-90:0.5)--++(-180:0.5)--++(-90:0.5)--++(-180:0.5)--++(-90:0.5) --++(-180:0.5)--++(-90:0.5)--++(-180:0.5)--++(-90:0.5)--++(-180:0.5)--++(-90:0.5)--++(-180:0.5)--++(-90:0.5)--++(-180:0.5)--++(-90:0.5)--++(-180:0.5)--++(-90:0.5)--++(-180:0.5)--++(-90:0.5)--++(-180:0.5)--(0,-0.5);

\draw[thick](5.5,-0.5)--(6.5,-0.5)--(6.5,-1)--(5.5,-1)--(5.5,-0.5);
\draw[thick](6.5,-0.5)--++(0:1)--++(-90:0.5)--++(180:1)--++(90:0.5);
\draw[thick](7.5,-0.5)--++(0:1)--++(-90:0.5)--++(180:1)--++(90:0.5);
 \draw[thick](5,-1)--++(0:1)--++(-90:0.5)--++(180:1)--++(90:0.5);

\draw[thick](0,-6)--++(0:0.5)--++(-90:1)--++(180:0.5)--++(90:1);
\draw[thick](0.5,-5.5)--++(0:0.5)--++(-90:1)--++(180:0.5)--++(90:1);
   \end{tikzpicture}\end{minipage}
    \
\begin{minipage}{2.8cm}\begin{tikzpicture}[scale=0.3]
%\draw(0,0)--(5.5,0)--(5.5,-0.5)--(0.5,-0.5)--(0.5,-6)--(0,-6)--(0,0);
\draw[fill=cyan!35](0,-0.5)--(5.5,-0.5)--++(-90:0.5)--++(-180:0.5)--++(-90:0.5)--++(-180:0.5)--++(-90:0.5)--++(-180:0.5)--++(-90:0.5) --++(-180:0.5)--++(-90:0.5)--++(-180:0.5)--++(-90:0.5)--++(-180:0.5)--++(-90:0.5)--++(-180:0.5)--++(-90:0.5)--++(-180:0.5)--++(-90:0.5)--++(-180:0.5)--++(-90:0.5)--++(-180:0.5)--++(-90:0.5)--++(-180:0.5)--(0,-0.5);

\draw[thick](5.5,-0.5)--(6.5,-0.5)--(6.5,-1)--(5.5,-1)--(5.5,-0.5);
\draw[thick](6.5,-0.5)--++(0:1)--++(-90:0.5)--++(180:1)--++(90:0.5);
  \draw[thick](5,-1)--++(0:1)--++(-90:0.5)--++(180:1)--++(90:0.5);
 \draw[thick](6,-1)--++(0:1)--++(-90:0.5)--++(180:1)--++(90:0.5);

\draw[thick](0,-6)--++(0:0.5)--++(-90:1)--++(180:0.5)--++(90:1);
\draw[thick](0.5,-5.5)--++(0:0.5)--++(-90:1)--++(180:0.5)--++(90:1);
   \end{tikzpicture}\end{minipage}
 \ % $$
%$$
\begin{minipage}{2.8cm}\begin{tikzpicture}[scale=0.3]
%\draw(0,0)--(5.5,0)--(5.5,-0.5)--(0.5,-0.5)--(0.5,-6)--(0,-6)--(0,0);
\draw[fill=cyan!35](0,-0.5)--(5.5,-0.5)--++(-90:0.5)--++(-180:0.5)--++(-90:0.5)--++(-180:0.5)--++(-90:0.5)--++(-180:0.5)--++(-90:0.5) --++(-180:0.5)--++(-90:0.5)--++(-180:0.5)--++(-90:0.5)--++(-180:0.5)--++(-90:0.5)--++(-180:0.5)--++(-90:0.5)--++(-180:0.5)--++(-90:0.5)--++(-180:0.5)--++(-90:0.5)--++(-180:0.5)--++(-90:0.5)--++(-180:0.5)--(0,-0.5);

\draw[thick](5.5,-0.5)--(6.5,-0.5)--(6.5,-1)--(5.5,-1)--(5.5,-0.5);
\draw[thick](6.5,-0.5)--++(0:1)--++(-90:0.5)--++(180:1)--++(90:0.5);
  \draw[thick](5,-1)--++(0:1)--++(-90:0.5)--++(180:1)--++(90:0.5);
 \draw[thick](4.5,-1.5)--++(0:1)--++(-90:0.5)--++(180:1)--++(90:0.5);

\draw[thick](0,-6)--++(0:0.5)--++(-90:1)--++(180:0.5)--++(90:1);
\draw[thick](0.5,-5.5)--++(0:0.5)--++(-90:1)--++(180:0.5)--++(90:1);
   \end{tikzpicture}\end{minipage}
 $$
  $$
  \begin{minipage}{2.8cm}\begin{tikzpicture}[scale=0.3]
%\draw(0,0)--(5.5,0)--(5.5,-0.5)--(0.5,-0.5)--(0.5,-6)--(0,-6)--(0,0);
\draw[fill=cyan!35](0,-0.5)--(5.5,-0.5)--++(-90:0.5)--++(-180:0.5)--++(-90:0.5)--++(-180:0.5)--++(-90:0.5)--++(-180:0.5)--++(-90:0.5) --++(-180:0.5)--++(-90:0.5)--++(-180:0.5)--++(-90:0.5)--++(-180:0.5)--++(-90:0.5)--++(-180:0.5)--++(-90:0.5)--++(-180:0.5)--++(-90:0.5)--++(-180:0.5)--++(-90:0.5)--++(-180:0.5)--++(-90:0.5)--++(-180:0.5)--(0,-0.5);

\draw[thick](5.5,-0.5)--(6.5,-0.5)--(6.5,-1)--(5.5,-1)--(5.5,-0.5);
\draw[thick](6.5,-0.5)--++(0:1)--++(-90:0.5)--++(180:1)--++(90:0.5);
\draw[thick](7.5,-0.5)--++(0:1)--++(-90:0.5)--++(180:1)--++(90:0.5);
 \draw[thick](5,-1)--++(0:1)--++(-90:0.5)--++(180:1)--++(90:0.5);

\draw[thick](0,-6)--++(0:0.5)--++(-90:1)--++(180:0.5)--++(90:1);
\draw[thick](0,-7)--++(0:0.5)--++(-90:1)--++(180:0.5)--++(90:1);
   \end{tikzpicture}\end{minipage}
    \
\begin{minipage}{2.8cm}\begin{tikzpicture}[scale=0.3]
%\draw(0,0)--(5.5,0)--(5.5,-0.5)--(0.5,-0.5)--(0.5,-6)--(0,-6)--(0,0);
\draw[fill=cyan!35](0,-0.5)--(5.5,-0.5)--++(-90:0.5)--++(-180:0.5)--++(-90:0.5)--++(-180:0.5)--++(-90:0.5)--++(-180:0.5)--++(-90:0.5) --++(-180:0.5)--++(-90:0.5)--++(-180:0.5)--++(-90:0.5)--++(-180:0.5)--++(-90:0.5)--++(-180:0.5)--++(-90:0.5)--++(-180:0.5)--++(-90:0.5)--++(-180:0.5)--++(-90:0.5)--++(-180:0.5)--++(-90:0.5)--++(-180:0.5)--(0,-0.5);

\draw[thick](5.5,-0.5)--(6.5,-0.5)--(6.5,-1)--(5.5,-1)--(5.5,-0.5);
\draw[thick](6.5,-0.5)--++(0:1)--++(-90:0.5)--++(180:1)--++(90:0.5);
  \draw[thick](5,-1)--++(0:1)--++(-90:0.5)--++(180:1)--++(90:0.5);
 \draw[thick](6,-1)--++(0:1)--++(-90:0.5)--++(180:1)--++(90:0.5);

\draw[thick](0,-6)--++(0:0.5)--++(-90:1)--++(180:0.5)--++(90:1);
\draw[thick](0,-7)--++(0:0.5)--++(-90:1)--++(180:0.5)--++(90:1);
   \end{tikzpicture}\end{minipage}
 \ % $$
%$$
\begin{minipage}{2.8cm}\begin{tikzpicture}[scale=0.3]
%\draw(0,0)--(5.5,0)--(5.5,-0.5)--(0.5,-0.5)--(0.5,-6)--(0,-6)--(0,0);
\draw[fill=cyan!35](0,-0.5)--(5.5,-0.5)--++(-90:0.5)--++(-180:0.5)--++(-90:0.5)--++(-180:0.5)--++(-90:0.5)--++(-180:0.5)--++(-90:0.5) --++(-180:0.5)--++(-90:0.5)--++(-180:0.5)--++(-90:0.5)--++(-180:0.5)--++(-90:0.5)--++(-180:0.5)--++(-90:0.5)--++(-180:0.5)--++(-90:0.5)--++(-180:0.5)--++(-90:0.5)--++(-180:0.5)--++(-90:0.5)--++(-180:0.5)--(0,-0.5);

\draw[thick](5.5,-0.5)--(6.5,-0.5)--(6.5,-1)--(5.5,-1)--(5.5,-0.5);
\draw[thick](6.5,-0.5)--++(0:1)--++(-90:0.5)--++(180:1)--++(90:0.5);
  \draw[thick](5,-1)--++(0:1)--++(-90:0.5)--++(180:1)--++(90:0.5);
 \draw[thick](4.5,-1.5)--++(0:1)--++(-90:0.5)--++(180:1)--++(90:0.5);

\draw[thick](0,-6)--++(0:0.5)--++(-90:1)--++(180:0.5)--++(90:1);
\draw[thick](0,-7)--++(0:0.5)--++(-90:1)--++(180:0.5)--++(90:1);
   \end{tikzpicture}\end{minipage}
    \ % $$
%$$
\begin{minipage}{2.8cm}\begin{tikzpicture}[scale=0.3]
%\draw(0,0)--(5.5,0)--(5.5,-0.5)--(0.5,-0.5)--(0.5,-6)--(0,-6)--(0,0);
\draw[fill=cyan!35](0,-0.5)--(5.5,-0.5)--++(-90:0.5)--++(-180:0.5)--++(-90:0.5)--++(-180:0.5)--++(-90:0.5)--++(-180:0.5)--++(-90:0.5) --++(-180:0.5)--++(-90:0.5)--++(-180:0.5)--++(-90:0.5)--++(-180:0.5)--++(-90:0.5)--++(-180:0.5)--++(-90:0.5)--++(-180:0.5)--++(-90:0.5)--++(-180:0.5)--++(-90:0.5)--++(-180:0.5)--++(-90:0.5)--++(-180:0.5)--(0,-0.5);

\draw[thick](5.5,-0.5)--(6.5,-0.5)--(6.5,-1)--(5.5,-1)--(5.5,-0.5);
\draw[thick](6.5,-0.5)--++(0:1)--++(-90:0.5)--++(180:1)--++(90:0.5);
  \draw[thick](5,-1)--++(0:1)--++(-90:0.5)--++(180:1)--++(90:0.5);

\draw[thick](0.5,-5.5)--++(0:0.5)--++(-90:1)--++(180:0.5)--++(90:1);

\draw[thick](0,-6)--++(0:0.5)--++(-90:1)--++(180:0.5)--++(90:1);
\draw[thick](0,-7)--++(0:0.5)--++(-90:1)--++(180:0.5)--++(90:1);
   \end{tikzpicture}\end{minipage}  \ % $$
%$$
\begin{minipage}{2.8cm}\begin{tikzpicture}[scale=0.3]
%\draw(0,0)--(5.5,0)--(5.5,-0.5)--(0.5,-0.5)--(0.5,-6)--(0,-6)--(0,0);
\draw[fill=cyan!35](0,-0.5)--(5.5,-0.5)--++(-90:0.5)--++(-180:0.5)--++(-90:0.5)--++(-180:0.5)--++(-90:0.5)--++(-180:0.5)--++(-90:0.5) --++(-180:0.5)--++(-90:0.5)--++(-180:0.5)--++(-90:0.5)--++(-180:0.5)--++(-90:0.5)--++(-180:0.5)--++(-90:0.5)--++(-180:0.5)--++(-90:0.5)--++(-180:0.5)--++(-90:0.5)--++(-180:0.5)--++(-90:0.5)--++(-180:0.5)--(0,-0.5);

\draw[thick](5.5,-0.5)--(6.5,-0.5)--(6.5,-1)--(5.5,-1)--(5.5,-0.5);
\draw[thick](6.5,-0.5)--++(0:1)--++(-90:0.5)--++(180:1)--++(90:0.5);
\draw[thick](7.5,-0.5)--++(0:1)--++(-90:0.5)--++(180:1)--++(90:0.5);
%  \draw[thick](5,-1)--++(0:1)--++(-90:0.5)--++(180:1)--++(90:0.5);

\draw[thick](0.5,-5.5)--++(0:0.5)--++(-90:1)--++(180:0.5)--++(90:1);

\draw[thick](0,-6)--++(0:0.5)--++(-90:1)--++(180:0.5)--++(90:1);
\draw[thick](0,-7)--++(0:0.5)--++(-90:1)--++(180:0.5)--++(90:1);
   \end{tikzpicture}\end{minipage} \ % $$
%$$
\begin{minipage}{2.8cm}\begin{tikzpicture}[scale=0.3]
%\draw(0,0)--(5.5,0)--(5.5,-0.5)--(0.5,-0.5)--(0.5,-6)--(0,-6)--(0,0);
\draw[fill=cyan!35](0,-0.5)--(5.5,-0.5)--++(-90:0.5)--++(-180:0.5)--++(-90:0.5)--++(-180:0.5)--++(-90:0.5)--++(-180:0.5)--++(-90:0.5) --++(-180:0.5)--++(-90:0.5)--++(-180:0.5)--++(-90:0.5)--++(-180:0.5)--++(-90:0.5)--++(-180:0.5)--++(-90:0.5)--++(-180:0.5)--++(-90:0.5)--++(-180:0.5)--++(-90:0.5)--++(-180:0.5)--++(-90:0.5)--++(-180:0.5)--(0,-0.5);

\draw[thick](5.5,-0.5)--(6.5,-0.5)--(6.5,-1)--(5.5,-1)--(5.5,-0.5);
%\draw[thick](6.5,-0.5)--++(0:1)--++(-90:0.5)--++(180:1)--++(90:0.5);
%\draw[thick](7.5,-0.5)--++(0:1)--++(-90:0.5)--++(180:1)--++(90:0.5);
  \draw[thick](5,-1)--++(0:1)--++(-90:0.5)--++(180:1)--++(90:0.5);
  \draw[thick](4.5,-1.5)--++(0:1)--++(-90:0.5)--++(180:1)--++(90:0.5);
\draw[thick](0.5,-5.5)--++(0:0.5)--++(-90:1)--++(180:0.5)--++(90:1);

\draw[thick](0,-6)--++(0:0.5)--++(-90:1)--++(180:0.5)--++(90:1);
\draw[thick](0,-7)--++(0:0.5)--++(-90:1)--++(180:0.5)--++(90:1);
   \end{tikzpicture}\end{minipage}$$

\noindent
  \caption{More examples of coefficients  $g(\rho,\rho,\tau^\la_\mu)>0$.  These belong to the block of weight 6 for the symmetric group of rank $78$.
We have that each of  $\tau^\la_\mu$ belongs to a Carter--Saxl pair of the form  $(\tau^\la_\mu,\tau^{(3,2,1)}_\varnothing)$.
 }
\label{afiguregofigure}
\end{figure}

Finally, we propose two extensions of Saxl's conjecture based on its modular representation theoretic interpretation.
The first conjecture reduces the problem to the case of $2$-regular partitions, but at the expense of working in the more difficult modular setting.  We remark that towards Saxl's conjecture over $\CC$,
  it has already been verified that for any 2-regular partition $\la$ of $n=k(k+1)/2$ the Kronecker coefficient
  $g(\rho(k),\rho(k),\la)$ is positive \cite{I15}, and so it is natural to hope that this can be extended to positive characteristic.

\begin{conjB}
Let  $\Bbbk$ be a field of characteristic 2.  We have that
$$
\dim_\Bbbk \left(\Hom_{\mathfrak{S}_n}\left( {\bf D}^\Bbbk(\rho(k)) \otimes  {\bf D}^\Bbbk(\rho(k)),
 {\bf D}^\Bbbk(\la)\right)\right) > 0.
$$
for any $2$-regular partition $\la$ of $n=k(k+1)/2$.
Equivalently: Saxl's 2-modular tensor square contains all indecomposable projective modules as direct summands with positive multiplicity.
\end{conjB}

What could be a suitable candidate for arbitrary $n$, not just triangular numbers?

\begin{conjA}
For  $n\in \mathbb{N}$ there exists a symmetric $p$-core $\lambda$ for some $p\le n$ such that
$
 {\bf D}^\CC(\la) \otimes  {\bf D}^\CC(\lambda)
$
contains all simple $\CC\mathfrak{S}_n$-modules with positive multiplicity.
\end{conjA}

While this sounds reasonable, in fact, for larger $n$  it hardly restricts the search for
a good candidate as almost any partition of $n$ is then a $p$-core for some $p\le n$.
So as a guide towards finding a simple module ${\bf D}^\CC(\la)$  whose tensor square
contains all simples, one would try to find a suitable symmetric $p$-core for a small prime~$p$.

\begin{ack}
We would like to thank Matthew Fayers for  helpful discussions on simple Specht modules and for sharing his extensive computer calculations which were fundamental in writing this paper.
We also wish to thank  Liron Speyer for  playing ``match-maker" in this collaboration.
 The second author would like   to thank both  the
Alexander von Humboldt Foundation and the Leibniz Universit\"at
 Hannover  for
financial support and an enjoyable summer.
 The third author is supported by Singapore MOE Tier 2 AcRF MOE2015-T2-2-003.
\end{ack}

\section{The Hecke algebra }\label{sec1}
Let $\Bbbk$ be a commutative integral domain.
    We let $\mathfrak{S}_n$ denote the symmetric group on $n$ letters, with   presentation
$$\mathfrak{S}_n=\langle s_1,\dots ,s_{n-1}\mid s_i s_{i+1}s_i = s_{i+1} s_{i}s_{i+1}, s_is_j = s_j s_i \text{ for }|i-j|>1\rangle.$$
  We are interested in the representation theory (over $\Bbbk$) of
symmetric groups and their
deformations.
 Given $q\in \Bbbk$, we define  the  Hecke algebra $H_q^\Bbbk(n)$ to be the unital associative $\Bbbk$-algebra with generators $T_1, T_2,,\dots, T_{n-1}$ and relations
$$
 (T_i -q)(T_i + 1) = 0
\quad
   T_iT_j =T_jT_i,
\quad T_{i}T_{i+1}T_{i} =T_{i+1}T_{i}T_{i+1}$$
for $i=1,\dots,n-1$ and $i\neq j$.
We let $e\in \mathbb N$ be the smallest integer such that  $1+q+q^2+\dots + q^{e-1}=0$ or set $e=\infty$ if no such integer exists.
  If $\Bbbk$ is a field of characteristic $p$ and
     $p=e$ , then $H^\Bbbk_q(n)$ is isomorphic to $\Bbbk\mathfrak{S}_n$.

  We define a {\sf composition}, $\lambda$,  of $n$ to be a   finite  sequence  of non-negative integers $ (\lambda_1,\lambda_2, \ldots)$ whose sum, $|\lambda| = \lambda_1+\lambda_2 + \dots$, equals $n$.
  If the sequence $(\lambda_1,\lambda_2, \ldots)$ is  weakly decreasing,  we say that $\la$ is a {\sf partition}.   The number of non-zero parts of a partition, $\la$,  is called its {\sf length}, $\ell(\lambda)$; the size of the largest part is called the width, $w(\la)=\la_1$.
   Given  $\lambda  \in \ptn n  $,
   its {\sf Young diagram} $[\lambda ]$
    is defined to be the configuration of nodes,
\[
[\lambda] = \{(r,c) \mid 1\le r \le \ell(\lambda), 1\leq  c\leq \lambda_r\}.
\]
 The conjugate partition, $\la^T$, is the partition obtained by interchanging the rows and columns of $\la$.
      Given $\lambda\in \ptn n$, we define a {\sf tableau} of shape $\lambda$ to be a filling of the nodes  of
the Young diagram of  $\boxla $ with the numbers
$\{1,\dots , n\}$.
We define a {\sf  standard tableau} to be a tableau  in which    the entries increase along both the rows and columns of each component.
 We let $\Std(\lambda)$ denote the set of all standard tableaux of shape $\lambda\in\ptn n$.
    Given
$\stt\in \Std(\lambda)$, we set $\Shape(\stt)=\la$.
  Given $1\leq k \leq n$, we let $\stt{\downarrow}_{\{1,\dots ,k\}}$ be the subtableau of $\stt$ whose  entries belong to the set
  $\{1,\dots,k\}$.
We write $\stt \trianglerighteq  \sts$ if $\stt(k)  \trianglerighteq  \sts(k)$ for all $1\leq k \leq s$ and refer to this as the dominance order on $\Std(\lambda)$.

 We let $\stt^\lambda$ and $\stt_\lambda$ denote the most and least dominant tableaux respectively.
 We let $w_\lambda\in \mathfrak{S}_n$ be the permutation such that $w_\lambda \stt^\lambda=\stt_\lambda$.
For example, $w_{(3,2,1)}=(2,4)(3,6)$ and
$$
\stt^{(3,2,1)}=\begin{minipage}{2cm}\gyoung(1;2;3,4;5,6)\end{minipage}\qquad
\stt_{(3,2,1)}=\begin{minipage}{2cm}\gyoung(1;4;6,2;5,3)\end{minipage}.
$$

\begin{defn}  Given $\la $ a partition of $n$, we set $\mathfrak{S}_\la=\mathfrak{S}_{\lambda_1} \times \mathfrak{S}_{\lambda_2}\dots \leq \mathfrak{S}_n$ and we set
$$
   x_\la=\sum_{w \in \mathfrak{S}_\la} T_w
   \qquad
      y_\la=\sum_{w \in \mathfrak{S}_\la} (-q)^{\ell(w)}T_w
$$
and we define the {\sf Specht module}, ${\bf S}^\Bbbk_{q}(\la)$, to be the left  $H^\Bbbk_q(n)$-module
 $$
{\bf S}^\Bbbk_{q}(\la):= H^\Bbbk_q(n)y_\lambda  T_{w_\lambda} x_\lambda.
 $$
  \end{defn}

\begin{rmk}
Letting $\Bbbk=\mathbb{C}$ and specialising $q=1$ we have that
$H^\Bbbk_q(n)$ is isomorphic to $\CC\mathfrak{S}_n$.  In this case,   we drop the subscript on the Specht modules
 and we have that
 $$\{{\bf S}^\CC (\la)\mid \la\in \ptn n\}$$ provide a complete set of non-isomorphic simple $\CC\mathfrak{S}_n$-modules.
 We let $\chi^\la$ denote the character of the complex irreducible module ${\bf S}^\CC (\la)$.
\end{rmk}

 \subsection{Modular representation theory }\label{facts}
 Let $\Bbbk$ be a field and $q\in \Bbbk$.
The group algebra of the  symmetric group $\Bbbk\mathfrak{S}_n$ is a semisimple  algebra if and only if $\Bbbk$ is a field of characteristic $p>n$.
The Hecke algebra of the symmetric group is a non-semisimple algebra if and only if
$q$ is a primitive $e$th root of unity for some $e\leq n$ or
$\Bbbk$ is a field of characteristic $p\leq n$.
 We shall now recall the basics of the non-semisimple representation theory of these algebras.

Modular representation theory seeks to  deconstruct the non-semisimple representations of an algebra
 in terms of their simple constituents.  To this end, we define
   the {\it radical} of a finite-dimensional $A$-module $M$, denoted
$\rad (M)$, to be the smallest submodule of $M$ such that the corresponding
quotient is semisimple.  We then let $\rad^2 M = \rad (\rad M)$ and
inductively define the {\it radical series}, $ \rad^i M $, of $M$ by
$\rad^{i+1} M = \rad(\rad^i M)$.
We have a finite chain
\[
M  \supset \rad (M) \supset \rad^2 (M) \supset \cdots \supset \rad^i
(M) \supset \rad^{i+1} (M) \supset \cdots  \supset\rad^{s} (M)= 0.
\]

In the non-semisimple case, the Specht modules are no longer simple but they continue to play an important role in the representation theory of $H^\Bbbk_q(n)$ as we shall now see.
 We say that  a partition $\lambda=(\la_1,\la_2,\dots,\la_\ell)$ is $e$-regular
  if
  there is no $1\leq i \leq \ell$ such that $\lambda_i=\la_{i+1}=\dots=\la_{i+e-1}>0$.
  We let $\regptn e n$ denote the set of all $e$-regular partition of $n$.
Occasionally, we will also use the notation $\lambda\vdash_e n$
in place of $\lambda \in \regptn e n$.
 For $\Bbbk$ an arbitrary field, we have that
\begin{equation}\label{simpless}\{ {\bf D}^\Bbbk_q(\mu)\mid {\bf D}^\Bbbk_q(\mu)={\bf S}^\Bbbk_q(\mu) / \rad({\bf S}_q^\Bbbk(\mu)) ,    \mu \in \regptn e n \}\end{equation}
 provides a full set of non-isomorphic  simple $H^\Bbbk_q(n)$-modules.
 Of course, the radical of a Specht module is not easy to compute!
     The passage between the Specht and  simple  modules is recorded in the
 {\sf decomposition matrix},
 $$ (d^\Bbbk _{\lambda\mu} )_{{\lambda\in \ptn n}\atop {\mu \in \regptn e n}}\,
 \quad
 d^\Bbbk_{\lambda\mu} = [{\bf S}^\Bbbk_q(\lambda):{\bf D}^\Bbbk_q(\mu)]\, $$
 where $ [{\bf S}^\Bbbk_q(\lambda):{\bf D}^\Bbbk_q(\mu)]$ denotes the multiplicity of ${\bf D}^\Bbbk_q(\mu)$ as a composition factor of
 $ {\bf S}^\Bbbk_q(\lambda)$.     This matrix is uni-triangular with respect to the dominance ordering on $\ptn n$.
 We have already seen in \cref{simpless} that every column of the decomposition matrix contains an entry equal to 1; namely if $\mu\in\regptn e n$ then $d_{\mu,\mu}=1$.  We  now recall James' regularisation theorem, which states that every row of the decomposition matrix  contains an entry equal to 1 (and identifies this entry).

 \begin{figure}[ht!]
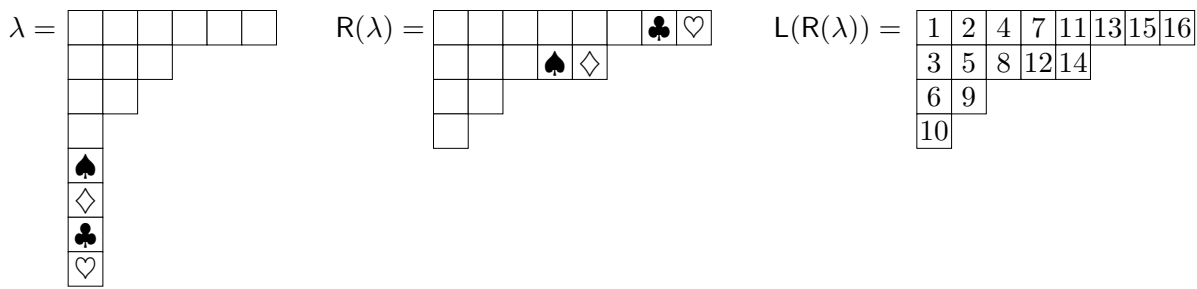

$$ \lambda=\gyoung(;;;;;;,;;;,;;,;,;\spadesuit,\diamondsuit,\clubsuit,\heartsuit)
\qquad
{\sf R}(\lambda)=
\gyoung(;;;;;;;\clubsuit;\heartsuit,;;;;\spadesuit;\diamondsuit,;;,;)
\qquad
\Lad({{\sf R}(\la)})=
\gyoung(1;2;4;7;\releven;\rthirt;\rfive;\rsix,3;5;8;\rtwelve;\rfour,6;9,\rten)$$
\caption{The partition $\la=\tau^{(1)}_{(1^2)}(4)$, its 2-regularisation ${\sf R}(\la)$, and the ladder tableau of shape ${\sf R}(\la)$.  }
\label{regularising}
\end{figure}

\begin{eg}
We picture a partition $\lambda$ and its 2-regularisation ${\sf R}(\lambda)$ in \cref{regularising}.
We have highlighted which nodes are moved and to where they have been moved.
\end{eg}

 First we require some combinatorics.
   We define the ($e$-){\sf ladder number} of a node $(r,c)\in [\lambda] $ to be $\mathfrak{l}(r,c)=r+c(e-1)$.
The $i$th {\sf ladder} of $\la$ is defined to be the set $$\mathscr{L}_i=\{(r,c)\in\mathbb{N}^2 \mid \mathfrak{l}(r,c)=i\}\cap[\la].$$
    Given a node $(r,c)\in [\la]$ we define the ($e$-){\sf residue} to be $\mathfrak{l}(r,c)$ modulo $e$.
 The $e$-regularisation of $\lambda$ is the partition ${\sf R}(\lambda)$
 obtained
 by moving all of the nodes of $\lambda$
  as high along their ladders as possible.
 When $q=-1$, each ladder of $\la$ is a complete north-east to south-westerly diagonal in $[\la]$.
 In particular, when $e=2$ the partition ${\sf R}(\lambda)$ is obtained from   $\la$ by sliding nodes  as high along their south-west to north-easterly diagonals as possible.

\begin{thm}[James' regularisation theorem]\label{James-reg}
Let $\la$ be a partition of $n$ and $\Bbbk$ be an arbitrary field.
  We have that
 $[{\bf S}^\Bbbk_q(\la):{\bf D}^\Bbbk_q(\mu)] $ is equal to 1 if $\mu =\la^R$ and is zero if $\mu\lhd \la^R$.
\end{thm}

\subsection{Brauer--Humphrey's reciprocity}
Given $\la$ an $e$-regular partition and
${\bf D}^\Bbbk_q(\la)$ the corresponding simple  $H^\Bbbk_q(n)$-module, we let
${\bf P}_q^\Bbbk(\la)$  denote its projective cover.
Brauer--Humphrey's reciprocity states that  ${\bf P}_q^\Bbbk(\mu)$ has a {\sf Specht filtration}, in other words
$$
0= S_1  \subset
S_2
\subset \dots
\subset S_z= {\bf P}_q^\Bbbk(\mu)
$$such that for each $1\leq r\leq z$, we have
$
S_r/ S_{r-1}
\cong
 {\bf S}^\Bbbk_q(\mu)$ for some $\la\in \ptn n$ dependent on $1\leq r \leq z$.
 Furthermore Brauer--Humphrey's reciprocity states that the multiplicity, $[ {\bf P}_q^\Bbbk(\mu): {\bf S}_q^\Bbbk(\la)]$,
 of ${\bf S}_q^\Bbbk(\la)$ in such a filtration is well-defined and that
\begin{equation}\label{labelll}[ {\bf P}_q^\Bbbk(\mu): {\bf S}_q^\Bbbk(\la)]=
  [{\bf S}^\Bbbk_q(\la):{\bf D}^\Bbbk_q(\mu)].\end{equation}
In other words, the $\lambda$th column  of the decomposition matrix determines the  Specht filtration multiplicities for ${\bf P}^\Bbbk_q(\la)$.
This will be a key observation for our applications to Kronecker coefficients in \cref{kronsax}.

  \subsection{2-blocks}\label{facts2}
We first recall the block-structure of Hecke algebras in (quantum)  characteristic $e=2$ (which will be the main case of interest in this paper).
Throughout this section $e=2$ and $\Bbbk$ can be taken to be an arbitrary field (although we are mainly interested in the cases when $\Bbbk=\CC$ or $\Bbbk$ is of characteristic $p=2$).
  The algebra $H^\Bbbk_{-1}(n)$ decomposes as a direct sum of primitive 2-sided ideals, called {\sf blocks}.
 All questions concerning modular representation theory break-down block-by-block according to this decomposition: in particular each simple/Specht module belongs to a unique block.

The rim of the Young diagram of $\la\vdash n$  is the collection of nodes
 $R[\la]=\{(r,c)\in[\la] \mid (r+1,c+1)\not \in [\la] \}$.
 Given  $(r,c)\in[\la]$, we define the   associated {\sf rim-hook} to be the set of nodes
 $h(r,c) = \{(i,j)\in R[\la]\mid r\leq i, c\leq j\}$.
 If $|h(r,c)|=e\in \mathbb{N}$, then we  refer to $h$ as a removable $e$-hook;
 if $e=2$ we refer to $h(r,c)$ as a removable domino.
  Removing $h(r,c)$ from $[\la]$ gives the Young diagram $[\la]\setminus h(r,c)$
  of a partition of $n-e$.
 It is easy to see that a partition   has no removable dominoes if and only if it is of the form $\rho(k)=(k,k-1,k-2,\dots,2,1)$ for some $k\geq 0$; in which case we say that it is a {\sf 2-core}.
We let ${\sf core}(\la)$ denote the 2-core partition obtained by successively removing all removable dominoes from $\lambda$
(this defines a unique partition).
 The number of dominoes removed from $\la$  is referred to as the {\sf weight} of the partition $\la$ and is denoted $w(\la)$.
Given $k,n \in \NN_0$, we define $B_k(n)=\{ \lambda\in \ptn n \mid  {\sf core}(\la)=\rho(k)\} $ to be the corresponding   {\sf  combinatorial 2-block}.
The set $\ptn n$ decomposes as the disjoint union of the non-empty $B_k(n)$.  We note that
it makes sense to speak of the {\sf weight of a 2-block}  since any  two partitions in the same 2-block necessarily have the same weight.
 Two simple $H_{-1}^\Bbbk(n)$-modules
 (or irreducible characters of $\CC \mathfrak{S}_n$)
 belong to the same 2-block if and only if their labelling partitions belong to the same combinatorial 2-block.

\begin{figure}[ht!]
$$     \begin{minipage}{4cm}\begin{tikzpicture} [scale=0.33]

 \draw[very thick](0,0)--++(0:9)--++(-90:1)--++(180:1)--++(-90:1)
 --++(180:3)--++(-90:1)
 --++(180:2)--++(-90:2)
  --++(180:1)--++(-90:1)
  --++(180:1)--++(-90:5)
   --++(180:1)--++(90:11)
 ;
  \clip(0,0)--++(0:9)--++(-90:1)--++(180:1)--++(-90:1)
 --++(180:3)--++(-90:1)
 --++(180:2)--++(-90:2)
  --++(180:1)--++(-90:1)
  --++(180:1)--++(-90:5)
   --++(180:1)--++(90:11)
 ;

  \path(0,0) coordinate (origin);
   \foreach \i in {1,...,19}
  {
    \path (origin)++(0:1*\i cm)  coordinate (a\i);
    \path (origin)++(-90:1*\i cm)  coordinate (b\i);
    \path (a\i)++(-90:10cm) coordinate (ca\i);
    \path (b\i)++(0:10cm) coordinate (cb\i);
    \draw[densely dotted] (a\i) -- (ca\i)  (b\i) -- (cb\i); }

 \draw[very thick](0,0)--++(0:9)--++(-90:1)--++(180:1)--++(-90:1)
 --++(180:3)--++(-90:1)
 --++(180:2)--++(-90:2)
  --++(180:1)--++(-90:1)
  --++(180:1)--++(-90:5)
   --++(180:1)--++(90:11)
 ;
 \fill[cyan!35](0,0)--++(0:5)--++(-90:1)--++(180:1)--++(-90:1)
 --++(180:1)--++(-90:1)
 --++(180:1)--++(-90:1)
  --++(180:1)--++(-90:1)  --++(180:1)--++(-90:1)
  ;

 \draw[very thick](0,0)--++(0:7)--++(-90:1)--++(180:1)--++(-90:1)
 --++(180:3)--++(-90:1)
 --++(180:1)--++(-90:1)
  --++(180:1)--++(-90:1)
  --++(180:1)--++(-90:5)
   --++(180:1)--++(90:11)
 ;

   \draw[very thick](origin)--++(-90:7)--++(0:1)--++(90:3)--++(0:1)--++(-90:2);
      \draw[very thick](origin)--++(-90:9)--++(0:1);

 \draw[very thick](0,0)--++(0:5)--++(-90:1)--++(0:3)--++(180:4) --++(-90:1) ;

 \path(0.5,-0.5) coordinate (origin);
 \foreach \i in {1,...,19}
  {
    \path (origin)++(0:1*\i cm)  coordinate (a\i);
    \path (origin)++(-90:1*\i cm)  coordinate (b\i);
    \path (a\i)++(-90:1cm) coordinate (ca\i);
        \path (ca\i)++(-90:1cm) coordinate (cca\i);
    \path (b\i)++(0:1cm) coordinate (cb\i);
    \path (cb\i)++(0:1cm) coordinate (ccb\i);
  }

   \end{tikzpicture}\end{minipage}\qquad
   \begin{minipage}{4cm}
   \begin{tikzpicture} [scale=0.33]

 \draw[very thick](0,0)--++(0:11)--++(-90:1)--++(180:1)--++(-90:1)
 --++(180:3)--++(-90:1)
 --++(180:3)--++(-90:1)
  --++(180:1)--++(-90:1)
  --++(180:1)
   --++(180:2)--++(90:5)
 ;

   \clip(0,0)--++(0:11)--++(-90:1)--++(180:1)--++(-90:1)
 --++(180:3)--++(-90:1)
 --++(180:3)--++(-90:1)
  --++(180:1)--++(-90:1)
  --++(180:1)
   --++(180:2)--++(90:5);

  \path(0,0) coordinate (origin);
   \foreach \i in {1,...,19}
  {
    \path (origin)++(0:1*\i cm)  coordinate (a\i);
    \path (origin)++(-90:1*\i cm)  coordinate (b\i);
    \path (a\i)++(-90:10cm) coordinate (ca\i);
    \path (b\i)++(0:10cm) coordinate (cb\i);
    \draw[densely dotted] (a\i) -- (ca\i)  (b\i) -- (cb\i); }

  \draw[very thick](0,0)--++(0:9)
  --++(-90:1)
    --++(0:1)  --++(180:1)--++(180:1)--++(-90:1)--++(180:3)--++(-90:1)
    --++(180:3)--++(-90:1)--++(0:1) ;
% --++(180:3)--++(-90:1)
% --++(180:2)--++(-90:2)
%  --++(180:1)--++(-90:1)
%  --++(180:1)--++(-90:5)
%   --++(180:1)--++(90:11)
 ;
 \fill[cyan!35](0,0)--++(0:5)--++(-90:1)--++(180:1)--++(-90:1)
 --++(180:1)--++(-90:1)
 --++(180:1)--++(-90:1)
  --++(180:1)--++(-90:1)  --++(180:1)--++(-90:1)
  ;

 \draw[very thick](0,0)--++(0:7)--++(-90:1)--++(180:1)--++(-90:1)
 --++(180:3)--++(-90:1)
 --++(180:1)--++(-90:1)
--++(180:1)--++(-90:1)
--++(180:1)--++(-90:5)
--++(180:1)--++(90:11)
 ;

      \draw[very thick](origin)--++(-90:9)--++(0:1);

  \draw[very thick](0,0)--++(0:5)--++(-90:1)--++(0:3)--++(180:4) --++(-90:1) ;

 \path(0.5,-0.5) coordinate (origin);
 \foreach \i in {1,...,19}
  {
    \path (origin)++(0:1*\i cm)  coordinate (a\i);
    \path (origin)++(-90:1*\i cm)  coordinate (b\i);
    \path (a\i)++(-90:1cm) coordinate (ca\i);
        \path (ca\i)++(-90:1cm) coordinate (cca\i);
    \path (b\i)++(0:1cm) coordinate (cb\i);
    \path (cb\i)++(0:1cm) coordinate (ccb\i);
  }

   \end{tikzpicture} \end{minipage}$$
   \caption{The partitions $\tau^{(2,2,1)}_{(3,1,1)}$ and
    $\tau^{(3^2,2,1^2)}_{\varnothing}$ for $\rho = \rho(5)$.}
   \label{below}
   \end{figure}
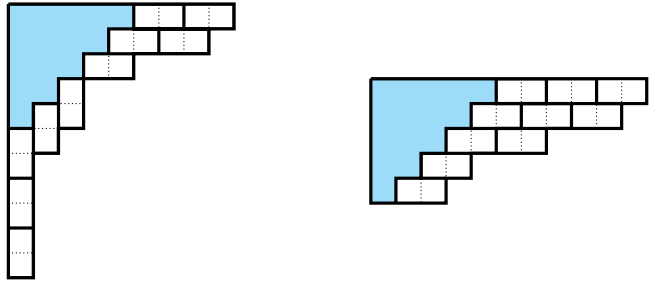

\begin{eg}
The partition $(9,8,5,3^2,2,1^5)$ has 4 removable dominoes:
two $(2)$-dominoes $\{(2,7),(2,8) \}$ and $\{(3,4),(3,5)\}$ and two
$(1^2)$-dominoes $\{(3,3),(4,3)\}$ and $\{(1,10),(1,11)\}$.  One can continue to successively remove such dominoes until one is left with the $2$-core $\rho(5)=(5,4,3,2,1)$ as depicted on the lefthand-side of  \cref{below}.
\end{eg}

   \begin{defn}\label{2sepdef}Let  $w_1,w_2   \in \NN_0$ be arbitrary   and
$\lambda\in \ptn {w_1} $ and $\mu \in \ptn {w_2} $ such that $\la^T_1+\mu_1\leq k+1$.
We let $\tau^\lambda_\mu$ denote the partition
$$
\tau^\la_\mu= (\rho(k) + 2\mu^T)^T + 2\lambda.
$$
We say that any partition, $\tau^\la_\mu$,  of this form is {\sf 2-separated}.  \end{defn}

\begin{rmk}We note that  2-separated partitions appear across all 2-blocks of the Hecke algebra.
If the weight of a block is small compared to the size of the core, then all partitions in that block are 2-separated.

\end{rmk}

\begin{rmk} While the name ``2-separated" may seem odd to some readers, it is motivated by the form this partition takes on a 2-abacus.  In   \cite{MR1402572} these partitions are referred to as  ``2-quotient separated".  \end{rmk}

\subsection{Characters of height zero}\label{subsec:height0}
%We %Let $\Bbbk$ be a field of characteristic 2 and
%  specialise $q=-1$.
We now wish to discuss the defect groups of 2-blocks of symmetric groups and their characters of height zero (see \cite{JK} for background and more details).
%We set $$P_k=\underbrace{C_2\wr C_2 \dots \wr C_2}_{\text{$k$ times}}.$$
Write $n =2^{a_1}+\ldots + 2^{a_s}$ where $a_1>\ldots > a_s \geq 0$; we set $s(n)=s$.
For $m\in \N$, let $m_2$ be the largest 2-power dividing $m$.
Then $(n!)_2 = 2^{n - s(n)}$ is the size of a Sylow 2-subgroup of $\mathfrak{S}_n$.
%Write $2w= 2^{w_1} + \dots + 2^{w_{\sigma(w)}}$, where $w_1> w_2 > \dots > w_{\sigma(w)} > 0.$
% Every Sylow $2$-subgroup of $\mathfrak{S}_{n}$ is conjugate to the direct product
% $$
% P_0^{s_0}\times  P_1^{s _1}\times \dots \times  P_t^{ s_{\sigma(n)}}.
% $$
Let   $B$ be a  $2$-block of $\mathfrak{S}_n$  of weight $w$;
then a defect group of   $B$    is
%conjugate to
% $$ \underbrace{C_2\wr C_2 \dots \wr C_2}_{\text{$w_1$ times}}\times  \underbrace{C_2\wr C_2 \dots \wr C_2}_{\text{$w_2$ times}} \times\dots   \times   \underbrace{C_2\wr C_2 \dots \wr C_2}_{\text{$\sigma(w)$ times}}.$$
isomorphic to a Sylow 2-subgroup of $\mathfrak{S}_{2w}$,
 and thus is of cardinality  $2^{2w-s(w)}$; the number $d(B)= 2w-s(w)$
 is called the \emph{defect} of~$B$.

\begin{eg}
  The   five 2-blocks of $\mathfrak{S}_{36}$ are indexed by the 2-cores
   $\varnothing, \rho(3),  \rho(4),  \rho(7)$ and $\rho(8)$.
   These blocks are of weight $18, 12, 10, 4,$ and $0$ respectively.
   Since $18=2^4+2$, the 2-block $B$ of weight 18 has defect $d(B)=34$.
%$C_2\wr C_2\wr C_2\wr C_2\wr C_2 \times C_2\wr C_2$.
%Similarly, the 2-adic expansion for the
%
%$C_2\wr C_2\wr C_2\wr C_2 \times  C_2  \wr C_2  \wr C_2 \times  C_2  \wr C_2 \times C_2 $,
\end{eg}

 We now recall the important notion of height 0 characters and simple modules.
First we recall the fact that the dimension of any simple module ${\bf D}^\Bbbk(\la)$ or
${\bf S}^\CC (\mu)$  belonging to a $2$-block $B$ of the symmetric group $\mathfrak{S}_n$
is divisible by $2^{n-s(n)-d(B)}$.
Such a module is said to be {\sf of height 0} if
this 2-power is the largest 2-power dividing its dimension.
%the $2$-part of the integer $$\frac{|\mathfrak{S}_n|}{\dim_\Bbbk({\bf D}^\Bbbk(\la))}$$
%is exactly the order, $2^{d(B_k(n))}$, of the  defect group of  the block $B_k(n)$.
We also say that the character $\chi^\mu$ associated to the Specht module  ${\bf S}^\CC (\mu)$
is a  {\sf height 0 character}.
For the 2-block $B$, we set
$${\rm Irr}^\CC_0(B)= \{ \chi^\lambda \mid \chi^\la \text{ is a    height zero character of  }  B\}.
$$
Generalising an earlier result of Macdonald on characters of odd degree, a combinatorial description for the partitions labels of height 0 characters was
given in terms of the so-called 2-core tower by Olsson (see \cite{O76, O-LN}).
A new characterisation was recently given in  \cite[Section 3.2]{MR3829558}, again
generalising an earlier version for the principal 2-block.
This says that a partition
 $\la$  in a 2-block $B=B_k(n)$  of weight
 $w=2^{w_1}+\ldots + 2^{w_{s(w)}}$, where $w_1 > \ldots > w_{s(w)}\ge 0$,
 labels a height~0 character if and only if there is a sequence
 $$ \la=\la(0)\supset \la(1) \supset  \dots \supset   \la(s(w)-1) \supset   \la(s(w))= \rho_k$$
  of partitions such that $\la(i-1) \setminus \la(i)$ is a
  $2^{w_i}$ rim-hook  for $i=1,\ldots , s(w)$.

A formula for the number $k_0(B)$ of height 0 characters in a 2-block
of weight $w$ was already given by Olsson \cite{O76}, and was also deduced
from the description above in \cite{MR3829558}.
With $B$ and $w$ given as above, we have
$$k_0(B)=\prod_{j=1}^{s(w)} 2^{w_j+1}.$$
 Since we get this number for each 2-block, the set of height 0 characters
 $\chi^\mu$ for $\mathfrak{S}_n$
 constitutes quite a large class of irreducible characters.

\begin{eg}
The irreducible  characters of height 0 belonging to the principal 2-block
of $\mathfrak{S}_n$
 are precisely  the characters of odd degree.

\end{eg}

\begin{eg}
The partitions $(9,8,3,2^3,1^2), (9,6,5,2^3,1^2) \vdash 28$
  and the partitions
  $(9,6,5^3,4,1^2)  $ $ (10,7,4^3,3,2^2)\vdash 36 $ all label height zero characters.
These partitions are depicted in \cref{2adicpic2836} in such a manner as to illustrate their combinatorial construction via adding rim hooks (detailed above).
\end{eg}

 \begin{eg}\label{36}
  The   five 2-blocks of $\mathfrak{S}_{36}$,
 their weights $w$, the 2-adic expansions of $2w$,  and the number of height 0 characters in the 2-block are recorded in the table below.
$$\begin{array}{c|c|c|c|ccc}
\text{$2$-core} & \text{weight $w$}		& 2w &  k_0(B)  \\
\hline
\varnothing & 18 		&2^5+2^2	& 2^7\\
\rho(3) & 15		&2^4+2^3+2^2+2	& 2^{10}\\
\rho(4) & 13		&2^4+2^3 + 2	& 2^{8}\\
\rho(7) & 4			&2^3	&  2^3 \\
\rho(8) & 0			& -	&  1 \\
\end{array}$$
In particular, there are in total 1417 height 0 characters in 2-blocks of $\mathfrak{S}_{36}$, amongst which there are 128 of odd degree.
\end{eg}

 The following theorem will be one of the key results we use later on.
 It says that while there are many complex characters of height 0, there
  is only one simple module  ${\bf D}^\Bbbk(\lambda )$ of height 0 in each 2-block.

\begin{thm}[{\cite[Theorem 1.4]{KOW}}]\label{KOW}
Let $\Bbbk$ be a field of characteristic 2.
For any 2-block $B$ of weight $w$ of the symmetric group $\mathfrak{S}_n$,
the module ${\bf D}^\Bbbk(\lambda)$ to the most dominant
partition $\lambda=\tau^{(w)}_\varnothing$
is the unique simple $\Bbbk\mathfrak{S}_n$-module in $B$
  of height 0.

\end{thm}

\section{KLR algebras and coloured tableaux }

 \label{clored}
 Given $n\in \mathbb{N}$ and an indeterminate $t$
we define  the {\sf quantum integers}
 and  {\sf quantum factorials}
$$
[n]_t= \frac{1+ t^2+ t^{4}\dots +  t^{2 n-2}}{t^{n-1}}
\qquad
[n]_t!=[1]_t[2]_t\dots[n]  _t
$$
and given $\mu \in \ptn n$ a partition of length $\ell$, we set
$$
 [\mu]_t! = [\mu_1]_t! [\mu_2]_t!\dots  [\mu_\ell]_t!
$$
 We now define the  {\sf quantum binomial coefficients} to be
\[
\stirling{a}{b}_t:=\frac{[a]_t!}{[b]_t![a-b]_t!},
\]
for all $a\geqslant b \geqslant 0$.
  The motivating observation for studying Hecke algebras is the following.
Let $\Bbbk$ be a field of characteristic $p$ and let $q\in\Bbbk$ be an element of order $e=p$:
 then $H^\Bbbk_q(n)$ is isomorphic to $\Bbbk\mathfrak{S}_n$.
  This gives us a way of factorising representation theoretic questions into two steps:
 firstly {\em specialise the quantum parameter}
 $q$ to be a $p$th root of unity
 (in $\CC$ and compatibly in $\Bbbk$)
 and study the non-semisimple algebra $H^\CC_q(n)$;
 then {\em reduce modulo $p$}  by studying $H^\Bbbk_q(n)=H^\ZZ_q(n)\otimes _\ZZ \Bbbk$.
This allows us to  factorise the problem of understanding decomposition matrices  as follows,
\begin{equation}\label{adjust}
[{\bf S}^\Bbbk_q(\la): {\bf D}_q^\Bbbk(\mu)]=
[{\bf S}_q^\CC(\la): {\bf D}^\CC_q(\nu)] \times
[{\bf D}_q^\CC(\nu)\otimes_\ZZ\Bbbk: {\bf D}^\Bbbk_q(\mu)].
\end{equation}
 On the right-hand side of the equality we have two matrices: the first is the {\sf decomposition matrix for $H^\CC_q(n)$}
and the second is known as {\sf ``James' adjustment matrix"}.  Therefore understanding the decomposition matrix of $H^\CC_q(n)$ serves as a first step toward understanding the decomposition matrix of $\Bbbk\mathfrak{S}_n$.

We now recall the manner in which the grading can be incorporated into the picture and its immense power in understanding the decomposition matrix for $H^\CC_q(n)$ (and hence, by \cref{adjust} give us a method for attacking the problem of calculating decomposition numbers for symmetric groups).
  Let  $t$ be an indeterminate over $\ZZ$.  The following theorem provides us with a
  $\ZZ$-graded presentation (which we record with respect to the indeterminate $t$) of the Hecke algebra.

\begin{thm}[\cite{MR2551762,MR2525917,ROUQ}]
\label{defintino1}
The  Hecke algebra    $H_q^\Bbbk(n)$
admits a graded presentation with generators
\begin{align}\label{gnrs}
\{e(\underline{i}) \ | \ \underline{i}=(i_1,\dots,i_n)\in   (\ZZ/e\ZZ)^n\}\cup\{y_1,\dots,y_n\}\cup\{\psi_1,\dots,\psi_{n-1}\},
\end{align}
subject to   a list of     relations given in \cite[Main Theorem]{MR2551762}.
The $\ZZ$-grading on $H_q^\Bbbk(n)$ is given by
$$
{\rm deg}(e(\underline{i}))=0, \quad
{\rm deg}( y_r)=2,\quad
{\rm deg}(\psi_r e(\underline{i}))=
\begin{cases}
-2		&\text{if }i_r=i_{r+1}, \\
1		&\text{if }i_r=i_{r+1}\pm 1 \ \& \ e\neq 2, \\
2		&\text{if }i_r=i_{r+1}\pm 1 \ \& \ e= 2, \\   %i_{r+1}\neq i_r. \\ %- 1 = i_r=i_{r+1}+ 1 \\
0 &\text{otherwise.}
\end{cases} $$
\end{thm}

The importance of  \cref{defintino1} is that it allows to consider an extra, richer graded structure on the Specht modules.
 We now recall the definition of this grading on the tableau basis of the Specht module.
Let $\la\in\ptn n$ and $\stt \in \Std(\la)$.
  We let $\stt^{-1}(k)$ denote the node  in $\stt$ containing the integer $k\in\{1,\dots,n\}$.
  Given $1\leq k\leq n$, we let ${\mathcal A}_\stt(k)$,
(respectively ${\mathcal R}_\stt(k)$)  denote the set of   all addable $\res (\stt^{-1}(k))$-nodes  (respectively all
removable   $\res (\stt^{-1}(k))$-nodes )  of the partition $\Shape(\stt{\downarrow}_{\{1,\dots ,k\}})$ which
   are above  $\stt^{-1}(k)$, i.e. those in an earlier row.
  Let $\la\in\ptn n$ and $\stt \in \Std(\la)$.  We define the degree of $\stt$ as follows:
$$
\deg(\stt) = \sum_{k=1}^n \left(	|{\mathcal A}_\stt(k)|-|{\mathcal R}_\stt(k)|	\right).
$$
  Given
$\stt\in \Std(\la)$  we  define the  {\sf  residue sequence} of $\stt$   as follows:
$$
\res(\stt) = (\res (\stt^{-1}(1)), \res (\stt^{-1}(2)),  \dots, \res(\stt^{-1}(n)) )  \in (\ZZ/e\ZZ)^n.
$$
 Let $t$ be an indeterminate over $\NN_0$. If $M=\oplus_{z\in\ZZ}M_z$ is
a free graded $\Bbbk$-module,
 then its \emph{graded dimension} is the Laurent  polynomial
\[\Dim{(M)}=\sum_{k\in\ZZ}(\dim_{\Bbbk}M_k)t^k.\]
If $M$ is a graded $H^\Bbbk_q(n)$-module and $k\in\ZZ$, define $M\langle k \rangle$ to be the same module with $(M\langle k \rangle)_i = M_{i-k}$ for all $i\in\ZZ$. We call this a \emph{degree shift} by $k$.  The graded dimensions of Specht modules admit a combinatorial description as follows:

\begin{thm}[\cite{bkw11}]\label{grading}
The Specht module ${\bf S}^\Bbbk_q(\la)$ is a free $\ZZ$-graded $\Bbbk$-module with basis  $\{\psi^\stt
 \ | \ \stt\in\Std(\lambda)\}$ and
where $\deg(\psi^\stt)=t^{\deg(\stt)}$.
 \end{thm}

Of course, this theorem gives us an added level of graded structure to consider: the {\sf graded} decomposition numbers of symmetric groups and their Hecke algebras.
By \cref{grading}, we obtain a grading on the module ${\bf D}_q(\mu)=  {\bf S}_q(\la)/ \rad ({\bf S}_q(\la))$.
 We define the {\sf graded decomposition number} to be the polynomial
\begin{equation}\label{gradeddecompdef}
 d_{\lambda,\mu}^\Bbbk(t)= \sum_{k \in \ZZ } [{\bf S}^\Bbbk_q(\la): {\bf D}^\Bbbk_q(\mu)\langle k \rangle ] t^k
\end{equation}
which records the composition multiplicity of each simple module and its relevant degree shift.
In particular upon specialisation $t\to 1$ the polynomials of \cref{gradeddecompdef} specialise to be the usual decomposition numbers.
While one might expect this grading to {\em increase} the level of difficulty of our question,
we find that by keeping track of this extra  grading information we are
{rewarded}
with an incredibly powerful
algorithm for understanding the decomposition numbers  of $H^\CC_q(n)$.

 Equation~(\ref{adjust}) hints  that we could first study the decomposition numbers of  $H^\CC_q(n)$
 as an intermediary first step toward understanding the decomposition numbers of symmetric groups in positive characteristic.  In fact, this approach has been incredibly successful: Lascoux, Leclerc and Thibon provided an iterative algorithm for understanding the {\em graded} decomposition numbers of  $H^\CC_q(n)$ in \cite{LLT}.
We now provide an elementary tableau-theoretic re-interpretation of this algorithm (using the work of Kleshchev and Nash \cite{KN10}).

\subsection{Coloured tableaux   }
We now recast ideas from \cite{KN10} in terms of  orbits of tableaux which we encode as ``coloured tableaux".
  Let   $\lambda   $ a partition  of $n$ and $\mu$ a composition of $n$.
 We define a {\sf    Young tableau of shape $  \lambda$ and weight $\mu$} to be  a filling of the nodes  of
   $\la$  with the entries
    $$\underbrace{1, \dots, 1}_{\mu_1}, \underbrace{2,\dots, 2}_{\mu_2},
  \ldots,   \underbrace{\ell ,\dots, \ell }_{\mu_\ell }. $$
We say that a tableau is {\sf row standard} if the entries are  weakly increasing along the rows  of $  \lambda$; we  denote the set of such tableaux by $\RStd(\la,\mu)$.
   We say that the  Young tableau is {\sf semistandard} if  the entries are  weakly increasing along the rows and
   are strictly increasing along the columns of $  \lambda$; we denote the set of such tableaux by $\SStd(\la,\mu)$.

\begin{defn} Given     $\mu\vdash_e  n$, we let ${\rm Lad}(\mu)$ denote the composition $\nu$ such that
$$
{\nu_i =\sharp\{ (r,c) \in \mu \mid \mathfrak{l}(r,c)=i \}.}
$$
where we have that $\nu_1=0$ by definition.   We define a semistandard coloured tableaux,
$\SSTS$,  to be a
semistandard tableau  of weight ${\rm Lad}(\mu)$ such that the entry of any node  is congruent to its residue.  We denote the set of all such tableaux of shape $\lambda$
by ${\rm CStd}(\la, \mu )\subseteq {\rm SStd}(\la, {\rm Lad}(\mu) )
$.
 We let $\Lad^\mu$ denote the unique element of
${\rm CStd}(\mu, \mu )$.
 We set
 $e(\mu)=e(\res (\Lad^\mu)) \in H^\Bbbk_q(n)$.

\end{defn}

\begin{eg} \label{eis2ex}
For $q=-1$, $e=2$
and
$
\mu=(6) , (5,1), (4,2)
$,
we have that
${\rm Lad}(\mu)$ is equal to $(0,1,1,1,1,1,1)$,
 $(0,1,2,1,1,1)$ and
 $(0,1,2,2,1)$ respectively.
 All semistandard coloured tableaux (up to conjugation) for the principal 2-block of $H^\Bbbk_{-1}(6)$
are listed in the table below.

$$
\renewcommand{\arraystretch}{1}
\begin{array}{c|c|c|c}
{\rm CStd}(\la, \mu ) & 6 & 5,1 & 4,2  \\
\hline
6 & \Yvcentermath1 \Ylinecolour{white}\gyoung(;,;)\Ylinecolour{black} \gyoungxy(0.9,0.9,!\gr2!\wh3!\gr4!\wh5!\gr6!\wh7) & * \Ylinecolour{white}\gyoung(;,;)\Ylinecolour{black} & * \\
\hline5,1 & \Yvcentermath1  \Ylinecolour{white}\gyoung(;,;,;)\Ylinecolour{black} \gyoungxy(0.9,0.9,!\gr2!\wh3!\gr4!\wh5!\gr6,!\wh7)  \Ylinecolour{white}\gyoung(;,;,;)\Ylinecolour{black} & \Yvcentermath1 \Ylinecolour{white}\gyoung(;,;,;)\Ylinecolour{black} \gyoungxy(0.9,0.9,!\gr2!\wh3!\gr4!\wh5!\gr6,!\wh3)  \Ylinecolour{white}\gyoung(;,;,;)\Ylinecolour{black} & * \\
\hline4,2 &  \Yvcentermath1\Ylinecolour{white}\gyoung(;,;,;) \Ylinecolour{black}\gyoungxy(0.9,0.9,!\gr2!\wh3!\gr4!\wh7,;5!\gr6) \Ylinecolour{white}\gyoung(;,;,;) &  \Yvcentermath1\Ylinecolour{white}\gyoung(;,;,;)\Ylinecolour{black} \gyoungxy(0.9,0.9,!\gr2!\wh3!\gr4!\wh5,;3!\gr6)  \Ylinecolour{white}\gyoung(;,;,;)&
 \Yvcentermath1 \Ylinecolour{white}\gyoung(;,;,;)\Ylinecolour{black}  \gyoungxy(0.9,0.9,!\gr2!\wh3!\gr4!\wh5,;3!\gr4)
  \Ylinecolour{white}\gyoung(;,;,;)
 \\
  \hline4,1^2 & \Yvcentermath1 \Ylinecolour{white}\gyoung(;,;,;,;)\Ylinecolour{black} \gyoungxy(0.9,0.9,!\gr2!\wh5!\gr6!\wh7,;3,!\gr4) \  \Yvcentermath1\gyoungxy(0.9,0.9,!\gr2!\wh3!\gr4!\wh7,;5,!\gr6) &  \Yvcentermath1
\Ylinecolour{white}\gyoung(;,;,;)\Ylinecolour{black} \gyoungxy(0.9,0.9,!\gr2!\wh3!\gr4!\wh5,;3,!\gr6)  \Ylinecolour{white}\gyoung(;,;,;,;)\Ylinecolour{black} &
 \Yvcentermath1 \Ylinecolour{white}\gyoung(;,;,;,;)\Ylinecolour{black}\gyoungxy(0.9,0.9,!\gr2!\wh3!\gr4!\wh5,;3,!\gr4)  \Ylinecolour{white}\gyoung(;,;,;,;)\Ylinecolour{black} \\

\hline3^2 & * &   * &  \Yvcentermath1 \Ylinecolour{white}\gyoung(;,;,;)\Ylinecolour{black}\gyoungxy(0.9,0.9,!\gr2!\wh3!\gr4,!\wh3!\gr4!\wh5)  \Ylinecolour{white}\gyoung(;,;,;)\Ylinecolour{black}  \\
\end{array}$$

 \end{eg}

 The importance of coloured  semistandard   tableaux of weight $\mu$ is that they encode
 an $\mathfrak{S}_{{\rm Lad}(\mu)}$-orbit  of   standard  Young tableaux;   we shall now make this idea more  precise.
 Given a composition $\nu$ and $c\geq1$,  we set $[\nu]_c= \nu_1+\nu_2+\dots + \nu_c\in \mathbb{N}$ and we set $\nu_0=0$.
 Let $\mu$  be an $e$-regular partition and   let $\sts$ be a standard   Young tableau of  shape     $\la$  such that the residue sequence of $\sts$ is given by
	$$ 0  ,\underbrace{-1, -1, \dots, -1}_{\nu_3\text{ times}}, \underbrace{-2, -2,\dots, -2}_{\nu_4\text{ times}},
\underbrace{-3, -3, \dots, -3}_{\nu_5\text{ times}},   \ldots    $$
 for $\nu=(0,1,\nu_3,\dots,\nu_\ell)={\rm Lad}(\mu)$; we refer to such an $\sts$ as a {\sf ladder tableau of ladder weight}~$\mu$.
Then define $\mu(\sts)$ to be  the  coloured tableau
 obtained from $\sts$ by replacing each entry $i$ for
  $[{{\rm Lad}(\mu)}]_{c-1} < i \leq [{{\rm Lad}(\mu)}]_c$ in $\sts$ by the entry $c$  for  $  c \geq  1$.

We identify a coloured semistandard   Young tableau, $\SSTS $, of weight $\mu$ with the set of standard  Young tableaux, $[\SSTS]_\mu=\{\sts \mid \mu(\sts)=\SSTS\}$.
 Given  $\SSTS\in  \SStd (\la,\mu)$  we  let $\sts^\lambda\in  [\SSTS]_{ \mu} $
denote the unique most dominant tableau in $ [\SSTS]_{ \mu}$.

\begin{eg}\label{followup}
Continuing with \cref{eis2ex}, we let $\SSTS\in {\rm CStd}((4,1^2),(4,2))$  depicted above.  We have that
$$
 \left[\ \Yvcentermath1\gyoungxy(0.9,0.9,!\gr2!\wh3!\gr4!\wh5,;3,!\gr4)
 \ \right]_{\!\mu}=\left \{
 \Yvcentermath1 \gyoungxy(0.9,0.9,!\gr1!\wh2!\gr4!\wh6,;3,!\gr5) \ , \
 \Yvcentermath1 \gyoungxy(0.9,0.9,!\gr1!\wh3!\gr4!\wh6,;2,!\gr5)  \ , \
 \Yvcentermath1 \gyoungxy(0.9,0.9,!\gr1!\wh3!\gr5!\wh6,;2,!\gr4)  \ , \
  \Yvcentermath1 \gyoungxy(0.9,0.9,!\gr1!\wh2!\gr5!\wh6,;3,!\gr4)
\right\}
$$
as an orbit of standard tableaux (which we have coloured in order to facilitate comparison).  \end{eg}

For $\mu \in \regptn e n$, we let $\SSTT^\mu $   be the unique element of ${\rm CStd}(\mu,\mu)$.
 We  set $\deg(\SSTT^\mu)=0$  so that
$$
\Dim (e(\mu){\bf S}_q(\mu))= \sum_{\stt\in\SSTT^\mu}t^{\deg(\stt)}
=
 [{{\rm Lad}(\mu) }]_t! =
 [{{\rm Lad}(\mu) }]_t!   \times t^{\deg(\SSTT^\mu)}
$$
which is  invariant under the {\sf bar map}   interchanging $t \leftrightarrow t^{-1}$ (see also \cite[Lemma 3.4]{KN10}).
We now provide a  general definition of the degree of a coloured tableau which allows us to calculate the graded characters of weight spaces of Specht modules in terms of coloured tableaux.
 Let $(a,b)\in \la\in\ptn n$ be a node of residue $i\in \ZZ/e\ZZ$
  and  $\mu \in \regptn en  $, $\SSTS \in {\rm CStd}(\la,\mu)$.
     We let ${\mathcal A}_\SSTS(a,b)$  denote the set  of   all addable
      $i$-nodes  of the partition
$$\la \cap \{(r,c) \mid \SSTS(r,c)\leq  \mathfrak{l}(a,b)\}$$
which are above $(a,b)\in\la$.
 We let ${\mathcal R}_\SSTS(a,b)$  denote the set  of   all removable $i$-nodes  of the partition
$$\la \cap \{(r,c) \mid \SSTS(r,c) <  \mathfrak{l}(a,b)\}$$
which are above $(a,b)\in\la$.
 We then define the degree of the node $(a,b)\in \lambda$ to be
$|{\mathcal A}_\SSTS(a,b)|-|{\mathcal R}_\SSTS(a,b)|$.
   We define  $\deg(\SSTS)$ to be the sum over the degrees of all nodes $(a,b)\in \la$.

We have seen that the tableaux of ${\rm CStd}(\la,\mu)$ are simply the orbits of  tableaux  from $\Std(\lambda)$ with a given residue sequence.  Therefore, by comparing the degree function for coloured tableaux with that of standard tableaux we obtain
\begin{equation}\label{ineedalabel2}
\Dim (e(\mu){\bf S}_q(\lambda))=
%[\mu]!
 [{{\rm Lad}(\mu) }]_t!
\sum _{\SSTS \in {\rm CStd}(\la,\mu)} t^{\deg(\SSTS)}.
\end{equation}
And so coloured standard tableaux provide a combinatorial description of the ladder-weight multiplicity as defined in \cite[Section 3.3]{KN10}.

   \begin{eg}
Continuing with   $\SSTS\in {\rm CStd}((4,1^2),(4,2))$  in  \cref{followup},
we have that
$$\mathcal{R}_\SSTS(a,b)=\emptyset \text{ for all }(a,b)\in  (4,1^2)
\quad\text{and}\quad
\mathcal{A}_\SSTS(a,b)=\begin{cases}
\{(2,2)	\} &\text{if }(a,b)=(3,1) \\
\emptyset 	&\text{otherwise}.
\end{cases}
$$
and therefore
$$
\deg_\SSTS(a,b)=\begin{cases}
1 &\text{if }(a,b)=(3,1) \\
0	&\text{otherwise}.
\end{cases}
$$
Therefore $\deg(\SSTS)=1$.
The   four distinct standard tableaux $\sts\in [\SSTS]_{(4,2)} $  are depicted in \cref{followup}; these four tableaux are obtained from each other by permuting the pairs $2, 3$ in the third ladder and the pairs $4, 5$ in the fourth ladder.
We have that $$\sum_{\sts \in\SSTS}t^{\deg(\sts)}=t^3+2t+t^{-1}= t\times (t+t^{-1}) ^2=  \deg(\SSTS)	 \times [2]_t! [2]_t!  =
 \deg(\SSTS)	\times [{\sf Lad}(\mu)]_t! $$

\end{eg}

With our new tableaux theoretic combinatorics in place, we can recast the (LLT) algorithm
from  \cite[Section 4]{KN10} in this combinatorial setting.

\begin{eg}\label{thisisalabel}
 We record the graded degrees of the coloured tableaux appearing in   \cref{eis2ex}
 and their conjugates (which are not pictured).  Notice that   conjugation does not preserve the degrees of    tableaux.
$$ \begin{array}{c|ccc}
  & 6 & 5,1 & 4,2\\
\hline
6 & 1 & * & *\\
5,1 & t & 1 & *\\
4,2 & 1 & t & 1\\
4,1^2 & 2t & t^2 & t \\
3^2 & * & * & t\\
 2^3 & * & * & t^2\\
 3,1^3 & 2t^2 & t & t^2 \\
  2^2,1^2 &   t^3 & t^2& t^3 \\
  2,1^4 & t^2 & t^3   & * \\
  1^6 &  t^3 & * & * \\
\end{array}
$$

\end{eg}

We are almost ready to restate the LLT algorithm in terms of our  combinatorics, we simply require two observations about the graded  structure of the Hecke algebra.  The first is almost trivial, but the proof of the  latter depends on incredibly deep  geometric or categorical insights.

\begin{thm}[{\cite[Theorem 4.18]{bk09}}]\label{bk09}
For $\la \in \ptn n$ and $\mu \in \regptn en$, the polynomial
$\Dim  (e(\mu) {\bf D}^\Bbbk_q(\la))$ is bar-invariant
(i.e., fixed under interchanging $t$ and $t^{-1}$).
\end{thm}
\begin{thm}[\cite{MR1722955}]\label{VV99} Let $\Bbbk=\mathbb C$.
For $\la \in \ptn n$ and $\mu \in \regptn en$
with $\mu\ne \la$,
$d_{\lambda,\mu}(t) \in t\NN_0[t]$.
\end{thm}

 Rearranging \cite[Theorem 3.8]{KN10} in terms of our  coloured tableaux, we obtain the following relationships between coloured tableaux, simple characters, and graded decomposition  numbers:

   \begin{prop}\label{troll}
For $\la \in \ptn n$ and $\mu \in \regptn en$
we
have that  $$ { \Dim{(e(\mu){\bf S}^{\mathbb C} (\lambda))}}  =  \!\!\sum_{\SSTS\in {\rm CStd}(\la,\mu)}t^{\deg(\SSTS)}{[ {\rm Lad}(\mu)]_t! }
 \in \NN_0[t,t^{-1}] \quad \text{and}\quad {\Dim{(e(\mu){\bf D}^\Bbbk_q(\lambda))}}  \in \NN_0[t+t^{-1}].$$
Moreover, the following hold:
\begin{itemize}[leftmargin=*]
\item[$(i)$]   if $ {\rm CStd}(\la,\mu)=\emptyset $, then $d_{\lambda,\mu}(t)=0$ and  $\Dim{(e(\mu){\bf D}^\Bbbk_q(\la) )}=0$;
\item[$(ii)$] we have $ \Dim{(e(\mu){\bf S}_q^\Bbbk (\mu)))} =  \Dim{(e(\mu){\bf D}^\Bbbk_q(\mu))} ={[ {\rm Lad}(\mu)]_t! } $;
\item[$(iii)$] we have that
\begin{equation*}
 {\Dim(e(\mu){\bf D}_q^\Bbbk(\la))} +d_{\lambda,\mu}(t){ [{{\rm Lad}(\mu) }]_t!}=
\!\!\! \sum_{\SSTS \in {\rm CStd}(\la,\mu)}\!\!\!\!\!\!\! t^{\deg(\SSTS)}{ [{{\rm Lad}(\mu) }]_t!}-
 \!\!\sum_{
\begin{subarray}c
\la \lhd \nu \lhd \mu\end{subarray}
}   \!\!\!{\Dim(e(\mu){\bf D}^\Bbbk_q(\nu))  d_{\lambda,\nu}(t)} \end{equation*}
\end{itemize}
   \end{prop}

 Now we set $\Bbbk=\CC$. The right-hand side of  the equation in \cref{troll}$(iii)$ is calculated by induction along the   dominance ordering. Any polynomial in $\NN_0[t,t^{-1}]$ can be written {\em uniquely} as the   sum of a bar-invariant polynomial from $\NN_0[t,t^{-1}]$ and a polynomial from $ t\NN_0[t]$.
Putting together \cref{bk09,VV99} we deduce that  the lefthand-side is uniquely determined by the righthand-side and induction on the dominance order.

\begin{eg}\label{eis2exfin}
We continue with \cref{eis2ex}.  Using the equation in \cref{troll}$(iii)$, we obtain the first 5 rows of the  graded decomposition matrix of the principal block of $H^\CC_{-1}(6)$  and
$\tfrac{1}{{ [{{\rm Lad}(\mu) }]_t!}}{\Dim (e(\mu){\bf D}_q(\la))}$   as follows:
$$ \begin{array}{c|ccc}
 & 6 & 5,1 & 4,2\\
\hline
6 & 1 & * & *\\
5,1 & t & 1 & *\\
4,2 & * & t & 1\\
4,1^2 & t & t^2 & t \\
3^2 & * & * & t\\
 2^3 & * & * & t^2\\
 3,1^3 &  t^2 & t & t^2 \\
  2^2,1^2 &   * & t^2& t^3 \\
  2,1^4 & t^2 & t^3   & * \\
  1^6 &  t^3 & * & * \\
\end{array} \qquad\qquad
\begin{array}{c|ccc}
 & 6 & 5,1 & 4,2\\
\hline
6 & 1 & * & *\\
5,1 & * & 1 & *\\
4,2 & 1 & * & 1\\
 \end{array}
$$ Notice that if we multiply  these two matrices together we obtain the matrix from \cref{thisisalabel}.
The remaining entries of the table can be deduced by applying the sign automorphism to the Specht modules (although this automorphism is not of degree zero and so the   entries will differ by a degree shift).
Comparing with the table in \cref{thisisalabel},
we observe that
the entry in the row labelled by $(4,2)$ and column labelled by $(6)$ is bar-invariant in
 \cref{thisisalabel}
 and so does not contribute to the decomposition matrix, but instead contributes a vector in the simple module ${\bf D}_{-1}(4,2)$.
Then in the row labelled by $(4,1^2)$ we see another discrepancy between the two tables: this is because
${\bf D}_{-1}(4,2)$ is a composition factor of ${\bf S}_{-1}(4,1^2)$ and so it contributes to the sum in   \cref{troll}$(iii)$.
 \end{eg}
 We are now ready to provide new upper bounds for (graded) decomposition numbers in terms of our coloured tableaux.
 \begin{thm}\label{hero}
For $\la\in \ptn n$ and $\mu \in \mathcal{R}^e_n$ and $\Bbbk$ an arbitrary field,  we have that
 $$[{\bf S}^\Bbbk_q(\la): {\bf D}^\Bbbk_q(\mu)\langle k \rangle ]  \leq   |\{  \SSTS \mid \SSTS \in  {\rm CStd}(\la,\mu) , \deg (\SSTS)=k  \}  |  $$
for $k\in \ZZ$ and in particular, $d_{\lambda,\mu}^\Bbbk \leq |{\rm CStd}(\la,\mu)   |$.    \end{thm}
 \begin{proof}
 It is immediate from \cref{ineedalabel2} that
  $$[{\bf S}^\Bbbk_q(\la): {\bf D}^\Bbbk_q(\mu)\langle k \rangle ]  \leq    [{\rm Lad}(\mu)]_t! \times |\{  \SSTS \mid \SSTS \in  {\rm CStd}(\la,\mu) , \deg (\SSTS)=k  \}    |$$
and indeed this is just   rephrasing  a classical observation due to Gordon James.
  The new observation is that by   \cref{troll}, we know that  $ [{\rm Lad}(\mu)]_t! $ divides  both
$$
\Dim
	(e(\mu)
	( {\bf S}^\Bbbk_q(\lambda))
\qquad
\Dim
	(e(\mu)
	( {\bf D}^\Bbbk_q(\mu)))
 $$
 and the result follows  by induction on the dominance ordering and  the equation in \cref{troll}$(iii)$.  In more detail,
 our base case for induction is when $\mu=\lambda$ mentioned above.  Now,   by    \cref{troll}$(iii)$
and  our inductive assumption, the result holds for all $\nu$ such that
  ${\mu\vartriangleright  \nu \vartriangleright   \la}$.
 Putting this together with   \cref{troll}$(ii)$,   we deduce that
  $[{\rm Lad}(\mu)]_t!$ divides $ \Dim (e(\mu){\bf D}^\Bbbk_q(\la))$ as required.
 \end{proof}

\begin{eg} If $p=2$ then the graded decomposition matrix of $\Bbbk\mathfrak{S}_6$ is given by the table in  \cref{thisisalabel}.
    In other words, the bounds of \cref{hero} are sharp.
 \end{eg}

 \begin{rmk}
 The inductive approach to calculating decomposition numbers of $H^\CC_q(n)$  highlighted   in  the equation in \cref{troll}$(iii)$ above
is used in the arXiv  appendix  to this paper to prove decomposability of an infinite family of Specht modules.
In \cref{2qs}
 the above algorithm will not work (as the set of 2-separated partitions is not saturated in the dominance order).  However,  we provide an analogous algorithm for calculating 2-separated decomposition numbers using ``2-dilated" coloured tableaux.
 \end{rmk}

\section{The Cherednik algebra and a simple criterion \\ for semisimplicity of a Specht module}\label{schur}

  The group $\mathfrak{S}_n  $ acts on the algebra,  $\CC\langle x_{1}, \dots, x_{n}, y_{1}, \dots y_{n}\rangle$, of polynomials in $2n$ non-commu\-ting variables.
 The  {\sf rational Cherednik algebra} $\mathscr{H}_q(\mathfrak{S}_{n})$ is a quotient of the semidirect product algebra  $\CC\langle x_{1}, \dots, x_{n}, y_{1}, \dots, y_{n}\rangle\rtimes \mathfrak{S}_{n} $ by
  commutation relations in the $x$'s and $y$'s that are similar to those of the Weyl algebra but involve an error term in $\CC\mathfrak{S}_n$ (see  \cite[Section 1]{EtingofGinzburg} for the full list of relations). In particular, these relations tell us that the $x$'s commute with each other and so do the $y$'s.
 The algebra $\mathscr{H}_q(\mathfrak{S}_{n})$ has three distinguished subalgebras: $\CC[\underline{y}] := \CC[y_{1}, \dots, y_{n}]$, $\CC[\underline{x}] := \CC[x_{1}, \dots, x_{n}]$, and the group algebra $\CC\mathfrak{S}_n$. The \emph{PBW theorem}  \cite[Theorem 1.3]{EtingofGinzburg} asserts that multiplication gives a vector space isomorphism
	\[
	\CC[\underline{x}] \otimes \CC\mathfrak{S}_n \otimes \CC[\underline{y}] \buildrel \cong \over \longrightarrow \mathscr{H}_q(\mathfrak{S}_{n})
	\]
	\noindent called the \emph{triangular decomposition} of $\mathscr{H}_q(\mathfrak{S}_{n})$, by analogy with the triangular decomposition of the universal enveloping algebra of a semisimple Lie algebra.
		We define the category ${\mathcal{O}}_q(\mathfrak{S}_n)$ to be the full subcategory consisting of all finitely generated $\mathscr{H}_q(\mathfrak{S}_n)$-modules on which $y_{1}, \dots, y_{n}$ act locally nilpotently.    The category   ${\mathcal{O}}_q(\mathfrak{S}_n)$ is a highest weight category with respect to the poset   $(\ptn  n,\rhd)$.   The standard modules are constructed as follows.
 Extend the action of $\mathfrak{S}_n$ on ${\bf S}^\CC(\lambda)$ to an action of $\CC[\underline{y}]\rtimes \mathfrak{S}_n$ by letting $y_{1}, \dots, y_{n}$ act by $0$. The algebra $\mathbb{C}[\underline{y}] \rtimes \mathfrak{S}_n$ is a subalgebra of $\mathscr{H}_q(\mathfrak{S}_n)$ and we define the Weyl modules,
	\[
	\Delta(\lambda) := {\rm Ind}_{\CC[\underline{y}]\rtimes \mathfrak{S}_n}^{\mathscr{H}_q(\mathfrak{S}_n)} {\bf S}^\CC( \lambda) :=
	\mathscr{H}_q(\mathfrak{S}_n) \otimes_{\CC[\underline{y}] \rtimes \mathfrak{S}_n}  {\bf S}^\CC(\lambda)
 = \CC[\underline{x}] \otimes {\bf S}^\CC(\lambda)
	\]
  where the last equality is \emph{only as $\CC[\underline{x}]$-modules} and follows from the triangular decomposition.
	 We let $L(\lambda)$ denote the unique irreducible quotient of $\Delta(\lambda)$.
	 In \cite[Theorem 7.4.]{RSVV} (see also \cite{Losev,Webster})  it is shown that   ${\mathcal{O}}_q(\mathfrak{S}_n)$ is standard Koszul.
    We do not recall the definition of a standard Koszul algebra here, but merely the following useful proposition.
The following proposition is proven in {\cite[Proposition 2.4.1]{MR1322847}} in the generality of all Koszul algebras.

 \begin{prop}
 \label{usefulKoszul}
For $\lambda,\mu\in\ptn n$ we have that
$$[\Delta (\lambda) : L (\mu)\langle i\rangle ]= \dim_\Bbbk \Hom_{\mathscr{H}_q(\mathfrak{S}_{n})} (\rad_i(\Delta^\CC(\lambda)),	L^\CC(\mu)).$$
 \end{prop}
\begin{proof}
    By \cite[Corollary 2.3.3]{MR1322847}, any Koszul algebra is quadratic.
Therefore, since $\Delta (\lambda)/ \rad(\Delta (\lambda))=L(\la)$ is simple and concentrated in degree zero, the radical filtration of $\Delta (\lambda)$ coincides with the grading filtration
of $\Delta (\lambda)$ by  \cite[Proposition 2.4.1]{MR1322847}.
\end{proof}

Now, there exists an exact functor (the Knizhnik--Zamolodchikov functor) relating the module categories of Cherednik and Hecke algebras,
$${\sf KZ}: {\mathcal{O}}_q(\mathfrak{S}_n) \longrightarrow H_q^{\CC}(n)  \text{-mod}.$$
A construction of this functor is given in \cite{GGOR03}, here we will only
 need the fact (from \cite[Section 6]{GGOR03}) that
\begin{equation}\label{alaebl}{\sf KZ}(\Delta(\lambda))= {\bf S}_q^\CC(\lambda),
\quad
{\sf KZ}(L(\lambda))=
\begin{cases}
 {\bf D}_q^\CC(\lambda)		&\text{if $\lambda$ is $e$-regular,}\\
 0					&\text{otherwise.}
 \end{cases}
\end{equation}
  This allows us to prove the following criterion for decomposability of Specht modules, denoted ${\bf S}^\CC_q(\lambda)$, for the Hecke algebra $H^\CC_q(n)$.

\begin{thm}\label{directsum}
Fix $\lambda \in \ptn n$.  Suppose that for all  $e$-regular partitions  $\mu$, we have that
\begin{equation}\label{assumper}[{\bf S}^\CC_q(\lambda) : {\bf D}^\CC_q(\mu)]=a_\mu t^{p(\lambda)}\end{equation}
for some fixed $p(\lambda)=z\in \mathbb{N}$ (independent of $\mu$) and some scalars $a_\mu\in \mathbb{N}$.
 It follows that
the Specht module ${\bf S}^\CC_q(\lambda) $ is semisimple.
\end{thm}

   \begin{proof}
 Throughout the proof, we let  $\lambda\in \ptn n$ be an arbitrary partition.
 For an $e$-regular partition $\mu\in \ptn n$, we have that
$$
[\Delta(\lambda) :  L(\mu)\langle k\rangle]=
[{\sf KZ}(\Delta(\lambda)) :  {\sf KZ}(L(\mu)\langle k\rangle)]=
[{\bf S}_q^\CC(\lambda) : {\bf D}_q^\CC(\mu)\langle k\rangle].
$$
by \cref{alaebl}.  Putting together \cref{usefulKoszul} and  our assumption in \cref{assumper},  we have that
$$
 [ \rad_i( \Delta(\lambda)) : L(\mu)] = 0$$
for $\mu$ any  $e$-regular   partition  and   any $i\neq z$.  Therefore
$$
{\sf KZ}(\rad_i( \Delta(\lambda)) =
\begin{cases}
  \bigoplus_\mu a_\mu {\bf D}_q^\CC(\mu) \langle z \rangle  		&\text{for }i=z, \\
  0										&\text{otherwise}.
\end{cases}$$
Therefore
$${\sf KZ}(\Delta(\lambda)) = {\sf KZ}(\rad_z(\Delta(\lambda)))=  \bigoplus_\mu a_\mu {\bf D}_q^\CC(\mu) \langle z \rangle,  $$
and the result follows.
    \end{proof}

\begin{rmk}
We have seen  the grading and radical structure of standard $\mathscr{H}_q(\mathfrak{S}_{n})$-modules are intimately related.
It is unknown as to whether or not the Schur functor preserves this property.  Thus \cref{directsum} represents all that is currently known about the relationship between
the grading and radical structure on Specht modules for  $H^\CC_q(n)$.
\end{rmk}

 \section{The Hecke algebra   and  \\   $2$-separated    partitions}\label{2qs}

Throughout this section, we shall consider the
representation theory of the Hecke algebra $H^\CC_{-1}(n)$ as a
first approximation to the $2$-modular representation theory of symmetric groups.
We shall focus on the Specht modules labelled by $2$-separated partitions.  We shall prove that these modules are semisimple and calculate their decomposition as a direct sum of graded simple modules.
 Throughout this section our character formulas will be given in terms of binomial coefficients that have been ``dilated" by a factor of 2.  We shall write
 $
\llbracket m\rrbracket_t =
[ m]_{t^2}
  $
   and extend this notation to the quantum factorials and binomials in the obvious fashion.

\begin{thm}\label{mainresultformodular}
We have that
$$
\Dim \left(
e\left({
\Lad^{\tau^{\alpha }
_
{\varnothing}
}
}\right)
{\bf S}^\CC_{-1}\left({
 {\tau^{\la }
_
{\mu}
}
}\right)\right)
=
  \left|\SStd(\lambda^T\cup \mu,\alpha^T)\right|
\times \llbracket \alpha^T\rrbracket_t ! ^2 \times t^{|\mu|}.
$$
 \end{thm}

Before embarking on the proof, we note the following immediate corollaries.

\begin{cor}
We have that
$$
e\left({
\Lad^{\tau^{\alpha }
_
{\varnothing}
}
}\right)
{\bf S}^\CC_{-1}\left({
 {\tau^{\la }
_
{\mu}
}
}\right)
=
t^{|\mu|} \sum_{\nu} c\left(\nu^T, \lambda^T,\mu\right)  e\left({
\Lad^{\tau^{\alpha }
_
{\varnothing}
}
}\right)
{\bf S}^\CC_{-1}\left({
 {\tau^{\nu }
_
{\varnothing}
}
}\right)
$$
and therefore
$$
\left[{\bf S}^\CC_{-1}\left({
 {\tau^{\la }
_
{\mu}
}
}\right) \right] = \sum_\nu t^{|\mu|}\times c\left(\nu^T, \lambda^T,\mu\right) \times \left[{\bf S}^\CC_{-1}\left({
 {\tau^{\nu }
_
{\varnothing}
}
}\right) \right].
$$
\end{cor}

Therefore the Specht modules ${\bf S}^\CC_{-1}({
 {\tau^{\nu }
_
{\varnothing}
}
})$ are simple (as their characters are bar-invariant)   and by \cref{schur} we deduce the following corollary.

\begin{cor}\label{gradedresult}
We have that,  as an $H^\CC_{-1}(n)$-module, any Specht module labelled by a 2-separated partition is semisimple and decomposes as follows
$$
 {\bf S}^\CC_{-1}\left({
 {\tau^{\la }
_
{\mu}
}
}\right)   = \bigoplus_\nu     c\left(\nu^T, \lambda^T,\mu\right)  {\bf  D}^\CC_{-1}\left({
 {\tau^{\nu }
_
{\varnothing}
}
}\right)  \left\langle {|\mu|} \right\rangle.
$$
In particular, the Specht $H^\CC_{-1}(n)$-module ${\bf S}^\CC_{-1}({ {\tau^{\la }_{\mu}}})$ is simple
if and only if $\lambda$ or $\mu$ is equal to $\varnothing$.
\end{cor}

We now turn to the proof of the main result.

\begin{proof}[Proof of \cref{mainresultformodular}]
We let $\rho:=\rho(k)$ and we assume for notational purposes  that $k$ is even; the $k$ odd case is identical except that the residues 0 and 1 must be transposed.
 We  let $w= |\alpha|= |\lambda|+ |\mu|$.    Consider the ladder tableau of the partition ${
 {\tau^{\alpha }
_
{\varnothing}
}
}$.
We have that
$$
\res ( {
\Lad^{\tau^{\alpha }
_
{\varnothing}
}
}  )=
\res(\rho) \circ
(\underbrace{0,0,\dots,0}_{\alpha_1^T\text{ times}} , \underbrace{1,1,\dots,1}_{\alpha_1^T\text{ times}}  )
\circ
(\underbrace{0,0,\dots,0}_{\alpha_2^T\text{ times}} , \underbrace{1,1,\dots,1}_{\alpha_2^T\text{ times}} )
\circ \cdots .
$$
Given $\stt \in \Std( {\tau^{\la }
_
{\mu}
}
)$, it is clear that
 $\res(\stt)=\res(\Lad^{\tau^{\alpha }_{\varnothing} }  )$
if and only if  $\stt= \Lad^\rho \circ \sts $
for some $\sts \in \Std( {\tau^{\la }
_
{\mu}
}
   \setminus \rho)$ with
$$\res(\sts)= (\underbrace{0,0,\dots,0}_{\alpha_1^T\text{ times}} , \underbrace{1,1,\dots,1}_{\alpha_1^T\text{ times}}  )
\circ
(\underbrace{0,0,\dots,0}_{\alpha_2^T\text{ times}} , \underbrace{1,1,\dots,1}_{\alpha_2^T\text{ times}} )
\circ \cdots
$$
  and we let $\Std_\alpha ( {\tau^{\la }
_
{\mu}
}
   \setminus \rho)$ denote the set of all such tableaux.
   All that  remains is to show that
\begin{equation}\label{maineqn}
\sum_{\sts\in  \Std_\alpha ( {\tau^{\la }
_
{\mu}
}
   \setminus \rho)} t^{{\rm deg}(\sts)} = \left|\SStd\left(\lambda^T\cup \mu,\alpha^T\right)\right|
\times \llbracket \alpha^T\rrbracket _t^2 \times t^{|\mu|}.
\end{equation}
  Given an integer $  j   \in \{1,2,\dots, 2w\}$
   we
  have that there exists a unique corresponding integer $a(j)\in\{1,\dots,w\}$  such that
  $$
  1+\sum_{  i=1    }^{a(j)} |\alpha^T_i | \leq j \leq
   1+\sum_{  i=1    }^{a(j)+1} |\alpha^T_i |;
  $$
  these integers will record the weight of the semistandard tableaux in the statement of  \cref{maineqn}.
Namely, we record    a skew-tableau $\sts$
    by placing both the usual entry $j\in\{1,2,\dots, 2w\}$
    but we also add a subscript  $a(j)$.  An example is depicted on the left-hand side of \cref{keyidea}.
 Recall that we can think of the partition $\tau^\la_\mu$ as being obtained by adding $(2)$-dominoes  to the right of $\rho$ and $(1^2)$-dominoes  to the bottom of $\rho$ in an intuitive fashion demonstrated in \cref{below}.
Take the partition $\rho$ and add a total of $\alpha_1^T$ nodes of residue 0; the resulting partition has
precisely $\alpha_1^T$ addable 1-nodes $X_1, \dots, X_{\alpha_1}$: namely, those which belong to the $(2)$- and $(1^2)$-dominoes  containing the nodes  $X_1, \dots, X_{\alpha_1}$.
 Repeating this observation as necessary, we deduce that any two nodes in the same domino of a tableau
${\sts\in  \Std_\alpha ( {\tau^{\la }
_
{\mu}
}
   \setminus \rho)} $   have the same subscript.
   Furthermore, we note that the fact that the residue sequence is of the form
 $$(\underbrace{0,0,\dots,0}_{\alpha_1^T\text{ times}} , \underbrace{1,1,\dots,1}_{\alpha_1^T\text{ times}}  )
\circ
(\underbrace{0,0,\dots,0}_{\alpha_2^T\text{ times}} , \underbrace{1,1,\dots,1}_{\alpha_2^T\text{ times}} )
\circ \cdots
$$ implies that no two $(2)$-dominoes of the same subscript  can be added in the same row
 and no two  $(1^2)$-dominoes  of the same subscript can be added in the same column.
   Therefore we obtain a well-defined map
   $$
\varphi:    \Std_\alpha \left( {\tau^{\la }
_
{\mu}
}
   \setminus \rho\right)  \longmapsto \SStd\left(\lambda^T \cup \mu, \alpha^T\right)
   $$
   given by scaling the sizes of all the dominoes  by $1/2$, conjugating $\la$, and recording {\em only}
   the subscripts (i.e., deleting the integers $\{1,\dots,2w\}$).  An example is depicted  in \cref{keyidea}.
  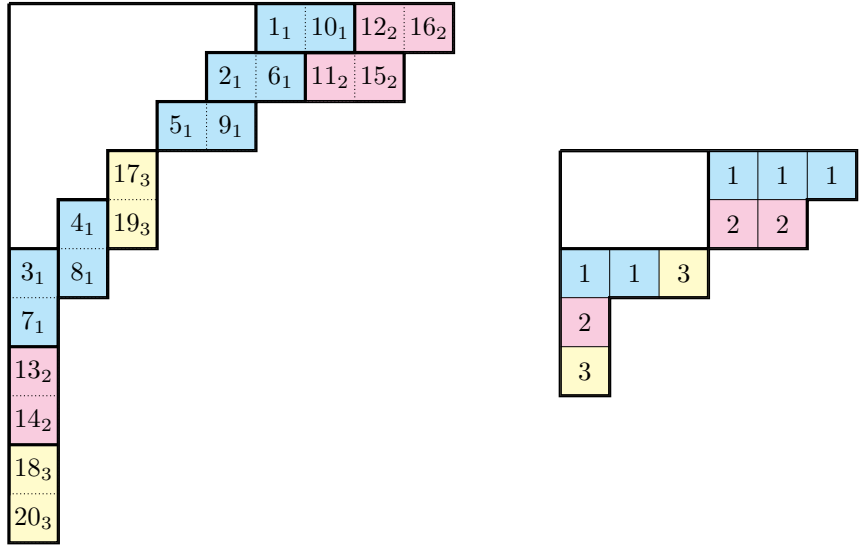
\begin{figure}[ht!]
 $$
\scalefont{0.9} \begin{minipage}{6.2cm}  \begin{tikzpicture} [scale=0.65]

 \draw[very thick](0,0)--++(0:9)--++(-90:1)--++(180:1)--++(-90:1)
 --++(180:3)--++(-90:1)
 --++(180:2)--++(-90:2)
  --++(180:1)--++(-90:1)
  --++(180:1)--++(-90:5)
   --++(180:1)--++(90:11)
 ;
  \clip(0,0)--++(0:9)--++(-90:1)--++(180:1)--++(-90:1)
 --++(180:3)--++(-90:1)
 --++(180:2)--++(-90:2)
  --++(180:1)--++(-90:1)
  --++(180:1)--++(-90:5)
   --++(180:1)--++(90:11)
 ;

 \fill[cyan!25] (0:5)--++(0:2)--++(-90:1)
 --++(180:1) --++(-90:1)
 --++(180:1) --++(-90:1)
 --++(180:2)
  --++(90:1)
 --++(0:1)  --++(90:1)
 --++(0:1)  --++(90:1)
 --++(0:1)
  ;

 \fill[magenta!25] (0:7)--++(0:2)--++(-90:1)
 --++(180:1) --++(-90:1)
 --++(180:1) --++(-90:1)
 --++(180:2)
  --++(90:1)
 --++(0:1)  --++(90:1)
 --++(0:1)  --++(90:1)
 --++(0:1)
  ;

 \fill[cyan!25] (90:-5)--++(-90:2)--++(0:1)
 --++(+90:1) --++(0:1)
  --++(+90:2)
  --++(90:1)
 --++(180:1)  --++(-90:1)
 --++(180:1)  --++(-90:1)
 --++(180:1)
  ;

 \fill[magenta!25] (90:-7)--++(-90:2)--++(0:1)
 --++(+90:2) --++(0:1)
    ;

 \fill[yellow!25] (90:-9)--++(-90:2)--++(0:1)
 --++(+90:2) --++(0:1)
    ;

 \fill[yellow!25] (2,-3)--++(-90:2)--++(0:1)
 --++(+90:2) --++(0:1)
    ;

  \path(0,0) coordinate (origin);
   \foreach \i in {1,...,19}
  {
    \path (origin)++(0:1*\i cm)  coordinate (a\i);
    \path (origin)++(-90:1*\i cm)  coordinate (b\i);
    \path (a\i)++(-90:10cm) coordinate (ca\i);
    \path (b\i)++(0:10cm) coordinate (cb\i);
    \draw[densely dotted] (a\i) -- (ca\i)  (b\i) -- (cb\i); }

 \draw[very thick](0,0)--++(0:9)--++(-90:1)--++(180:1)--++(-90:1)
 --++(180:3)--++(-90:1)
 --++(180:2)--++(-90:2)
  --++(180:1)--++(-90:1)
  --++(180:1)--++(-90:5)
   --++(180:1)--++(90:11)
 ;
 \fill[white](0,0)--++(0:5)--++(-90:1)--++(180:1)--++(-90:1)
 --++(180:1)--++(-90:1)
 --++(180:1)--++(-90:1)
  --++(180:1)--++(-90:1)  --++(180:1)--++(-90:1)
  ;

 \draw[very thick](0,0)--++(0:7)--++(-90:1)--++(180:1)--++(-90:1)
 --++(180:3)--++(-90:1)
 --++(180:1)--++(-90:1)
  --++(180:1)--++(-90:1)
  --++(180:1)--++(-90:5)
   --++(180:1)--++(90:11)
 ;

   \draw[very thick](origin)--++(-90:7)--++(0:1)--++(90:3)--++(0:1)--++(-90:2);
      \draw[very thick](origin)--++(-90:9)--++(0:1);

 \draw[very thick](0,0)--++(0:5)--++(-90:1)--++(0:3)--++(180:4) --++(-90:1) ;

 \path(0.5,-0.5) coordinate (origin);
 \foreach \i in {1,...,19}
  {
    \path (origin)++(0:1*\i cm)  coordinate (a\i);
    \path (origin)++(-90:1*\i cm)  coordinate (b\i);
    \path (a\i)++(-90:1cm) coordinate (ca\i);
        \path (ca\i)++(-90:1cm) coordinate (cca\i);
    \path (b\i)++(0:1cm) coordinate (cb\i);
    \path (cb\i)++(0:1cm) coordinate (ccb\i);
  }
\draw(a5) node {$1_1$};
\draw(a6) node {$10_1$};
 \draw(a7) node {$12_2$};
  \draw(a8) node {$16_2$};

\draw(ca4) node {$2_1$};
\draw(ca5) node {$6_1$};
 \draw(ca6) node {$11_2$};
  \draw(ca7) node {$15_2$};

\draw(cca3) node {$5_1$};
\draw(cca4) node {$9_1$};

\draw(b5) node {$3_1$};
\draw(b6) node {$7_1$};
 \draw(b7) node {$13_2$};
  \draw(b8) node {$14_2$};
 \draw(b9) node {$18_3$};
  \draw(b10) node {$20_3$};

\draw(cb4) node {$4_1$};
\draw(cb5) node {$8_1$};

\draw(ccb3) node {$17_3$};
\draw(ccb4) node {$19_3$};

   \end{tikzpicture} \end{minipage}
   \quad   \quad   \quad
 \begin{minipage}{4cm}  \begin{tikzpicture} [scale=0.65]

  \draw[very thick](0,0)--++(0:6)--++(-90:1)--++(180:1)--++(-90:1)
 --++(180:2)--++(-90:1)
 --++(180:2)--++(-90:2)
  --++(180:1)--++(90:5) ;
  \clip(0,0)--++(0:6)--++(-90:1)--++(180:1)--++(-90:1)
 --++(180:2)--++(-90:1)
 --++(180:2)--++(-90:2)
  --++(180:1)--++(90:5)
 ;

 \fill[cyan!25] (3,0)--++(-90:1)--++(0:3)
 --++(+90:1) --++(180:3)
    ;

 \fill[cyan!25] (0,-2)--++(-90:1)--++(0:2)
 --++(+90:1) --++(180:2)
    ;

     \fill[yellow!25] (2,-2)--++(-90:1)--++(0:2)
 --++(+90:1) --++(180:2)
    ;

 \fill[yellow!25] (0,-4)--++(-90:1)--++(0:2)
 --++(+90:1) --++(180:2)
    ;

     \fill[magenta!25] (0,-3)--++(-90:1)--++(0:2)
 --++(+90:1) --++(180:2)
    ;
     \fill[magenta!25] (3,-1)--++(-90:1)--++(0:2)
 --++(+90:1) --++(180:2)
    ;

  \path(0,0) coordinate (origin);
   \foreach \i in {1,...,19}
  {
    \path (origin)++(0:1*\i cm)  coordinate (a\i);
    \path (origin)++(-90:1*\i cm)  coordinate (b\i);
    \path (a\i)++(-90:10cm) coordinate (ca\i);
    \path (b\i)++(0:10cm) coordinate (cb\i);
    \draw[thin ] (a\i) -- (ca\i)  (b\i) -- (cb\i); }

   \fill[white](0,0)--++(0:3)--++(-90:2)--++(180:3)--++(90:2)
   ;

 \draw[very thick](0,0)--++(0:6)--++(-90:1)--++(180:1)--++(-90:1)
 --++(180:2)--++(-90:1)
 --++(180:2)--++(-90:2)
  --++(180:1)--++(90:5)
  ;
  \draw[very thick](0,0)--++(0:3)--++(-90:2)--++(180:3)--++(90:2);
   \path(0.5,-0.5) coordinate (origin);
 \foreach \i in {1,...,19}
  {
    \path (origin)++(0:1*\i cm)  coordinate (a\i);
    \path (origin)++(-90:1*\i cm)  coordinate (b\i);
    \path (a\i)++(-90:1cm) coordinate (ca\i);
        \path (ca\i)++(-90:1cm) coordinate (cca\i);
    \path (b\i)++(0:1cm) coordinate (cb\i);
    \path (cb\i)++(0:1cm) coordinate (ccb\i);
  }
\draw(a3) node {$1 $};
\draw(a4) node {$ 1$};
 \draw(a5) node {$1$};

\draw(ca3) node {$2 $};
\draw(ca4) node {$2$};

\draw(b2) node {$1$};
\draw(b3) node {$2$};
 \draw(b4) node {$3$};

\draw(cb2) node {$1$};
\draw(ccb2) node {$3$};

%\draw(ccb3) node {$17_1$};
%\draw(ccb4) node {$19_1$};

   \end{tikzpicture} \end{minipage}
 $$
 \caption{A $(3^2,2,1^2)$-decorated standard tableau of shape $\tau^{(2,2,1)}_{(3,1,1)}
\setminus \rho$
and the corresponding element of
$\SStd((3,2)\cup (3,1,1) , (5,3,2))$.  The associated sequence of partitions (as in \cref{sequence}) is
$\tau\subset \tau^{(1^3)}_{(2)}\subset \tau^{(2^2,1)}_{(2,1)}\subset \tau^{(2^2,1)}_{(3,1^2)}$.}
 \label{keyidea}
 \end{figure}

 All that remains to show is that
\begin{equation}\label{gradedequationtoprove}
 \sum_{\{\sts \mid \varphi(\sts)= \SSTS \}} t^{{\rm deg}(\sts)}=
 \llbracket  \alpha^T\rrbracket _t!^2  \times t^{|\mu|}
\end{equation}
for any $\SSTS  \in \SStd(\lambda^T\cup \mu,\alpha^T)$.
The set ${\{\sts \mid \varphi(\sts)= \SSTS \}} $ consists of an orbit
$$\mathfrak{S}_{\alpha_1}\times \mathfrak{S}_{\alpha_1}
\times
\mathfrak{S}_{\alpha_2}\times \mathfrak{S}_{\alpha_2}
\times \dots$$
of standard tableaux. In other words,  $\sts,\stt$ are such
$\varphi(\sts)= \varphi(\stt)$ if and only if they differ by permuting nodes whose subscripts  and residues are both matching.    Therefore, we have that
$$
 \sum_{\{\sts \mid \varphi(\sts)= \SSTS \}} 1^{{\rm deg}(\sts)}
 =\left(\alpha^T_1!\alpha^T_2!\dots\right)^2= \left(
 \llbracket \alpha^T\rrbracket _t!^2 \right)|_{t=1},   $$
and so the ungraded version of \cref{gradedequationtoprove} follows.

It remains to consider the grading.  We  first cut the  diagram of any 2-separated partition  $\tau^\la_\mu$ into four regions by  drawing
 a vertical line immediately after the  $ \mu_1$th column of $\tau^\la_\mu$  and a horizontal line immediately below the    $\la_1^T$th row.  An example is depicted in \cref{sfdjakhfasklfhalkfskdfhaslkjfdha}.
 We label the three of the four  quarters of  the diagram  $X:=X(\tau^\la_\mu)$, $Y:=Y(\tau^\la_\mu)$, and $Z:=Z(\tau^\la_\mu)$ as suggested in  \cref{sfdjakhfasklfhalkfskdfhaslkjfdha}.
The intersection of $\tau^\la_\mu$ with the region $Y$ is equal to the staircase partition of width $\rho_1 -\mu_1-\la_1^T$; we set  $\rho_Y:= [\tau^\la_\mu] \cap Y$.

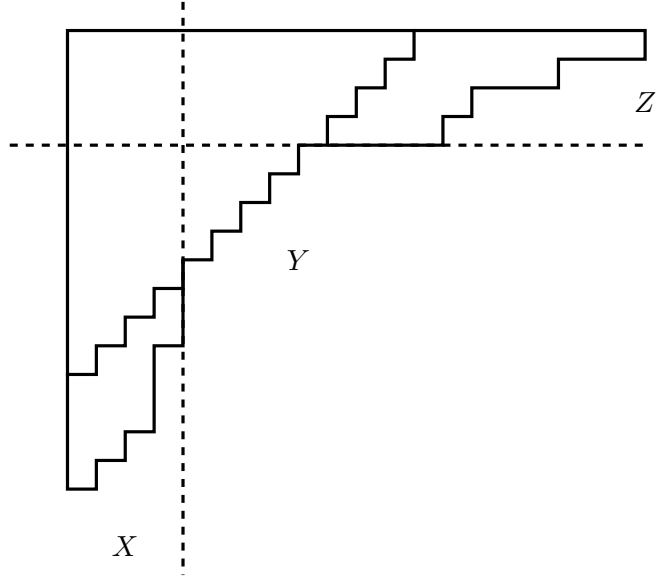
\begin{figure}[ht!]
$$
 \begin{tikzpicture} [scale=0.38]

 \draw[very thick](-2,0)--++(90:1)--++(0:12)--++(-90:1)--++(180:1)--++(-90:1)--++(180:1) --++(-90:1)--++(180:1) --++(-90:1)--++(180:1) --++(-90:1)--++(180:1) --++(-90:1)--++(180:1) --++(-90:1)--++(180:1) --++(-90:1)--++(180:1) --++(-90:1)--++(180:1) --++(-90:1)--++(180:1) --++(-90:1)--++(180:1) --++(-90:1)--++(180:1)  --(-2,0);

  ;

   \draw[very thick](10,1)--(18,1)--(18,0)--(15,0)--++(-90:1)--++(180:3)--++(-90:1)--++(180:1)--++(-90:1)--++(180:4);

     \draw[very thick](-2,-11) --++(-90:4)--++(0:1) --++(90:1) --++(0:1) --++(90:1) --++(0:1) --++(90:3)--++(0:1) --++(90:3) ;

      \draw[very thick, dashed] (2,2)--(2,-18);
            \draw[very thick, dashed] (-4,-3)--(18,-3);
      \draw(0,-17) node {$X$};
            \draw(18,-1.5) node {$Z$};
            \draw(6,-7) node {$Y$};

   \end{tikzpicture} $$

   \caption{Dividing the partition $\tau^\la_\mu$ into regions $X:=X(\tau^\la_\mu)$, $Y:=Y(\tau^\la_\mu)$ and $Z:=Z(\tau^\la_\mu)$.  In this case $\rho_Y$ is the copy of the partition $(4,3,2,1)$ in region $Y$.  }
   \label{sfdjakhfasklfhalkfskdfhaslkjfdha}
   \end{figure}

We shall calculate the lefthand-side of \cref{gradedequationtoprove} by peeling off a row of $\alpha$ at a time, in a manner which  we now make precise.
Fix $\alpha \in \ptn n$ a partition  and set  $\ell = \ell(\alpha^T)$ and   $\SSTS\in \SStd(\lambda^T\cup\mu, \alpha^T) $. We define
$$
\mu^{(i)}_j = |
\{(j,c)\in \mu\mid \SSTS(j,c)\leq  i\}
\}
|
\qquad
\lambda^{(i)}_j = |
\{(r,j)\in \la\mid \SSTS(r,j)\leq  i\}|
$$
for $1\leq i \leq \ell(\alpha)$.  We set $\mu {(i)}=(\mu^{(i)}_1,\dots,\mu^{(i)}_\ell)$
and $\la {(i)}=(\la^{(i)}_1,\dots,\la^{(i)}_\ell)$
We consider the associated sequence of partitions
\begin{equation}\label{sequence}
\tau =\tau^{\la {(0)}}_{\mu {(0)} }\subseteq
	\tau^{\la {(1)}}_{\mu {(1)} }
	 \subseteq
	\tau^{\la {(2)}}_{\mu {(2)} }
	\dots  \subseteq
	\tau^{\la {(\ell)}}_{\mu {(\ell)} }=\tau ^\la_\mu ,
\end{equation}
 an example is given in \cref{keyidea}.
Setting $\alpha^T=(a_1,\dots,a_\ell)$, we will show that
 $$
 \sum_{\sts \in \Std_{(a_m)}(\tau^{\lambda(m)}_{\mu(m)}\setminus  \tau^{\lambda(m-1)}_{\mu(m-1)})} t^{{\rm deg}(\sts)} = \llbracket a_m \rrbracket_t! ^2 \times t^{|\mu(m)|-|\mu(m-1)|},
 $$
 and hence deduce the result.  Clearly we can calculate the degree contribution of the node $(r,c)$   to   the tableau $${\sts \in \Std_{(a_m)}(\tau^{\lambda(m)}_{\mu(m)}\setminus  \tau^{\lambda(m-1)}_{\mu(m-1)})} $$ by considering the  addable/removable nodes above $(r,c)$ from the regions $$X_m:=X(\tau^{\la(m)}_{\mu(m)})\quad  Y_m:=Y(\tau^{\la(m)}_{\mu(m)})\quad Z_m:=Z(\tau^{\la(m)}_{\mu(m)})$$  separately.   We let $\deg_{X_m}(r,c)$, $\deg_{Y_m}(r,c)$ and $\deg_{Z_m}(r,c)$ denote these respective contributions.
We first consider the contribution from $Y_m$.
By definition, the sum total $$\deg_{Y_m}(\sts)= \sum_{1\leq k \leq 2m}\deg_{Y_m}\sts^{-1}(k)$$ is independent of the tableau $\sts\in \Std(\tau^{\lambda(m)}_{\mu(m)}\setminus  \tau^{\lambda(m-1)}_{\mu(m-1)})$.
 For $\rho_{Y_m}$ a staircase partition of $p(p+1)/2$, we have that
$$|{\rm Rem}_0(\rho_{Y_m})|=0\quad |{\rm Rem}_1(\rho_{Y_m})|=p
\quad |{\rm Add}_0(\rho_{Y_m})|=p+1\quad
|{\rm Add}_1(\rho_{Y_m})|
=0.$$
 Therefore $\deg_{Y_m}(\sts)$ is equal to the number of $(1^2)$-bricks   in region $X_m \cap \tau^{\la(m)}_{\mu(m)}\setminus \tau^{\la(m-1)}_{\mu(m-1)}$.  In other words    $\deg_{Y_m}(\sts)=|\mu|-|\mu'|$.
Finally, it remains to prove that
\begin{equation}\label{final} \sum_{\sts\in \Std_{(a_m)}(\tau^{\lambda(m)}_{\mu(m)}\setminus  \tau^{\lambda(m-1)}_{\mu(m-1)})}
t^{\deg_{X_m \cup Z_m}(\sts)} =\llbracket a_m\rrbracket_t! ^2  \end{equation}
which we shall do by induction.  We set $\Std_m :=  \Std (\rho(a_m) \setminus \rho(a_m-1))$.  It is easy to see that
$$  \sum_{\sts\in  \Std_{(a_m)}(\tau^{\lambda(m)}_{\mu(m)}\setminus  \tau^{\lambda(m-1)}_{\mu(m-1)})}
\!\!\!\!\!\!\!t^{\deg_{X_m \cup Z_m}(\sts)}  =
\Bigg(\sum_{\sts\in \Std_m}t^{\deg(\sts)} \Bigg)\times \Bigg(\sum_{\sts\in \Std_m}
t^{\deg(\sts)} \Bigg)  $$
 where the first/second multiplicand on the righthand-side counts the contribution of all the residue 0-boxes/1-boxes respectively.
  We set $\square_j=(j ,a_m+1-j)\in \rho(a_m) $  for $1\leq j \leq a_m$.
   We have that
\begin{align*}
  \sum_{\begin{subarray}c
 1\leq j \leq a_m \\ \{\sts \in \Std_m\mid \sts^{-1}(a_m)= \square_j\}
 \end{subarray} }\!\!\!\!\!\!t^{\deg(\sts )}	 =
   \sum_{\begin{subarray}c \stt  \in \Std_{m-1} \\ 1\leq j \leq a_m \end{subarray} }t^{a_m-j} \times t^{1-j}\times  t^{\deg(\stt  )}
  = \sum_{1\leq j \leq a_m}   t^{a_m+1-2j} \llbracket a_m-1\rrbracket_t !
 \end{align*}
 which is equal to $\llbracket a_m \rrbracket_t ! $.
Here the first equality follows by considering both the degree
of the node $\sts^{-1}(a_m)$ (which is equal to $1-j$) and the resulting shift to the degrees of each of the  $(a_m-j)$ nodes
below $\sts^{-1}(a_m)$. The second equality holds
 by induction.  Therefore \cref{final} holds and the result follows.     \end{proof}

 \section{Kronecker coefficients and Saxl's conjecture }\label{kronsax}
Let $\lambda,\mu, \nu$ be partitions of $n$.
We define the Kronecker coefficients $g(\lambda,\mu,\nu)$
to be the coefficients in the expansion
$${\bf D}^\CC(\lambda)\otimes  {\bf D}^\CC(\mu) = \bigoplus_{\nu \vdash n} g(\lambda,\mu,\nu) \, {\bf D}^\CC(\nu)\:.$$
We now recall  Saxl's conjecture concerning the positivity of these coefficients.
We let $\chi^{\lambda}$ denote the complex irreducible $\mathfrak{S}_n$-character
to the partition $\lambda$ of $n$, i.e., the character of the Specht module ${\bf S}^\CC(\lambda)$.

\begin{saxl}
Let  $n = k(k+1)/2$ and $\rho=(k,k-1,\dots,2,1)$.
For all $\lambda \vdash n$, the multiplicity of
   $\chi^\lambda$   in the Kronecker product   $\chi^\rho\cdot \chi^\rho$    is  strictly positive.
\end{saxl}

In
\cite{MR3056296},
Heide, Saxl, Tiep and Zalesski  verified that
for almost all finite simple groups of Lie type the square of the Steinberg character
contains all irreducible characters as constituents.
They also conjectured that for all alternating groups there is some irreducible character with this property.
For symmetric groups $\mathfrak{S}_n$
to triangular numbers $n$, Saxl
then suggested the candidate $\chi^\rho$ as stated above.
Saxl's conjecture has been attacked by algebraists and complexity theorists
using a variety of combinatorial and probabilistic methods \cite{MR3856528,I15,LS,PP16}.
From our perspective, a particularly useful result is the following.

\begin{thm}[{\cite[Theorem 2.1]{I15}}]\label{iken}
Let  $n = k(k+1)/2$ and $\rho=(k,k-1,\dots,2,1)$.  If $\lambda$ is a partition of $n$ such that $\lambda \trianglerighteq \rho$ or $\lambda \trianglelefteq \rho$, then $g(\rho,\rho,\la) >0$.
\end{thm}

We let $\Bbbk$ be a field of characteristic 2.
We keep the notation $\rho=\rho(k)$ and $n=k(k+1)/2$, and we note that
${\bf D}^\Bbbk(\rho)={\bf S}^\Bbbk(\rho)={\bf P}^\Bbbk(\rho)$ is a simple projective
$\Bbbk\mathfrak{S}_n$-module.  Therefore
its tensor square is also projective and decomposes as a direct sum of indecomposable projective modules labelled by 2-regular partitions; we let  $ G(\rho,\rho,\nu) $ denote the corresponding coefficients as follows
$${\bf D}^\Bbbk(\rho)\otimes  {\bf D}^\Bbbk(\rho) =
\bigoplus_{\nu \vdash_2 n} G(\rho,\rho,\nu) \, {\bf P}^\Bbbk(\nu)\:.$$
Equivalently, on the level of complex characters we have for the irreducible character $\chi^{\rho}$ the decomposition of its square into characters $\xi^\nu$ to projective indecomposable modules (i.e., to integral lifts of the projective modules at characteristic~2):
$$(\chi^\rho)^2 =
\bigoplus_{\nu \vdash_2 n} G(\rho,\rho,\nu) \, \xi^\nu \:.$$

We wish to pass information back and forth between the 2-modular coefficients $G(\lambda,\mu,\nu)$ and the Kronecker coefficients $g(\lambda,\mu,\nu)$
in order to make headway on Saxl's conjecture.
The following observation is immediate
\begin{equation}\label{easypeasy}
g(\rho,\rho,\la )= \sum_\nu G(\rho,\rho,\nu)d_{\lambda,\nu}.
\end{equation}

  We first want to explain how to apply this to obtain positivity for new classes
 of Kronecker coefficients.
Recall from \cref{facts2} that the 2-blocks  of $\mathfrak{S}_n$
 are parameterized by
 the common 2-core of the partitions labelling the simple modules in characteristic~2
 and the Specht modules in the block,
 together with the weight.
 Further recall from \cref{facts} that the decomposition matrix for each block is unitriangular
 with respect to %any linear order compatible with
   the dominance ordering.
 In particular, the most dominant partition, $\tau^{(w)}_\varnothing$,
 in a given 2-block of weight $w$
 labels a Specht module with simple reduction mod~2.
 Hence, applying Ikenmeyer's result we deduce that
 any non-zero entry of the first column (labelled by $\tau^{(w)}_\varnothing$)  of
 the 2-decomposition matrix of any 2-block  corresponds to a non-zero Kronecker coefficient.
This allows us to verify Kronecker positivity in Saxl's conjecture for two new  infinite families of partitions:

\begin{thm}\label{height0}
Let $n=k(k+1)/2$, $\rho=\rho(k)$ and $\lambda \vdash n$ such that $\chi^\lambda$
is of height~0.
Then $g(\rho,\rho,\lambda)>0$.
In particular, all $\chi^\lambda$ of odd degree are constituents of
the Saxl square.
 \end{thm}
\begin{proof}
Let $B$ be the 2-block of $\mathfrak{S}_n$ to which $\chi^\lambda$ belongs.
Because $\chi^\lambda$ is a character of height~0,
the modulo 2 reduction
 ${\bf S}^\Bbbk(\la)$ must
 have a composition factor of height zero  and this composition factor must  appear with odd multiplicity (simply by comparing the dimensions, see \cref{subsec:height0}).
 By \cref{KOW},  the 2-block $B$ contains a unique simple module
$D={\bf D}^\Bbbk(\mu)$  of height~0,
with  $\mu=\tau^{(w)}_\varnothing=\tau +(2w)$ the most dominant partition
belonging to the block.
%,
%hence ${\bf S}^\Bbbk(\mu) ={\bf D}^\Bbbk(\mu)$.
 The discussion preceding the theorem now implies
$$g(\rho,\rho,\lambda) \ge d_{ \lambda , \tau+(2w)} >0.\qedhere$$
\end{proof}

Since we get the large number $k_0(B)$ of height~0  irreducible characters
for each 2-block $B$, this  constitutes quite a large class of
 constituents in the Saxl square.

  \begin{eg}
  We have already seen (in \cref{36}) that the  block $B_{3}(36)$ contains 1024 height~0 characters,
  which all appear  in Saxl's tensor square.
\end{eg}

\begin{eg}
We consider the partition $\lambda=(9,8,3,2,2,2,1,1)\vdash 28$.
 This is the first of the  four partitions pictured in \cref{2adicpic2836}.
In this case,
  the desired positivity of $g(\rho(7),\rho(7),\lambda)$
cannot be deduced using
  the available non-vanishing criteria in  the literature  \cite{MR3856528,PP16},
  and $\lambda$ is incomparable to $\rho$ in the dominance order, so \cite{I15}
  does not apply.
 The character $\chi^\lambda$ belongs to the 2-block of weight $w=11$ and
2-core $\rho(3)$, and it is of height~0.
Instead of computing the degree explicitly, this can also be seen by
applying one of the combinatorial descriptions for labels of
height~0 characters, e.g., the one due to \cite{MR3829558}
recalled in \cref{subsec:height0}.
  Finally, referring forward in this paper: we remark that  $(9,8,3,2,2,2,1,1)$ is  not 2-separated.  Thus
\cref{height0} provides us with constituents of the Saxl  square that cannot be deduced using any other results in the literature.
\end{eg}

For the second new family;
we will also need to apply our new results
on Specht modules for the Hecke algebra.

 \begin{thm}\label{stairacse}
For  $\tau^{(m)}_{(1^\ell)}$   any framed staircase partition of $n=k(k+1)/2$,
we have that
$g(\rho,\rho,\tau^{(m)}_{(1^\ell)})>0$.
 \end{thm}
 \begin{proof}
Let $w=\ell+m$ be the weight of the 2-block $B$ to which $\tau^{(m)}_{(1^\ell)}$ belongs.
We have that $\tau^{(w)}_\varnothing$ is the most dominant partition in $B$ and so the corresponding Specht module is simple.  Now
$$[ {\bf S}^\Bbbk_{-1}(\tau^{(m)}_{(1^\ell)})
 : {\bf D}^\Bbbk_{-1}(\tau^{(w)}_\varnothing )]
 \geq
 [{\bf S}^\CC_{-1}(\tau^{(m)}_{(1^\ell)})
 : {\bf D}^\CC_{-1}(\tau^{(w)}_\varnothing )]
 = c((1^w),(1^m),(1^{\ell}))=1>0$$
 and so the result follows by the discussion above.
  \end{proof}

For reasons that will soon become apparent, we now recall Carter's criterion explicitly.

\begin{thm}[{\cite{MR1687552}}]\label{JamMAt}
We let $\Bbbk$ be a field of characteristic 2.
Let $\lambda=(\lambda_1,\lambda_2,\dots,\lambda_\ell)$ be a partition.  Then the Specht module
 ${\bf S}^\Bbbk (\lambda)$
is simple if and only if one of the following conditions holds:
\begin{itemize}
\item[$(i)$] $\la_i-\la_{i+1}\equiv -1$ modulo $2^{\ell_2(\la_{i+1}-\la_{i+2}})$ for all $i\geq 1$;
\item[$(ii)$] the transpose partition, $\lambda^T$, satisfies $(i)$;
\item[$(iii)$] $\lambda=(2,2)$,
\end{itemize}
where here $\ell_2(k)$ is the least non-negative integer such that $k < 2^{\ell_2(k)}$.
We say that any partition as in  $(i)$ satisfies Carter's criterion.
\end{thm}

 \begin{eg}
 The most dominant partition, $\tau^{(w)}_{\varnothing}$, in a 2-block of weight $w$ satisfies Carter's criterion.
 \end{eg} \begin{eg}
 In a 2-block of weight $w=m(m+1)/2$, we
 find the partition
  $\tau^{\rho({m})}_{\varnothing} $
 that satisfies Carter's criterion.
 \end{eg}

If $\lambda$ is a 2-regular partition, then all the rows of $\lambda $ are of distinct length.
It immediately follows that $\lambda\trianglerighteq \rho$ and therefore  $g(\rho,\rho,\lambda) >0$.
 If furthermore the partition $\lambda$ satisfies Carter's criterion, then by \cref{labelll} we have that
 ${\bf P}^\Bbbk (\lambda)$ is the unique projective module in which   ${\bf S}^\Bbbk (\lambda)$ appears as a composition factor of a Specht filtration.  Putting these two statements together (in light of   \cref{easypeasy}) we obtain the following.

 \begin{prop}
 Let $n=k(k+1)/2$ and $\alpha \vdash n$.
 If $\alpha$  satisfies Carter's criterion, then we have that
 $G(\rho(k),\rho(k),\alpha)>0$.
 \end{prop}

 The following result is immediate by  \cref{easypeasy}. It is the key to all of our results on Kronecker positivity (as it relates this problem to that of determining the positivity of modular decomposition numbers) and vastly generalises \cref{stairacse}.

 \begin{thm}\label{ineedalabel}
 Let $n=k(k+1)/2$ and $\alpha \vdash n$.    If there exists some $\beta$ satisfying  Carter's criterion such that $d_{\alpha ,\beta}=m>0$, then
 $g(\rho(k),\rho(k),\alpha)\geq m$.
 We refer to such a pair $(\alpha,\beta)$ as an {\sf $m$-Carter--Saxl pair}.
 \end{thm}

We are now ready to use the results of \cref{schur,2qs} toward the Kronecker problem.

   \begin{thm}\label{app1}
   Let $w=k(k+1)/2$, $n=w(2w+1)$ and $\tau =  \rho(2w-1)$.
   Then for $\la,\mu$ any pair such that  $c(\rho(k),\la,\mu^T)>0$
   we have that
   $$g(\rho{(2w)},\rho({2w}),\tau^\la_\mu)\geq c(\rho(k),\la,\mu^T)> 0.$$
      \end{thm}

 \begin{proof}
Clearly, $\rho(2w)$ is a partition of $n=w(2w+1)$.
 We restrict our attention to the block $B$
of weight $w$ in $\Bbbk\mathfrak{S}_n$ with 2-core $\tau =  \rho(2w-1)$.
For $\nu=\rho(k)$, the partition  $\tau^{\nu}_\varnothing$  belongs to this block, and it  satisfies Carter's criterion.
Note that $c(\rho(k),\la,\mu^T)>0$ implies that $|\lambda|+|\mu|=w$, and
$\tau^{\lambda}_{\mu}$ also belongs to $B$.
  Finally,  we have
   that $$[
 {\bf S}^\Bbbk_{-1}(\tau^{\lambda}_{\mu} )
 : {\bf D}^\Bbbk_{-1}(\tau^{\nu}_\varnothing )]
 \geq
 [{\bf S}^\CC_{-1}(\tau^{\lambda}_{\mu} )
 : {\bf D}^\CC_{-1}(\tau^{\nu}_\varnothing )]
 = c(\rho(k),\lambda,\mu^T) >0, $$
 and the result follows from \cref{ineedalabel}.
  \end{proof}

We remark that none of the partitions above  are covered by existing results in the literature.
It is clear that they are not hooks or double-hooks and providing that neither $\la$ or $\mu$ is the empty partition, then these partitions are not comparable with $\rho(2w)$ in the dominance order (as $\tau^\la_\mu$ is both wider and longer that $\rho(2w)$).

An explicit new infinite family with unbounded Kronecker coefficients is given in the following corollary.

  \begin{cor}
 For
   $w=k(k+1)/2$, $n=w(2w+1)$ and $\tau=\rho(2w-1)$, we
 have that
  $$g(\rho({2w}),\rho({2w}),\tau^{\rho({k-1})}_{(k-1,1)})\geq k-1.  $$
      \end{cor}
 \begin{proof}
We have
$$ c(\rho(k),\rho(k-1),(k-1,1)^T)=k-1, $$
from which
the result follows from \cref{app1}.
  \end{proof}

\begin{eg}
In particular, for $w=k(k+1)/2$,
  $g(\rho{(2w)},\rho({2w}),\tau^\la_\mu)> 0$
for $\la,\mu$ any pair such that $\la+\mu^T=\rho(k)$.
Examples of such partitions $\tau^\la_\mu$ are pictured in \cref{afiguregofigure}.
\end{eg}

\begin{eg}
Let $n=210$ and consider the 2-block of weight 20 with 2-core $\rho({19})$.   There exist 35 Carter--Saxl pairs belonging to pairs $(\la,\mu)$ such that $\la+\mu^T$ is equal to either $(10)$ or $(4,3,2,1)$.  There are many more Carter--Saxl pairs in this block.
\end{eg}

 Finally, we conclude this section by remarking that we have only used positivity of decomposition numbers for the Hecke algebra over $\mathbb C$.
   These are the easiest decomposition numbers to calculate, but only provide lower bounds for  the decomposition numbers of symmetric groups.

\section{More semisimple decomposable Specht modules}\label{more}

Semisimplicity and   decomposability of  Specht modules
 has long been a  subject   of major interest:  the  highlight being the  recent
progress on  the
 classification of {\em simple} Specht modules for symmetric groups and their Hecke algebras \cite{MR3011340,James,MR2207757,MR1687552,MR2137291,MR1402572,MR2339470,MR1425323,MR2089249,MR2488560,MR2588144}.
 Progress towards understanding the wider family of {\em semi}-simple  and decomposable modules has been snail-like in comparison
 \cite{MR591250,MR3250450,MR2081930,MR2905253,MR2422270,MR3533560}
  and reserved solely   to  near-hook partitions.
   All examples of decomposable Specht modules discovered to date have been labelled by 2-separated partitions.
   Our Theorem A   proves that for any 2-separated partition, the corresponding Specht module  for the algebra  $H^\CC_{-1}(n)$  is decomposable.
    It is natural to ask whether the converse is true: {\em are all decomposable Specht modules for $H^\CC_{-1}(n)$ and more generally $H^\Bbbk_{-1}(n)$ indexed by 2-separated partitions?}.  In  \cite[Section 8.2]{MR2905253},  Dodge and Fayers asked exactly this question for the symmetric group with $\operatorname{char}\Bbbk=2$.
   In this section, we provide  counterexamples to this question for the Hecke algebra.

\subsection{Two new infinite families of decomposable Specht modules}\label{2infinite}

 Given $k\in \mathbb{N}$ and $l\in 2\mathbb{N}+1$, we define $\alpha_k \vdash (k +2)^2-4$ and $\beta_l \vdash (l +2)^2-2$, respectively, to be the partitions
 $$\alpha_k =((k+2)^k,k^2),\qquad
 \beta_l = (l+3,(l+2)^{l-1},l^2,1).
 $$
     For example, the Young diagrams of the partitions $\alpha_5$ and $\beta_5$, along with their residues, are drawn as follows:
\[
[\alpha_5]=\gyoung(;0;1;0;1;0;1;0,;1;0;1;0;1;0;1,;0;1;0;1;0;1;0,;1;0;1;0;1;0;1,;0;1;0;1;0;1;0,;1;0;1;0;1,;0;1;0;1;0)
\qquad
[\beta_5]=\gyoung(;0;1;0;1;0;1;0!\dgr1,!\wh1;0;1;0;1;0;1,;0;1;0;1;0;1;0,;1;0;1;0;1;0;1,;0;1;0;1;0;1;0,;1;0;1;0;1,;0;1;0;1;0,!\dgr1)
\]
In the arXiv appendix to this paper, we prove that
$$
{\bf S}^\CC_{-1}(\alpha_{[k]})
\quad\text{and}\quad
 {\bf S}^\CC_{-1}(\beta_{[l]})
$$
   are decomposable for all $k\geq 1$ and odd $l \ge 1$.
    In fact, we show that both Specht modules have a direct summand equal to a (different) simple Specht module.  Namely,
 $${\bf S}^\CC_{-1} ((k+4)^k)\cong {\bf D}_{-1}^\CC {\left(2k+3,2k+1,2k-1,\dots,9,7,5\right)} \text{ is a direct summand of } {\bf S}^\CC_{-1}(\alpha_{[k]}).$$
 These provide the first examples of decomposable Specht modules indexed by partitions which are {\bf not} 2-separated.    The proof  of decomposability is not difficult, but it does involve twenty pages of extensive calculations.   The basic idea is to (1) show that
$$
\Hom_{H^\CC_{-1}(n)}({\bf S}^\CC_{-1}((k+4)^k),{\bf S}^\CC_{-1}(\alpha_{[k]}))\neq
 0
$$
using results  on semistandard homomorphisms and $(2)$ prove that
$$
[{\bf S}^\CC_{-1}(\alpha_{[k]}):{\bf D}_{-1}^\CC {\left(2k+3,2k+1,2k-1,\dots,9,7,5\right)} ]\leq 1
$$
by counting   corresponding  coloured tableaux.
One hence deduces that this simple composition factor occurs exactly once as a composition factor but  in both  the head and the socle of ${\bf S}^\CC_{-1}(\alpha_{[k]})$, and thus is a direct summand.  We refer the reader to the arXiv appendix for more details.
\begin{conj}\label{conj:exceptional}

For $k\in\mathbb{N}$,  we set
$\alpha_k^C=(2k+3,2k+1,2k-1,\dots,9,7,5)$.
Then
we expect that
 $$
 {\bf S}_{-1}^\CC \left(\alpha_k\right) =
 {\bf D}_{-1}^\CC \left(\alpha_k^R\right)\left\langle w\left(\alpha_k\right)/2 \right\rangle \oplus  {\bf D}_{-1}^\CC \left(\alpha_k^C\right)\left\langle w\left(\alpha_k\right)/2 \right\rangle.
 $$
 \end{conj}

By $1$-induction, we conjecture the direct sum decomposition of ${\bf S}^\CC_{-1}(\beta_{[l]})$.

\begin{conj}
 For $l\in 2\mathbb{N}+1$, we set
 $\beta_l^C=(2l+3,2l+1,2l-1,\dots,9,7,6,1)$.
Then
we expect that
$$
{\bf S}_{-1}^\CC \left(\beta_l\right) =
{\bf D}_{-1}^\CC \left(\beta_l^R\right)\left\langle w\left(\beta_l\right)/2 \right\rangle \oplus  {\bf D}_{-1}^\CC \left(\beta_l^C\right)\left\langle w\left(\beta_l\right)/2 \right\rangle.
$$
\end{conj}

\subsection{Other decomposable Specht modules}\label{2infinite2}
We are indebted to   Matt Fayers for sharing the following  examples (which he   discovered  by computer) after we posited that the two families in \cref{2infinite} might be the only counterexamples to
  the quantised version of his question   \cite[Section 8.2]{MR2905253}.
\begin{figure}[ht!]
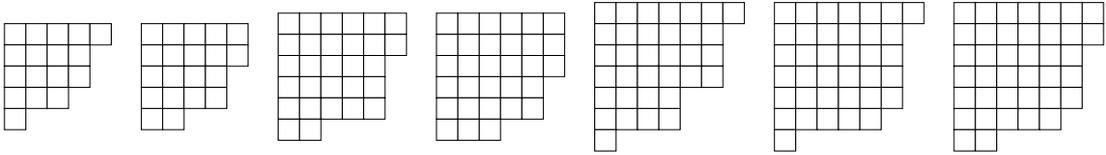

$$\Yvcentermath1\Yboxdim{8pt}
\gyoung(;;;;;,;;;;,;;;;,;;;,;)
\quad
\gyoung(;;;;;,;;;;;,;;;;,;;;;,;;)
\quad
\gyoung(;;;;;;,;;;;;;,;;;;;,;;;;;,;;;;;,;;)
\quad
\gyoung(;;;;;;,;;;;;;,;;;;;;,;;;;;,;;;;;,;;;)
\quad
\gyoung(;;;;;;;,;;;;;;,;;;;;;,;;;;;;,;;;;,;;;;,;)
\quad
\gyoung(;;;;;;;,;;;;;;,;;;;;;,;;;;;;,;;;;;;,;;;;;,;)
\quad
\gyoung(;;;;;;;,;;;;;;;,;;;;;;,;;;;;;,;;;;;;,;;;;;,;;)$$
\caption{More partitions labelling  semisimple Specht modules.}
\label{alabel}
\end{figure}

We hope that the examples in \cref{alabel} serve as inspiration for further work towards a classification of semisimple Specht modules.  Several more examples can be obtained from those in \cref{alabel}  by $i$-induction for $i=0,1$ (analogously to \cref{2infinite}) namely: $(8,7,6^2,4^2,2,1)$, $(7^3,6^3,3)$ and $(8,7^3,6^2,4,1)$.  Finally we have one partition which breaks the mould: $(6,5^3,3,1)$  which is  the only partition in this section   not equal to its own conjugate.

\subsection{Patterns}

The examples of \cref{2infinite,2infinite2} do share several striking similarities.  Firstly, all the examples in \cref{2infinite,2infinite2} have a direct summand which is isomorphic to a simple Specht module.
Secondly, all those of \cref{2infinite2} decompose as a direct sum of simples concentrated in one degree and so are semisimple by \cref{directsum}.  We conjecture this is also true of the infinite families in \cref{directsum}.
 It is interesting to speculate whether the converse of \cref{directsum} is also true:
 is semisimplicity of  a Specht module  equivalent to its composition factors being focussed in one degree?

\newpage
\appendix

\section{Two new infinite families of decomposable Specht modules \\ of Hecke algebras over $\mathbb{C}$}\label{sounusual!}

 We now prove that the  two infinite  families of decomposable Specht modules announced in \cref{more} are indeed  decomposable.

\subsection{Semistandard homomorphisms}\label{ss:permutationmodules}
We now recall the construction of Young permutation modules for Hecke algebras
and how they can be used to calculate homomorphisms between Specht modules.
Given $\la $ a partition of $n$, we define the $H^\Bbbk_q(n)$-module
$$
{\bf M}^\Bbbk_q(\la) = H^\Bbbk_q(n) x_\la
$$
which we refer to as the {\sf permutation module} labelled by $\la$.
   For $\SSTS\in\RStd(\la,\mu)$, there exists a   homomorphism $\Theta_{\SSTS}:{\bf M}^\Bbbk_q(\la)\rightarrow {\bf M}^\Bbbk_q(\mu)$ (see \cite[Section 4]{Mathas} for explicit details) and the set of all row standard tableau homomorphisms  $\left\{\Theta_{\SSTS}\mid \SSTS\in\RStd(\la,\mu) \right\}$  forms a basis of $\operatorname{Hom}_{H^\Bbbk_q(n)}\left({\bf M}^\Bbbk_q(\la),{\bf M}^\Bbbk_q(\mu)\right)$.
We let $\widehat{\Theta}_{\SSTS}$ denote the restriction of $\Theta_{\SSTS}$ to ${\bf S}^\Bbbk_q({\la})$.

\begin{thm}\cite[Corollary 8.7]{dj91}
	The set
$
 \{ \widehat{\Theta}_{\SSTS} \mid \SSTS\in\SStd( \la,\mu)  \}
$
is a linearly independent  subset of $\operatorname{Hom}_{H^\Bbbk_q(n)}\left({\bf S}^\Bbbk_q(\la),{\bf M}^\Bbbk_q(\mu)\right)$.
 If either $q\neq -1$ or $\la$ is $2$-regular, then $
 \{ \widehat{\Theta}_{\SSTS} \mid \SSTS\in\SStd( \la,\mu)  \}
$ provides a basis of this space.
\end{thm}

 This paper focuses on the case $q=-1$.
While the  homomorphisms $\widehat{\Theta}_{\SSTS}$ for $ \SSTS\in\RStd(\la,\mu) $
do  not necessarily provide a basis of
 $\operatorname{Hom}_{H^\Bbbk_{-1}(n)}\left({\bf S}^\Bbbk_{-1}(\la),{\bf S}^\Bbbk_{-1}(\mu)\right)$, we shall
 see  that     tableau-theoretic homomorphisms are a useful tool for analysing  the structure of Specht modules   in \cref{sounusual!}.
Let $\la\vdash n$. Fix $d\in\mathbb{N}$ and $t$ such that $0\leqslant t<\la_{d+1}$. We
set
\[
\la_i^{d,t}=\begin{cases}
\la_i+\la_{i+1}-t
&\text{if $i=d$,}\\
t&\text{if $i=d+1$,}\\
\la_i&\text{otherwise.}
\end{cases}
\]
Let $\SSTS_{d,t}\in\RStd(\la,\la^{d,t})$ be such that all of its entries in row $i$ are equal to $i$, except for row $d+1$ which contains $\mu_{d+1}-t$ entries equal to $d$ and $t$ entries equal to $d+1$. We now write $\varphi_{d,t}$ for the homomorphism $\Theta_{\SSTS_{d,t}}:{\bf M}^\Bbbk_q(\lambda)\rightarrow {\bf M}^\Bbbk_q(\lambda ^{d,t})$. The homomorphism $\varphi_{d,t}$ enables us to give the following alternative definition of a Specht module.

\begin{thm}[{\cite[Theorem 7.5]{dj86}}]\label{sadwtoiyreiotyweiortuyuwoiertyuweoirty}
	Let $\lambda\vdash n$. Then   the {Specht module} ${\bf S}^\Bbbk_{q}(\la)$ can be constructed as the intersection of the kernels of these homomorphisms as follows,
	\[
{\bf S}^\Bbbk_{q}(\la)=\bigcap_{\substack{d\geqslant 1\\
			1\leqslant t\leqslant \lambda_{d+1}}}
	\operatorname{ker}\varphi_{d,t}.
	\]
 \end{thm}

 Therefore we
 can  use the homomorphisms  of \cref{sadwtoiyreiotyweiortuyuwoiertyuweoirty} to determine whether or not the image of  a tableau homomorphism $\widehat{\Theta}_{\SSTS}$ lies in the Specht module ${\bf S}^\Bbbk_q(\la)$.

\begin{cor}
Let $\lambda,\mu\vdash n$ and suppose that $\widehat{\Theta}_{\SSTT}\in\operatorname{Hom}_{H^\Bbbk_q(n)}\left({\bf S}^\Bbbk_q(\la),{\bf M}^\Bbbk_q(\mu)\right)$ for some $\SSTT\in\RStd(\la,\mu)$. Then ${\rm im}(\widehat{\Theta}_{\SSTT})\subseteq {\bf S}^\Bbbk_q(\mu)$ if and only if $\varphi_{d,t}\circ\widehat{\Theta}_{\SSTT}=0$ for all $d\geqslant 1$ and $0\leqslant t<\mu_{d+1}$.
\end{cor}

\subsection{Statement and outline of proof}
\begin{defn}Given $k\in \mathbb{N}$, we define $\alpha_k \vdash (k +2)^2-4$ to be the partition
 $\alpha_k =((k+2)^k,k^2)$.
Given $k\in 2\mathbb{N}+1$, we define $\beta_k \vdash (k +2)^2-2$ to be the partition
 $\beta_k = (k+3,(k+2)^{k-1},k^2,1)$.
\end{defn}

\begin{eg}
The Young diagrams of the partitions $\alpha_5$ and $\beta_5$, along with their residues, are drawn as follows:
\[
[\alpha_5]=\gyoung(;0;1;0;1;0;1;0,;1;0;1;0;1;0;1,;0;1;0;1;0;1;0,;1;0;1;0;1;0;1,;0;1;0;1;0;1;0,;1;0;1;0;1,;0;1;0;1;0)
\qquad
[\beta_5]=\gyoung(;0;1;0;1;0;1;0!\dgr1,!\wh1;0;1;0;1;0;1,;0;1;0;1;0;1;0,;1;0;1;0;1;0;1,;0;1;0;1;0;1;0,;1;0;1;0;1,;0;1;0;1;0,!\dgr1).
\]
\end{eg}

We now state our main result in this section (which we later split into two results).

\begin{thm}\label{unusual}
The Specht modules ${\bf S}^\CC_{-1}(\alpha_k)$ and ${\bf S}^\CC_{-1}(\beta_{k})$ are decomposable.
\end{thm}

 We remark  that $\alpha_k=((k+2)^k,k^2)$ is 2-separated if and only if $k=1$, in which case, we already know from \cref{2qs} that ${\bf S}^\CC_{-1}(3,1^2)$ is semisimple (see also \cite{MR3250450}).
We shall focus on the decomposability of ${\bf S}^\CC_{-1}(\alpha_k)$;
 we shall see later that it is relatively straightforward to determine the decomposability of ${\bf S}^\CC_{-1}(\beta_k)$ from ${\bf S}^\CC_{-1}(\alpha_k)$; this is because $\beta_k$ is obtained from $\alpha_k$ by adding all possible addable nodes of residue 1 (the observant reader will notice that this is why $k$ being odd is a necessary condition in the definition of $\beta_k$).

\begin{prop}\cite[Theorem 4.7]{MR1402572}\label{noofparts}
	Let $\alpha,\mu\vdash n$ with $\alpha_r\geqslant r$ for some $r\in\mathbb{N}$ and $\mu$ a $2$-regular partition. If $[{\bf S}^\CC_{-1}(\alpha):D^\CC_{-1}({\mu})]>0$, then $\ell(\mu)\geq r$.
\end{prop}

Fayers--Lyle~\cite{MR3011340} provide a conjectural classfication of all  irreducible Specht $ {H}^{\mathbb{C}}_{-1}(n)$-modules.
Mathas has provided a proof of  the Fayers--Lyle conjecture for   rectangular partitions, using the above result.  {\em The following result and its proof is unpublished and conveyed to us by Fayers.}

\begin{thm}\label{rectangle}
Let $\nu=(m^k)$ for some $m,k\in\mathbb{N}$. Then ${\bf S}^\CC_{-1}(\lambda)$ irreducible. In particular, ${\bf S}^\CC_{-1}(\lambda)\cong {\bf D}^\CC_{-1}(\lambda^R)$.
\end{thm}

\begin{proof}
We  suppose without loss of generality (as ${\bf S}^\CC_{-1}(\lambda)$ and  ${\bf S}^\CC_{-1}(\lambda^T)$ have the same decomposition multiplicities) that $m\geqslant k$ and so $\nu_k\geqslant k$. Now, by applying \cref{noofparts}, we know that if $D^\CC_{-1}({\mu})$ is a composition factor of ${\bf S}^\CC_{-1}(\lambda)$ for some $2$-regular $\mu\vdash n$, then $\mu$ has at least $k$ non-zero parts.
 On the other hand, we know from \cref{James-reg} that $[{\bf S}^\CC_{-1}(\lambda):{\bf D}^\CC_{-1}({\mu})]\neq 0$ implies that $\mu\trianglerighteq  \nu^R$, and hence $\mu$ has at most $k$ parts.

We can thus assume that $\mu$ has exactly $k$ parts, and write $\mu=(\mu_1,\mu_2,\dots,\mu_k)$.
	We now apply induction on the statement that ${\bf S}^\CC_{-1}(\lambda)$ is simple  for all $m\in\mathbb{N}$. If $m=1$, then ${\bf S}_{-1}^\CC{(m)}\cong
	{\bf D}_{-1}^\CC{(m)}$   and   the result holds. Now suppose that $m>1$, we obtain
\[
\left[{\bf S}^\CC_{-1}(\lambda):{\bf D}^\CC_{-1}({\mu})\right]=
\begin{cases}
\left[ {\bf S}^\CC_{-1} \left({(m-1)^k}\right): {\bf D}^\CC_{-1} {\left(\mu_1-1,\mu_2-1,\dots,\mu_k-1\right)}\right] &\text{if $\ell(\mu)=k$} \\
0																&\text{otherwise,}
\end{cases}
\]
by applying the column removal result \cite[Rule 5.8]{j90}.
 By induction, we have that
 $$ {\bf S}^\CC_{-1} \left({(m-1)^k}\right) \cong {\bf S}^\CC_{-1} \left(k^{m-1}\right) \cong {\bf D}^\CC_{-1} {\left(\mu_1-1,\mu_2-1,\dots,\mu_k-1\right)},$$
 and the result follows.\end{proof}

\begin{cor}
Let $\la=((k+4)^k)$ for all $k\in\mathbb{N}$. Then ${\bf S}^\CC_{-1}(\lambda)$ is irreducible and  $${\bf S}^\CC_{-1}\left(\lambda\right)\cong {\bf D}_{-1}^\CC {\left(2k+3,2k+1,2k-1,\dots,9,7,5\right)}.$$
\end{cor}

\noindent In what follows, we shall denote $$\alpha_k^C=((k+4)^k)^R=(2k+3,2k+1,2k-1,\dots,9,7,5).$$

 While we are able to calculate both the  non-zero decomposition numbers in the above
and prove decomposability, we do not have a proof that all other decomposition numbers are zero (and hence we do not prove semisimplicity either).
Our proof of  \cref{unusual} proceeds in two steps as follows:
\begin{itemize}
\item[$(i)$]
Prove that  ${\bf D}_{-1}^\CC {(\alpha_k^{\rm C})}$ belongs to the socle of  ${\bf S}_{-1}^\CC (\alpha_k)$;
\item[$(ii)$]  Prove that  $ [ {\bf S}_{-1}^\CC (\alpha_k): {\bf D}_{-1}^\CC {(\alpha_k^{\rm C})} ] \leq 1$ and thus deduce (by self-duality) that ${\bf D}_{-1}^\CC  (\alpha_k^{\rm C}) $ is a direct summand of ${\bf S}_{-1}^\CC (\alpha_k)$.
\end{itemize}
 Using the previous corollary, it is enough to show that
$$\Hom_{H^\CC_{-1}(k^2-2)}\left({\bf S}_{-1}^\CC {\left( (k+4)^k\right)}, {\bf D}_{-1}^\CC \left(\alpha_k\right)\right)\neq 0$$
to immediately deduce part $(i)$.  To deduce part $(ii)$, we shall prove that
\begin{align}\label{1}
\dim_t\left(e_{\Lad\left(\alpha_k^{\rm C}\right)}{\bf S}_{-1}^\CC \left(\alpha_k\right)\right)&=2 t^{\left\lfloor w\left(\alpha_k\right)/2\right\rfloor }
\dim_t\left(e_{\Lad\left(\alpha_k^{\rm C}\right)}{\bf D}_{-1}^\CC{\left(\alpha_k^{\rm C}\right)}\right)
\\\label{2}
 \dim_t\left(e_{\Lad\left(\alpha_k^{\rm C}\right)}
 {\bf D}_{-1}^\CC{\left(\alpha_k^{\rm R}\right)})\right) &=
\dim_t\left(e_{\Lad(\alpha_k^{\rm C})}{\bf D}_{-1}^\CC{\left(\alpha_k^{\rm C}\right)}\right)
+ \dots
\end{align}
where the other terms in \cref{2} belong to $t\NN_0[t]$.  We deduce from  \cref{1}
that $ [ {\bf S}_{-1}^\CC (\alpha_k): {\bf D}_{-1}^\CC {(\alpha_k^{\rm C})} ] \leq 2$ and then
refine this to the sharp upper bound of 1 using \cref{2} and part $(i)$ along with the results of Kleshchev--Nash from \cref{clored}.

\subsection{Proof part $(i)$: the homomorphism }\label{louise}

We employ the method set out in \cref{ss:permutationmodules} to describe a Specht module homomorphism between ${\bf S}_{-1}^\CC {( (k+4)^k)}$ and ${\bf S}_{-1}^\CC {((k+2)^k,k^2)}$ from the following definition of a semistandard tableau.

Let $\mu=((k+4)^k)$ and $\lambda=((k+2)^k,k^2)$ for all $k>1$. We define the tableau $\SSTA \in \SStd(\mu,\lambda)$ as follows:
\[\scalefont{0.8}
\SSTA=\gyoungxy(1.6pt,1.6pt,;1;1;1;1;1;1;1;1;1;1;1;1;\hdts;1!\gr2!\wh3,;2;2;2;2;2;2;2;2;2;2;2;\hdts;2!\gr3!\wh4;4,;3;3;3;3;3;3;3;3;3;3;\hdts;3!\gr4!\wh5;5;5,;4;4;4;4;4;4;4;4;4;\hdts;4!\gr5!\wh6;6;6;6,;5;5;5;5;5;5;5;5;\hdts;5!\gr6!\wh7;7;7;7;7,;6;6;6;6;6;6;6;\hdts;6!\gr7!\wh8;8;8;8;8;8,;\vdts;\vdts;\vdts;\vdts;\vdts;\vdts;\vdts;!\gr\iddots!\wh;;\vdts;\vdts;\vdts;\vdts;\vdts;\vdts,;<k{-}4>;<k{-}4>;<k{-}4>;<k{-}4>;<k{-}4>;<k{-}4>;<k{-}4>!\gr<k{-}3>!\wh<k{-}2>;\hdts;<k{-}2>;<k{-}2>;<k{-}2>;<k{-}2>;<k{-}2>;<k{-}2>,;<k{-}3>;<k{-}3>;<k{-}3>;<k{-}3>;<k{-}3>;<k{-}3>!\gr<k{-}2>!\wh<k{-}1>;\hdts;<k{-}1>;<k{-}1>;<k{-}1>;<k{-}1>;<k{-}1>;<k{-}1>;<k{-}1>,;<k{-}2>;<k{-}2>;<k{-}2>;<k{-}2>;<k{-}2>!\gr<k{-}1>!\wh<k>;\hdts;<k>;<k>;<k>;<k>;<k>;<k>;<k>;<k>,;<k{-}1>;<k{-}1>;<k{-}1>;<k{-}1>!\gr<k>!\wh<k{+}1>;\hdts;<k{+}1>;<k{+}1>;<k{+}1>;<k{+}1>;<k{+}1>;<k{+}1>;<k{+}1>;<k{+}1>;<k{+}1>,;<k>;<k>;<k>!\gr<k{+}1>!\wh<k{+}2>;\hdts;<k{+}2>;<k{+}2>;<k{+}2>;<k{+}2>;<k{+}2>;<k{+}2>;<k{+}2>;<k{+}2>;<k{+}2>;<k{+}2>)
\]
Since $\SSTA$ is semistandard, we know that $\widehat{\Theta}_{\SSTA}:{\bf S}_{-1}^\CC {( (k+4)^k)}\rightarrow {M}_{-1}^\CC {((k+2)^k,k^2)}$ is a non-zero   homomorphism; we claim that the image of $\widehat{\Theta}_{\SSTA}$ lies in ${\bf S}_{-1}^\CC {((k+2)^k,k^2)}$.
We will heavily rely on Fayers' algorithm~\cite[Theorem 3.1]{MR2927046}, which semistandardises tableau homomorphisms solely by manipulating row-standard tableaux.

\begin{thm}
	\label{ssalgorithm}
	Suppose that $\lambda$ and $\mu$ are partitions of $n$, and $\SSTT$ is a row-standard $\lambda$-tableau of type $\mu$. Suppose $i\geqslant 1$, and $A,B,C$ are multisets of positive integers such that $|B|>\lambda_i$ and $$(A\sqcup B\sqcup C)\cap\{(r,c)\mid r\in \{i,i+1\}, c\geq 1\} =\SSTT\cap \{(r,c)\mid r\in \{i,i+1\},c\geq 1\}.$$  Let $\mathscr{B}$ be the set of all pairs $(D,E)$ of multisets such that $|D|=\lambda_i-|A|$ and $B=D\sqcup E$. For each such pair $(D,E)$, define $\SSTT_{D,E}$ to be the row-standard tableau with
	\[
	\SSTT_{D,E} \cap \{(r,c)\mid c\geq 1\}=
	\begin{cases}
	A\sqcup D &(r=i)\\
	C\sqcup E &(r=i+1)\\
	\SSTT \cap \{(r,c)\mid c\geq 1\}
	&(\text{otherwise}).
	\end{cases}
	\]
	Then
	\begin{equation}\label{eqnssalgorithm}
	\sum_{(D,E)\in\mathscr{B}}
	\prod_{i\geqslant 1}
	\stirling{A_i+D_i}{A_i}
	\stirling{C_i+E_i}{C_i}
	\prod_{i<j}q^{\left(A_jD_i+C_iE_j\right)}
	\widehat{\Theta}_{\SSTT_{D,E}}=0.
	\end{equation}\end{thm}

Given $\SSTT$ as in \cref{ssalgorithm}, we know that there exists some $c\geqslant 1$ such that the entries in the boxes $(i,c)$ and $(i+1,c)$ of $\SSTT$ are equal, that is $\SSTT(i,c)=\SSTT(i+1,c)=j$ for some $j\geqslant 1$. We let $\SSTT^{(i)}=\{(i,c) \mid c\geqslant 1\}$ for all $i\geqslant 1$. We now explicitly define the multisets $A$, $B$ and $C$ as follows:
\begin{align*}
A &= \{ k\in\SSTT^{(i)} \mid k<j \}\\
B &= \{k\in\SSTT^{(i)} \mid k\geqslant j \} \sqcup \{k\in\SSTT^{(i+1)} \mid k\leqslant j \}\\
C &= \{ k\in\SSTT^{(i+1)} \mid k>j \}.
\end{align*}

Recall from \cref{ss:permutationmodules} that $\widehat{\Theta}_{\SSTA}\subseteq{\rm im}({\bf S}_{-1}^\CC {((k+2)^k,k^2)})$ if the map composition of $\widehat{\Theta_{\SSTA}}$ with $\varphi_{d,t}$ is zero for all $d$ and $t$.
In order to do so, we now make extensive use of Fayers' semistandardising algorithm to show that certain row-standard tableau homomorphisms are zero; we invite the general reader to skip to the main result of this Subsection on Page 29 to avoid the following technical results, should they wish to do so.
We note that when $q=-1$, the tableau homomorphism $\widehat{\Theta}_{\SSTT_{D,E}}$ appears in the above sum of \cref{eqnssalgorithm} with coefficient $-1$, $0$ or $1$; we shade diagrams throughout for clarity.

 \begin{lem}\label{lem2}
	Let $\lambda=((k+4)^k)$ and $\mu=((k+2)^k,k^2)$ for some $k>1$, and let $d\in\{2,\dots,k-1\}$.
	\begin{enumerate}[leftmargin=*]
		\item Suppose that $t\geqslant u\geqslant 1$, $w\geqslant 3$ such that $w$ is odd, and $v\geqslant 0$.
		We have that $\widehat{\Theta}_{\SSTS}=0$ for any tableau $\SSTS\in\RStd(\la,\mu)$  whose $d$th and $(d+1)$th rows have the following form:
		\[\scalefont{0.8}
		\begin{tikzpicture}[scale=1.6]
		 \tgyoung(0cm,0cm,;<d>;\hdts;<d>;<d>;\hdts;<d>;<d>;\hdts;<d>;<d>;<d>;<d{+}2>;\hdts;<d{+}2>,;<d>;\hdts;<d>!\gr<d{+}1>;\hdts<d{+}1>!\wh<d{+}2>;\hdts;<d{+}2>;<d{+}3>;<d{+}3>;<d{+}3>;\hdts;<d{+}3>)
		\draw[decoration={brace,mirror,raise=5pt},decorate]	(0.05,-0.44) -- node[below=6pt] {$u$ times} (1.35,-0.44);
		\draw[decoration={brace,mirror,raise=5pt},decorate]	(1.4,-0.44) -- node[below=6pt] {$v$ times} (2.75,-0.44);
		\draw[decoration={brace,mirror,raise=5pt},decorate]	(2.8,-0.44) -- node[below=6pt] {$w$ times} (4.1,-0.44);
		\draw[decoration={brace,raise=5pt},decorate]	(0.05,0.43) -- node[above=6pt] {$t$ times} (5,0.43);
		\end{tikzpicture}
		\]
		
		\item Suppose that $t\geqslant u\geqslant 2$, $w\geqslant 3$ such that $w$ is odd, and $v\geqslant 0$.
		We have that $\widehat{\Theta}_{\SSTT}=0$ for any tableau  $\SSTT\in\RStd(\la,\mu)$   whose $d$th and $(d+1)$th rows have the following form:
		\[\scalefont{0.8}
		\begin{tikzpicture}[scale=1.6]
		 \tgyoung(0cm,0cm,;<d>;\hdts;<d>;<d>;\hdts;<d>;<d>;\hdts;<d>;<d>!\gr<d{+}1>!\wh<d{+}2>;\hdts;<d{+}2>,;<d>;\hdts;<d>!\gr<d{+}1>;\hdts<d{+}1>!\wh<d{+}2>;\hdts;<d{+}2>;<d{+}3>;<d{+}3>;<d{+}3>;\hdts;<d{+}3>)
		\draw[decoration={brace,mirror,raise=5pt},decorate]	(0.05,-0.44) -- node[below=6pt] {$u$ times} (1.35,-0.44);
		\draw[decoration={brace,mirror,raise=5pt},decorate]	(1.4,-0.44) -- node[below=6pt] {$v$ times} (2.75,-0.44);
		\draw[decoration={brace,mirror,raise=5pt},decorate]	(2.8,-0.44) -- node[below=6pt] {$w$ times} (4.1,-0.44);
		\draw[decoration={brace,raise=5pt},decorate]	(0.05,0.43) -- node[above=6pt] {$t$ times} (4.55,0.43);
		\end{tikzpicture}
		\]
	\end{enumerate}	
\end{lem}

\begin{proof}
	We let $i=d$, and apply \cref{ssalgorithm} to $\widehat{\Theta}_{\SSTS}$ and $\widehat{\Theta}_{\SSTT}$, respectively. In both cases, we  set $A=\emptyset$ and $C =\{(d+1)^{v},(d+2)^w,(d+3)^{k+4-u-v-w}\}$, by which we mean
	\[C=\left\{ \underbrace{d+1,\dots,d+1}_{\text{$v$ times}}, \underbrace{d+2,\dots,d+2}_{\text{$w$ times}}, \underbrace{d+3,\dots,d+3}_{\text{$k+4-u-v-w$ times}} \right\}.\]
	Thus every  morphism $\widehat{\Theta}_{\SSTS_{D,E}}$ and $\widehat{\Theta}_{\SSTT_{D,E}}$ in \cref{eqnssalgorithm} appears with some scalar multiple of
	\[
	\stirling{E_{d+2}+w}{w}=\begin{cases}
	1&\text{if $E_{d+2}=0$}\\
	0&\text{if $E_{d+2}\geqslant 1$}.
	\end{cases}
	\]
	Therefore we need now only consider the case $E_{d+2}=0$, which we now do in both cases:
	
	\smallskip
	\noindent{\bf Case 1:} Observe that
	$B=\{ d^{t+u}, ({d+2 })^ {k-t+4 } \}$
	and hence when we set  $E_{d+2}=0$ we obtain only the original tableau $\SSTS$ (which always  appears with coefficient 1) in \cref{eqnssalgorithm};   hence   $\widehat{\Theta}_{\SSTS}=0$.

	\smallskip
	\noindent{\bf Case 2:}  Observe that
	$B=\{  d^{t+u}, d+1,  (d+2)^{k-t+3} \}$.
	There are two tableaux obtained from setting $E_{d+2}=0$, namely $\SSTT$ and the tableau $\SSTR$ that is identical to $\SSTT$ except for its $d$th and $(d+1)$th rows, which have the following form
	\[\scalefont{0.8}
	\begin{tikzpicture}[scale=1.6]
	 \tgyoung(0cm,0cm,;<d>;\hdts;<d>;<d>;<d>;\hdts;<d>;<d>;\hdts;<d>;<d>;<d{+}2>;\hdts;<d{+}2>,;<d>;\hdts;<d>!\gr<d{+}1>;<d{+}1>;\hdts<d{+}1>!\wh<d{+}2>;\hdts;<d{+}2>;<d{+}3>;<d{+}3>;\hdts;<d{+}3>)
	\draw[decoration={brace,mirror,raise=5pt},decorate]	(0.05,-0.44) -- node[below=6pt] {$u-1$ times} (1.35,-0.44);
	\draw[decoration={brace,mirror,raise=5pt},decorate]	(1.4,-0.44) -- node[below=6pt] {$v+1$ times} (3.15,-0.44);
	\draw[decoration={brace,mirror,raise=5pt},decorate]	(3.25,-0.44) -- node[below=6pt] {$w$ times} (4.55,-0.44);
	\draw[decoration={brace,raise=5pt},decorate]	(0.05,0.43) -- node[above=6pt] {$t+1$ times} (5,0.43);
	\end{tikzpicture}
	\]
	We can thus write $\alpha_{\SSTR}\widehat{\Theta}_{\SSTR}+\widehat{\Theta}_{\SSTT}=0$ for some $\alpha_{\SSTR}\in\{-1,1\}$.
	It follows from part (1) that $\widehat{\Theta}_{\SSTR}=0$, and hence $\widehat{\Theta}_{\SSTT}=0$, as required.\qedhere
\end{proof}

\begin{lem}\label{lem5}
	Let $\lambda=((k+4)^k)$ and $\mu=((k+2)^k,k^2)$ for some $k>1$, and let $d\in\{2,\dots,k-1\}$.
	Suppose that either
	\begin{enumerate}[leftmargin=*]
		\item $u>s\geqslant 0$, $v\geqslant 3$ such that $v$ is odd, and $t\geqslant 0$, or
		\item $u>s+1>0$, $v\geqslant 2$ such that $v$ is even, and $t\geqslant 0$.
	\end{enumerate}
	We have that
	$\widehat{\Theta}_{\SSTT}=0$ for any
	$\SSTT\in\RStd(\la,\mu)$  whose $(d-1)$th and $d$th rows have the following form
	\[\scalefont{0.8}
	\begin{tikzpicture}[scale=1.6]
	 \tgyoung(0cm,0cm,;<d{-}1>;\hdts;<d{-}1>;<d>;<d>;<d>;\hdts;<d>!\gr<d{+}1>;<d{+}1>;\hdts;<d{+}1>,!\wh<d>;\hdts;<d>;<d>;<d>!\gr<d{+}1>;\hdts;<d{+}1>!\wh<d{+}2>;<d{+}2>;\hdts;<d{+}2>)
	\draw[decoration={brace,mirror,raise=5pt},decorate]	(0.05,-0.44) -- node[below=6pt] {$u$ times} (2.25,-0.44);
	\draw[decoration={brace,mirror,raise=5pt},decorate]	(2.35,-0.44) -- node[below=6pt] {$v$ times} (3.6,-0.44);
	\draw[decoration={brace,raise=5pt},decorate]	(0.05,0.43) -- node[above=6pt] {$s$ times} (1.3,0.43);
	\draw[decoration={brace,raise=5pt},decorate]	(1.4,0.43) -- node[above=6pt] {$t$ times} (3.6,0.43);
	\end{tikzpicture}
	\]
\end{lem}
\begin{proof}
	
	Fix $i=d-1$ and set
	$$A=\{  (d-1)^s\}, \quad  B=\{ d^{t+u }, (d+1)^{k+4-s-t} \},\quad
	C=\{ (d+1)^{v}, (d+2)^{k+4-u-v } \}.$$
	
	\smallskip
	\noindent{\bf Case 1:}

	Observe that every homomorphism $\widehat{\Theta}_{\SSTT_{D,E}}$ in \cref{ssalgorithm} appears with some scalar multiple of
	\[
	\stirling{E_{d+1}+v}{v}=\begin{cases}
	1&\text{if $E_{d+1}=0$}\\
	0&\text{if $E_{d+1}\geqslant 1$}.
	\end{cases}
	\]
	The only tableau $\SSTT_{D,E}$ with $E_{d+1}=0$ is $\SSTT$ itself, and hence $\widehat{\Theta}_{\SSTT}=0$.
	
	\smallskip
	\noindent{\bf Case 2:}   In this case, every homomorphism $\widehat{\Theta}_{\SSTT_{D,E}}$ in \cref{ssalgorithm} appears with some scalar multiple of
	\[
	\stirling{E_{d+1}+v}{v}=\begin{cases}
	1&\text{if $E_{d+1}=0$ or $1$}\\
	0&\text{if $E_{d+1}\geqslant 2$}.
	\end{cases}
	\] There are two possible tableaux $\SSTT_{D,E}$ such that the above quantum binomial coefficient is non-zero, namely $\SSTT$ (with $E_{d+1}=0$) and the tableau $\SSTS$ (with $E_{d+1}=1$) that is identical to $\SSTT$ except for its $(d-1)$th and $d$th rows, which have the following form
	\[\scalefont{0.8}
	\begin{tikzpicture}[scale=1.6]
	 \tgyoung(0cm,0cm,;<d{-}1>;\hdts;<d{-}1>;<d>;<d>;<d>;\hdts;<d>;<d>!\gr<d{+}1>;\hdts;<d{+}1>,!\wh<d>;\hdts;<d>;<d>!\gr<d{+}1>;<d{+}1>;\hdts;<d{+}1>!\wh<d{+}2>;<d{+}2>;\hdts;<d{+}2>)
	\draw[decoration={brace,mirror,raise=5pt},decorate]	(0.05,-0.44) -- node[below=6pt] {$u-1$ times} (1.8,-0.44);
	\draw[decoration={brace,mirror,raise=5pt},decorate]	(1.9,-0.44) -- node[below=6pt] {$v+1$ times} (3.65,-0.44);
	\draw[decoration={brace,raise=5pt},decorate]	(0.05,0.43) -- node[above=6pt] {$s$ times} (1.3,0.43);
	\draw[decoration={brace,raise=5pt},decorate]	(1.4,0.43) -- node[above=6pt] {$t+1$ times} (4.1,0.43);
	\end{tikzpicture}
	\]
	We thus have $\alpha_{\SSTS}\widehat{\Theta}_{\SSTS}+\widehat{\Theta}_{\SSTT}=0$ for some $\alpha_{\SSTS}\in\{-1,1\}$.
	It follows from part (1) that $\widehat{\Theta}_{\SSTS}=0$, and hence $\widehat{\Theta}_{\SSTT}=0$.
\end{proof}

The following result will be the most instrumental in proving that there exists a non-zero homomorphism between ${\bf S}_{-1}^\CC {((k+4)^k)}$ and ${\bf S}_{-1}^\CC {((k+2)^k,k^2)}$; in particular, we cancel out most  pairs of relevant tableau homomorphisms (two exceptional cases will be deferred to the next lemma).

\begin{lem}\label{lem1}
	Let $\lambda=((k+4)^k)$ and $\mu=((k+2)^k,k^2)$ for some $k>1$  and  $d\in\{2,\dots,k-1\}$. The set of tableaux in $\RStd(\la,\mu)$  obtained from $\SSTA$ by replacing $\bar{t}\in\{1,\dots,k+2\}$ of the entries  equal to $d+1$ with the entry $d$ fall into two families:
	
	The first family is denoted by $\SSTS^l$.   The $(d-1)$th, $d$th and $(d+1)$th rows of $\SSTS^l$ are
	\[\scalefont{0.8}
	\begin{tikzpicture}[scale=1.6]
	 \tgyoung(0cm,0cm,;<d{-}1>;<d{-}1>;\hdts;<d{-}1>;<d{-}1>;\hdts;<d{-}1>;<d{-}1>;<d{-}1>;<d>!\dgr<d>;\hdts;<d>!\gr<d{+}1>;\hdts;<d{+}1>,!\wh<d>;<d>;\hdts;<d>;<d>;\hdts;<d>;<d>!\gr<d{+}1>!\wh<d{+}2>;<d{+}2>;\hdts;<d{+}2>;<d{+}2>;\hdts;<d{+}2>,!\dgr<d>;<d>;\hdts;<d>!\gr<d{+}1>;\hdts;<d{+}1>!\wh<d{+}2>;<d{+}3>;<d{+}3>;<d{+}3>;\hdts;<d{+}3>;<d{+}3>;\hdts;<d{+}3>)
	\draw[decoration={brace,mirror,raise=5pt},decorate]	(0.05,-0.9) -- node[below=6pt] {$l$ times} (1.8,-0.9);
	\draw[decoration={brace,mirror,raise=5pt},decorate]	(3.7,-0.9) -- node[below=6pt] {$d+1$ times} (7.3,-0.9);
	\draw[decoration={brace,raise=5pt},decorate]	(0.05,0.43) -- node[above=6pt] {$k-d+4$ times} (4.1,0.43);
	\draw[decoration={brace,raise=5pt},decorate]	(4.6,0.43) -- node[above=6pt] {$\bar{t}-l$ times} (5.9,0.43);
	\end{tikzpicture}
	\]
	with $l\leqslant \bar{t}$,
	and all other rows agree with those of $\SSTA$.
	The second family is $\SSTT^l $. The   $(d-1)$th, $d$th and $(d+1)$th rows of  $\SSTT^l$ are
	\[\scalefont{0.8}
	\begin{tikzpicture}[scale=1.6]
	 \tgyoung(0cm,0cm,;<d{-}1>;\hdts;<d{-}1>;<d{-}1>;<d{-}1>;\hdts;<d{-}1>;<d{-}1>;<d{-}1>;<d>!\dgr<d>;\hdts;<d>!\gr<d{+}1>;\hdts;<d{+}1>,!\wh<d>;\hdts;<d>;<d>;<d>;\hdts;<d>;<d>!\dgr<d>!\wh<d{+}2>;<d{+}2>;\hdts;<d{+}2>;<d{+}2>;\hdts;<d{+}2>,!\dgr<d>;\hdts;<d>!\gr<d{+}1>;<d{+}1>;\hdts;<d{+}1>!\wh<d{+}2>;<d{+}3>;<d{+}3>;<d{+}3>;\hdts;<d{+}3>;<d{+}3>;\hdts;<d{+}3>)
	\draw[decoration={brace,mirror,raise=5pt},decorate]	(0.05,-0.9) -- node[below=6pt] {$l-1$ times} (1.35,-0.9);
	\draw[decoration={brace,mirror,raise=5pt},decorate]	(3.7,-0.9) -- node[below=6pt] {$d+1$ times} (7.3,-0.9);
	\draw[decoration={brace,raise=5pt},decorate]	(0.05,0.43) -- node[above=6pt] {$k-d+4$ times} (4.1,0.43);
	\draw[decoration={brace,raise=5pt},decorate]	(4.6,0.43) -- node[above=6pt] {$\bar{t}-l$ times} (5.9,0.43);
	\end{tikzpicture}
	\]
	with $l\leqslant \bar{t}$,
	and all other rows agree with those of $\SSTA$.
	\begin{enumerate}[leftmargin=*]
		\item If $k-d$ is odd and  $l\in\{1,\dots,k-d+2\}$, then $\widehat{\Theta}_{\SSTS^l}=-\widehat{\Theta}_{\SSTT^l}$.
		\item If $k-d$ is even and  $l\in\{1,\dots,k-d+2\}$, then $\widehat{\Theta}_{\SSTS^l}=0$.
	\end{enumerate}
\end{lem}

\begin{proof}
	We fix $i=d$  and apply \cref{ssalgorithm} to $\SSTS^l$ only. We set
	$$ A=\emptyset, \quad
	B=\{d^{k-d+l+3 },d+1,(d+2)^ {d  } \},
	\quad C=\{ (d+1)^{k-d-l+2},d+2,(d+3)^ {d+1 } \},$$
	so $|D|=k+4$, $|E|=l$. Observe that
	\[
	\alpha^l_{\SSTS_{D,E}}:=
	\prod_{i\geqslant 1}
	\stirling{A_i+D_i}{A_i}
	\stirling{C_i+E_i}{C_i}
	=\stirling{E_{d+1}+k-d-l+2}{k-d-l+2}
	[E_{d+2}+1] \in\{0,1\}.
	\]
	If $E_{d+2}$ is odd then this product is zero since $[E_{d+2}+1]=0$. So suppose that $E_{d+2}$ is even. In particular, we note that if $E_{d+2}=0$, then
	\begin{equation}
	\label{albrl}
	\SSTS^l=\SSTS^l_{D,E} \text{ where } E=\{  d^l \}
	\qquad\text{ and }\qquad
	\SSTT^l=\SSTS^l_{D,E} \text{ where }
	E=\{ d^{l-1},d+1 \}.
	\end{equation}
	Let $n_{D,E}=\sum_{1\leqslant i<j}\left( A_jD_i+C_iE_j \right)$. Then substituting \cref{albrl} into \cref{eqnssalgorithm} we obtain
	\[
	\widehat{\Theta}_{\SSTS^l}+\widehat{\Theta}_{\SSTT^l}+\sum_{\{(D,E)\mid E_{d+2}\geqslant 1\}}(-1)^{n_{D,E}}\alpha^l_{\SSTS_{D,E}}\widehat{\Theta}_{\SSTS^l_{D,E}}=0.
	\]
	
	\smallskip
	\noindent{\bf Case 1:} Suppose that $k-d$ is odd. We proceed to show that $\alpha^l_{\SSTS_{D,E}}\widehat{\Theta}_{\SSTS^l_{D,E}}$ in \cref{eqnssalgorithm} equals zero for all pairs $(D,E)$ except
	in the cases of \cref{albrl}.
	\begin{enumerate}[leftmargin=*]
		\item Assume that $E_{d+1}=1$. In other words, we first consider tableaux $\SSTS^l_{D,E}$ such that
		$
		E=\{  d^{l-r},d+1,(d+2)^{r-1 } \}
		$
		with $1<r\leqslant l$. Then
		\[
		\stirling{E_{d+1}+k-d-l+2}{k-d-l+2}
		=[k-d-l+3]
		=\begin{cases}
		1&\text{ if $l$ is odd,}\\
		0&\text{ if $l$ is even}.
		\end{cases}
		\]
		We now suppose that $l$ is odd and recall that $l\leqslant \bar{t}$;  it thus follows that $\bar{t}\geqslant {3}$. Since $E_{d+2}=r-1$, and we assume above that $E_{d+2}$ is even, it suffices to   consider $\SSTS^l_{D,E}$ such that $r$ is odd. Under this assumption, we can further suppose that $d\geqslant 3$. If $r<l$ (so that $E_d>0$), then rows $d$ and $d+1$ of $\SSTS^l_{D,E}$ are of the form
		\[\scalefont{0.8}
		\begin{tikzpicture}[scale=1.6]
		 \tgyoung(0cm,0cm,;<d>;<d>;\hdts;<d>;<d>!\dgr<d>;<d>;\hdts;<d>;<d>!\wh<d{+}2>;<d{+}2>;\hdts;<d{+}2>,!\dgr<d>;\hdts;<d>!\gr<d{+}1>;\hdts;<d{+}1>!\wh<d{+}2>;<d{+}2>;\hdts;<d{+}2>;<d{+}3>;<d{+}3>;\hdts;<d{+}3>)
		\draw[decoration={brace,mirror,raise=5pt},decorate]	(0.05,-0.43) -- node[below=6pt] {$l-r$ times} (1.35,-0.43);
		\draw[decoration={brace,mirror,raise=5pt},decorate]	(2.8,-0.43) -- node[below=6pt] {$r$ times} (4.55,-0.43);
		\draw[decoration={brace,mirror,raise=5pt},decorate]	(4.65,-0.43) -- node[below=6pt] {$d+1$ times} (6.35,-0.43);
		\draw[decoration={brace,raise=5pt},decorate]	(0.05,0.43) -- node[above=6pt] {$k-d+r+3$ times} (4.55,0.43);
		\draw[decoration={brace,raise=5pt},decorate]	(4.6,0.43) -- node[above=6pt] {$d-r+1$ times} (6.35,0.43);
		\end{tikzpicture}
		\]
		It follows from part (1) of \cref{lem2} that $\widehat{\Theta}_{\SSTS^l_{D,E}}=0$.
		
		We now suppose that $r=l$, and consider the exceptional tableau homomorphism $\SSTR:=\SSTS^l_{D,E}$ with $E=\{ d+1, (d+2)^{l-1} \}$.  We now proceed to show that  $\widehat{\Theta}_{\SSTR}=0$ by repeated application of \cref{ssalgorithm}. Observe that the $(d-1)$th and $d$th rows in $\SSTR$ have the following form
		\[\scalefont{0.8}
		\begin{tikzpicture}[scale=1.6]
		 \tgyoung(0cm,0cm,;<d{-}1>;\hdts;<d{-}1>;<d{-}1>;\hdts;<d{-}1>;<d>!\dgr<d>;<d>;\hdts;<d>!\gr<d{+}1>;<d{+}1>;\hdts;<d{+}1>,!\wh<d>;\hdts;<d>!\dgr<d>;\hdts;<d>;<d>;<d>!\wh<d{+}2>;\hdts<d{+}2>;<d{+}2>;<d{+}2>;\hdts;<d{+}2>)
		\draw[decoration={brace,mirror,raise=5pt},decorate]	(0.05,-0.43) -- node[below=6pt] {$k-d+l+3$ times} (3.65,-0.43);
		\draw[decoration={brace,mirror,raise=5pt},decorate]	(3.7,-0.43) -- node[below=6pt] {$d-l+1$ times} (6.85,-0.43);
		\draw[decoration={brace,raise=5pt},decorate]	(0.05,0.43) -- node[above=6pt] {$k-d+4$ times} (2.7,0.43);
		\draw[decoration={brace,raise=5pt},decorate]	(3.25,0.43) -- node[above=6pt] {$\bar{t}-l$ times} (5,0.43);
		\end{tikzpicture}
		\]
		We now apply \cref{ssalgorithm} to $\widehat{\Theta}_{\SSTR}$ with $i=d-1$. We set
		$$A=\{ (d-1)^{k-d+4} \},\quad   B=\{d^{k-d+\bar{t}+4},(d+1)^{d-\bar{t}+l-1  } \},\quad C=\{  (d+2)^{d-l+1 } \},$$
		so that $|D|=d$, $|E|=k+l-d+3$.
		We write $\sum_{(D,E)}\alpha_{\SSTR_{D,E}}\widehat{\Theta}_{\SSTR_{D,E}}=0$, and observe that $0\neq\alpha_{\SSTR_{D,E}}\in\{-1,1\}$ for all possible pairs $(D,E)$ with $A,B,C$ as above.
		
		Observe that $$ D = \{ d^{\bar{t}-l+u+1}, (d+1)^{d-\bar{t}+l-u-1} \}, \quad
		E= \{ d^{k+l-d+3-u },(d+1)^{u} \}$$ with $u\geq 0$.
		If $u=0$, we observe that $\SSTR_{D,E}=\SSTR$.
		So we can write $$\widehat{\Theta}_{\SSTR} + \sum_{ \{(D,E)\mid E_{d+1}>0\} } \alpha_{{\SSTR}_{D,E}} \widehat{\Theta}_{\SSTR_{D,E}}=0,$$
		where $\alpha_{{\SSTR}_{D,E}}\in\{-1,0,1\}$.
		We now suppose that $u>0$.
		Then rows $d-1$ and $d$ of $\SSTR_{D,E}$ are as follows
		\[\scalefont{0.8}
		\begin{tikzpicture}[scale=1.6]
		 \tgyoung(0cm,0cm,;<d{-}1>;\hdts;<d{-}1>;<d{-}1>;\hdts;<d{-}1>;<d>!\dgr<d>;\hdts;<d>;<d>!\gr<d{+}1>;<d{+}1>;\hdts;<d{+}1>,!\wh<d>;\hdts;<d>!\dgr<d>;<d>;\hdts;<d>!\gr<d{+}1>;<d{+}1>;\hdts;<d{+}1>!\wh<d{+}2>;<d{+}2>;\hdts;<d{+}2>)
		\draw[decoration={brace,mirror,raise=5pt},decorate]	(0.05,-0.43) -- node[below=6pt] {$k-d+l+3-u$ times} (3.15,-0.43);
		\draw[decoration={brace,mirror,raise=5pt},decorate]	(5.1,-0.43) -- node[below=6pt] {$d-l+1$ times} (6.85,-0.43);
		\draw[decoration={brace,raise=5pt},decorate]	(0.05,0.43) -- node[above=6pt] {$k-d+4$ times} (2.7,0.43);
		\draw[decoration={brace,raise=5pt},decorate]	(2.8,0.43) -- node[above=6pt] {$\bar{t}-l+u+1$ times} (5,0.43);
		\draw[decoration={brace,mirror,raise=5pt},decorate]	(3.25,-0.43) -- node[below=6pt] {$u$ times} (5,-0.43);
		\end{tikzpicture}
		\]
		We now show that $\widehat{\Theta}_{\SSTR_{D,E}}=0$ for all $u>0$, treating the cases $u=1$ and $u>1$ separately, and hence $\widehat{\Theta}_{\SSTR}=0$.
		
		%%%%%%%%%%%%%%%%%%%%%%%%%
		
		We first suppose that $u>1$ and write $\SSTR'=\SSTR_{D,E}$ in this case.
		If $l>u+1$ and $u>1$, then we can observe from above that $\SSTR'(d-1,c)=\SSTR'(d,c)=d$ for some column $c$, and hence $\widehat{\Theta}_{\SSTR'}=0$ by \cref{lem5}.
		We now suppose that $l\leqslant u+1$; we observe that rows $d-1$ and $d$ in $\SSTR'$ above are standard. Since $d\geqslant 3$, we see that the $(d-2)$th and $(d-1)$th rows in $\SSTR'$ have the following form
		\[\scalefont{0.8}
		\begin{tikzpicture}[scale=1.6]
		 \tgyoung(0cm,0cm,;<d{-}2>;<d{-}2>;\hdts;<d{-}2>;<d{-}2>;<d{-}1>;<d>;<d>;\hdts;<d>;<d>,;<d{-}1>;<d{-}1>;\hdts;<d{-}1>;<d>!\dgr<d>;\hdts;<d>!\gr<d{+}1>;\hdts;<d{+}1>)
		\draw[decoration={brace,raise=5pt},decorate]	(0.05,0.43) -- node[above=6pt] {$k-d+5$ times} (2.25,0.43);
		\draw[decoration={brace,raise=5pt},decorate]	(2.8,0.43) -- node[above=6pt] {$d-2$ times} (5,0.43);
		\draw[decoration={brace,mirror,raise=5pt},decorate]	(0.05,-0.43) -- node[below=6pt] {$k-d+4$ times} (1.8,-0.43);
		\draw[decoration={brace,mirror,raise=5pt},decorate]	(1.9,-0.43) -- node[below=6pt] {$\bar{t}-l+u+1$ times} (3.65,-0.43);
		\end{tikzpicture}
		\]
		We now apply \cref{ssalgorithm} to $\widehat{\Theta}_{\SSTR'}$ with $i=d-2$. We set
		$$A = \{ (d-2)^{k-d+5}, d-1 \},\quad B = \{ (d-1)^{k-d+4}, d^{\bar{t}+d-l+u-1} \}, \quad C = \{ (d+1)^{d-\bar{t}+l-u-1} \},$$
		so that $|D|=d-2$ and $|E|=k+\bar{t}-l-d+u+5$.
		Observe that each homomorphism $\widehat{\Theta}_{\SSTR'_{D,E}}$ in \cref{ssalgorithm} appears with some scalar multiple of $[D_{d-1}+1]$, which equals zero if and only if $D_{d-1}$ is odd. So suppose that $D_{d-1}$ is even. If $D_{d-1}=0$, then ${\SSTR'}_{D,E}=\SSTR'$. Now suppose that $D_{d-1}>0$ and consider $\SSTR'_{D,E}$ with
		$$ D = \{ (d-1)^{2j}, d^{d-2j-2} \}, \quad E = \{ (d-1)^{k-d-2j+4}, d^{\bar{t}-l+u+2j+1} \},$$
		where $j\geqslant 1$; we abuse notation and let $\SSTR''=\SSTR'_{D,E}$ in this case.
		Now observe that the $(d-2)$th and $(d-1)$th rows of  $\SSTR'_{D,E}$ are
		\[\scalefont{0.8}
		\begin{tikzpicture}[scale=1.6]
		 \tgyoung(0cm,0cm,;<d{-}2>;<d{-}2>;\hdts;<d{-}2>;<d{-}2>;<d{-}1>;\hdts;<d{-}1>;<d>;\hdts;<d>,;<d{-}1>;<d{-}1>;\hdts;<d{-}1>;<d>!\dgr<d>;\hdts;<d>!\gr<d{+}1>;\hdts;<d{+}1>)
		\draw[decoration={brace,raise=5pt},decorate]	(0.05,0.43) -- node[above=6pt] {$k-d+5$ times} (2.25,0.43);
		\draw[decoration={brace,raise=5pt},decorate]	(3.7,0.43) -- node[above=6pt] {$d-2j-2$ times} (5,0.43);
		\draw[decoration={brace,mirror,raise=5pt},decorate]	(0.05,-0.43) -- node[below=6pt] {$k-d-2j+4$ times} (1.8,-0.43);
		\draw[decoration={brace,mirror,raise=5pt},decorate]	(3.65,-0.43) -- node[below=6pt] {$d-\bar{t}+l-u-1$ times} (5,-0.43);
		\end{tikzpicture}
		\]
		Suppose that $\bar{t}+u-2j>l+1$. Then we can apply \cref{ssalgorithm} to $\widehat{\Theta}_{\SSTR''}$ with $i=d-2$. We observe that in this case, we have $A_{d-1}=2j+1$. Hence the only tableau homomorphism $\widehat{\Theta}_{\SSTR''_{D,E}}$ that appears in \cref{ssalgorithm} with a non-zero coefficient is $\widehat{\Theta}_{\SSTR''}$ itself, and hence $\widehat{\Theta}_{\SSTR''}=0$.
		Instead suppose that $l\geqslant \bar{t}+u-2j-1$, and observe rows $d-1$ and $d$ of $\SSTR''$ are of the form
		\[\scalefont{0.8}
		\begin{tikzpicture}[scale=1.6]
		 \tgyoung(0cm,0cm,;<d{-}1>;<d{-}1>;\hdts;<d{-}1>;<d>!\dgr<d>;\hdts;<d>!\gr<d{+}1>;\hdts;<d{+}1>,!\wh<d>;<d>;\hdts;<d>;<d>;<d{+}1>;\hdts;<d{+}1>;<d{+}2>;\hdts;<d{+}2>)
		\draw[decoration={brace,mirror,raise=5pt},decorate]	(3.7,-0.43) -- node[below=6pt] {$d-l+1$ times} (5,-0.43);
		\draw[decoration={brace,mirror,raise=5pt},decorate]	(2.35,-0.43) -- node[below=6pt] {$u$ times} (3.6,-0.43);
		\draw[decoration={brace,raise=5pt},decorate]	(0.05,0.43) -- node[above=6pt] {$k-d-2j+4$ times} (1.8,0.43);
		\draw[decoration={brace,raise=5pt},decorate]	(3.65,0.43) -- node[above=6pt] {$d-\bar{t}+l-u-1$ times} (5,0.43);
		\end{tikzpicture}
		\]
		Since $\bar{t}\geqslant 3$, we have that $l>u-2j+1$. It thus follows that there exists a column $c$ such that $\SSTR''(d-1,c)=\SSTR''(d,c)=d$, and hence $\widehat{\Theta}_{\SSTR''}=0$ by \cref{lem5}.

		%%%%%%%%%%%%%%%%%%%%

		We now suppose that $u=1$; we write $\SSTU=\SSTR_{D,E}$ with $D = \{ d^{\bar{t}-l+2}, (d+1)^{d-\bar{t}+l-2} \}$, $E=\{ d^{k+l-d+2 },d+1 \}$. We now apply \cref{ssalgorithm} to $\widehat{\Theta}_{\SSTU}$ with $i=d-1$.
		We set
		$$ A = \{ (d-1)^{k-d+4} \}, \quad B = \{ d^{k+\bar{t}-d+4}, (d+1)^{d-\bar{t}+l-2} \}, \quad C = \{ d+1, (d+2)^{d-l+1} \},$$
		so that $|D|=d$ and $|E|=k+l-d+2$. Thus each homomorphism $\widehat{\Theta}_{\SSTU_{D,E}}$ appears in \cref{ssalgorithm} with some scalar multiple of $[E_{d+1}+1]$, which equals zero if and only if $E_{d+1}$ is odd. So suppose that $E_{d+1}$ is even. If $E_{d+1}=0$, then $\SSTU_{D,E}=\SSTU$.
		We can thus write $$\widehat{\Theta}_{\SSTR} + \widehat{\Theta}_{\SSTU} + \sum_{ \{(D,E)\mid E_{d+1} \text{ even, } E_{d+1}>1\} } \alpha_{{\SSTR}_{D,E}} \widehat{\Theta}_{\SSTR_{D,E}}=0,$$
		where $\alpha_{{\SSTR}_{D,E}}\in\{-1,1\}$.
		For $E_{d+1}>0$, we let
		$$ D = \{ d^{\bar{t}-l+2m+2}, (d+1)^{d-\bar{t}+l-2m-2} \}, \quad E = \{ d^{k+l-d-2m+2}, (d+1)^{2m} \} $$
		with $m\geqslant 1$, and we write $\SSTU'=\SSTU_{D,E}$ in this case.
		
		By following the above proof to show that $\widehat{\Theta}_{\SSTR_{D,E}}=0$ when $u>1$, we can similarly show that $\widehat{\Theta}_{\SSTU'}=0$ by applying the substitution $u=2m+1$.

		\item Assume that $E_{d+1}=0$. In other words,  consider the tableaux $\SSTS^l_{D,E}$ with
		$$ D = \{ d^{k-d+r+3}, d+1, (d+2)^{d-r} \}, \quad
		E=\{ d^{l-r },(d+2)^{r  } \},
		$$
		where $1\leqslant r\leqslant l$. Since $E_{d+2}=r$, and  $E_{d+2}$ is even (by assumption), it thus suffices to consider $\SSTS^l_{D,E}$ with $r$ even.
		If $l>r$, then it follows from part (2) of \cref{lem2} that $\widehat{\Theta}_{\SSTS^l_{D,E}}=0$.
		We now suppose that $r=l$ is even and consider the tableau homomorphism $\SSTR:=\widehat{\Theta}_{\SSTS^l_{D,E}}$ with
		$$ D = \{ d^{k-d+l+3}, d+1, (d+2)^{d-l} \},\quad E = \{ (d+2)^l \}.$$
		Observe that the $(d-1)$th and $d$th rows of $\SSTR$ are of the form
		\[\scalefont{0.8}
		\begin{tikzpicture}[scale=1.6]
		 \tgyoung(0cm,0cm,;<d{-}1>;<d{-}1>;\hdts;<d{-}1>;<d>!\dgr<d>;<d>;\hdts;<d>!\gr<d{+}1>;<d{+}1>;\hdts;<d{+}1>,!\dgr<d>;<d>;\hdts;<d>!\wh<d>;<d>;\hdts;<d>!\gr<d{+}1>!\wh<d{+}2>;<d{+}2>;\hdts;<d{+}2>)
		\draw[decoration={brace,mirror,raise=5pt},decorate]	(0.05,-0.43) -- node[below=6pt] {$k-d+l+3$ times} (3.6,-0.43);
		\draw[decoration={brace,mirror,raise=5pt},decorate]	(4.2,-0.43) -- node[below=6pt] {$d-l$ times} (5.95,-0.43);
		\draw[decoration={brace,raise=5pt},decorate]	(0.05,0.43) -- node[above=6pt] {$k-d+4$ times} (1.8,0.43);
		\draw[decoration={brace,raise=5pt},decorate]	(2.3,0.43) -- node[above=6pt] {$\bar{t}-l$ times} (4.1,0.43);
		\end{tikzpicture}
		\]
		We now apply \cref{ssalgorithm} to $\widehat{\Theta}_{\SSTR}$ with $i=d-1$. We set
		$$ A = \{ (d-1)^{k-d+4} \}, \quad B = \{ d^{k-d+\bar{t}+4}, (d+1)^{d-\bar{t}+l-1} \}, \quad C = \{ d+1, (d+2)^{d-l} \},$$
		so that $|D|=d$ and $|E|=k+l-d+3$. Observe that $\widehat{\Theta}_{\SSTR_{D,E}}$ appears in \cref{ssalgorithm} with some scalar multiple of $[E_{d+1}+1]$, which is zero if and only if $E_{d+1}$ is odd. So suppose that $E_{d+1}$ is even, and let
		$$ D = \{ d^{k-d+\bar{t}-2m+3}, (d+1)^{2d-k-\bar{t}+2m-3} \}, \quad E = \{ d^{2m+1}, (d+1)^{k+l-d-2m+2}  \}$$
		with $m\geqslant 1$. Observe that $\SSTR_{D,E}=\SSTR$ if $2m=k-d+l+2$. So suppose that $2m<k-d+l$, and write $\SSTR'=\SSTR_{D,E}$ in this case. Then rows $d-1$ and $d$ of $\SSTR'_{D,E}$ are of the form
		\[\scalefont{0.8}
		\begin{tikzpicture}[scale=1.6]
		 \tgyoung(0cm,0cm,;<d{-}1>;<d{-}1>;\hdts;<d{-}1>;<d>;<d>;<d>;\hdts;<d>!\gr<d{+}1>;<d{+}1>;\hdts;<d{+}1>,!\wh<d>;<d>;\hdts;<d>;<d>!\gr<d{+}1>;<d{+}1>;\hdts;<d{+}1>!\wh<d{+}2>;<d{+}2>;\hdts;<d{+}2>)
		\draw[decoration={brace,mirror,raise=5pt},decorate]	(0.05,-0.43) -- node[below=6pt] {$2m+1$ times} (2.2,-0.43);
		\draw[decoration={brace,mirror,raise=5pt},decorate]	(4.2,-0.43) -- node[below=6pt] {$d-l$ times} (5.95,-0.43);
		\draw[decoration={brace,raise=5pt},decorate]	(0.05,0.43) -- node[above=6pt] {$k-d+4$ times} (1.8,0.43);
		\draw[decoration={brace,raise=5pt},decorate]	(1.9,0.43) -- node[above=6pt] {$k-d+\bar{t}-2m+3$ times} (4.1,0.43);
		\end{tikzpicture}
		\]
		Observe that if $2m>k-d+3$, then there exists a column $c$ such that $\SSTR'(d-1,c)=\SSTR'(d,c)=d$. Hence $\widehat{\Theta}_{\SSTR''}=0$ by part two of \cref{lem5}.
		We now suppose that $2m\leqslant k-d+2$ and observe that the $(d-2)$th and $(d-1)$th rows of $\SSTR'$ are of the form
		\[\scalefont{0.8}
		\begin{tikzpicture}[scale=1.6]
		 \tgyoung(0cm,0cm,;<d{-}2>;<d{-}2>;\hdts;<d{-}2>;<d{-}2>;<d{-}1>;<d>;\hdts;<d>;<d>;<d>;\hdts;<d>,;<d{-}1>;<d{-}1>;\hdts;<d{-}1>;<d>;<d>;<d>;\hdts;<d>!\gr<d{+}1>;<d{+}1>;\hdts;<d{+}1>)
		\draw[decoration={brace,raise=5pt},decorate]	(0.05,0.43) -- node[above=6pt] {$k-d+5$ times} (2.2,0.43);
		\draw[decoration={brace,raise=5pt},decorate]	(2.75,0.43) -- node[above=6pt] {$d-2$ times} (5.9,0.43);
		\draw[decoration={brace,mirror,raise=5pt},decorate]	(0.05,-0.43) -- node[below=6pt] {$k-d+4$ times} (1.8,-0.43);
		\draw[decoration={brace,mirror,raise=5pt},decorate]	(2.3,-0.43) -- node[below=6pt] {$k-d+\bar{t}-2m+3$ times} (4.1,-0.43);
		\end{tikzpicture}
		\]
		Since $k-d\geqslant 2m-2$ and $\bar{t}\geqslant 3$, we have $k-d+\bar{t}-2m\geqslant 0$. Hence $\SSTR'(d-2,c)=\SSTR'(d-1,c)=d$ for some column $c$. We can thus apply \cref{ssalgorithm} to $\widehat{\Theta}_{\SSTR'}$ with $i=d-2$. We set
		$$ A = \{ (d-2)^{k-d+5}, d-1 \},\quad B = \{ (d-1)^{k-d+4}, d^{k+\bar{t}-2m+1} \},\quad C = \{ (d+1)^{2d-k-\bar{t}+2m-3} \},  $$
		so that $|D|=d-2$ and $|E|=2k-2d+\bar{t}-2m+7$. We observe that $\widehat{\Theta}_{\SSTR'_{D,E}}$ appears in \cref{ssalgorithm} with some scalar multiple of $[D_{d-1}+1]$, which is zero if and only if $D_{d-1}$ is odd. So suppose that $D_{d-1}$ is even, and write
		$$ D = \{ (d-1)^{2j}, d^{d-2j-2} \}, \quad E = \{ (d-1)^{k-d-2j+4}, d^{k+\bar{t}-d-2m+2j+3} \}$$
		with $j\geqslant 0$.
		We have $\SSTR'_{D,E}=\SSTR'$ when $j=0$. So suppose that $j>0$, and write $\SSTR''=\SSTR'_{D,E}$ in this case. Then the $(d-2)$th and $(d-1)$th rows of $\SSTR''$ are of the form
		\[\scalefont{0.8}
		\begin{tikzpicture}[scale=1.6]
		 \tgyoung(0cm,0cm,;<d{-}2>;<d{-}2>;\hdts;<d{-}2>;<d{-}2>;<d{-}1>;<d{-}1>;\hdts;<d{-}1>;<d>;<d>;\hdts;<d>,;<d{-}1>;<d{-}1>;\hdts;<d{-}1>;<d>;<d>;<d>;\hdts;<d>!\gr<d{+}1>;<d{+}1>;\hdts;<d{+}1>)
		\draw[decoration={brace,raise=5pt},decorate]	(0.05,0.43) -- node[above=6pt] {$k-d+5$ times} (2.2,0.43);
		\draw[decoration={brace,raise=5pt},decorate]	(4.2,0.43) -- node[above=6pt] {$d-2j-2$ times} (5.9,0.43);
		\draw[decoration={brace,mirror,raise=5pt},decorate]	(0.05,-0.43) -- node[below=6pt] {$k-d-2j+4$ times} (1.8,-0.43);
		\draw[decoration={brace,mirror,raise=5pt},decorate]	(4.2,-0.43) -- node[below=6pt] {$2d-k-\bar{t}+2m-3$ times} (5.9,-0.43);
		\end{tikzpicture}
		\]
		If $k+\bar{t}-d-2m-2j\geqslant 0$, then there exists a column $c$ such that $\SSTR''(d-2,c)=\SSTR''(d-1,c)=d$. We can thus apply \cref{ssalgorithm} to $\widehat{\Theta}_{\SSTR''}$ with $i=d-2$.
		We observe that $\widehat{\Theta}_{\SSTR''_{D,E}}$ appears in \cref{ssalgorithm} with some scalar multiple of
		\[
		\stirling{A_{d-1}+D_{d-1}}{A_{d-1}}=\stirling{D_{d-1}+2j+1}{2j+1},
		\]
		which equals zero if and only if $D_{d-1}>0$. If $D_{d-1}=0$, then $\SSTR''_{D,E}=\SSTR''$. Hence $\widehat{\Theta}_{\SSTR_{D,E}}$ is the only homomorphism in the sum of \cref{ssalgorithm}, and thus must be zero.
		We now suppose that $k+\bar{t}-d-2m-2j+1\leqslant 0$, and observe that the $(d-1)$th and $d$th rows of $\SSTR''$ have the form
		\[\scalefont{0.8}
		\begin{tikzpicture}[scale=1.6]
		 \tgyoung(0cm,0cm,;<d{-}1>;<d{-}1>;\hdts;<d{-}1>;<d>;<d>;\hdts;<d>!\gr<d{+}1>;<d{+}1>;\hdts;<d{+}1>,!\wh<d>;<d>;\hdts;<d>;<d{+}1>;<d{+}1>;\hdts;<d{+}1>;<d{+}2>;<d{+}2>;\hdts;<d{+}2>)
		\draw[decoration={brace,mirror,raise=5pt},decorate]	(0.05,-0.43) -- node[below=6pt] {$2m+1$ times} (1.8,-0.43);
		\draw[decoration={brace,mirror,raise=5pt},decorate]	(4.2,-0.43) -- node[below=6pt] {$d-l$ times} (5.45,-0.43);
		\draw[decoration={brace,raise=5pt},decorate]	(0.05,0.43) -- node[above=6pt] {$k-d-2j+4$ times} (1.8,0.43);
		\draw[decoration={brace,raise=5pt},decorate]	(3.7,0.43) -- node[above=6pt] {$2d-k-\bar{t}+2m-3$ times} (5.45,0.43);
		\end{tikzpicture}
		\]
		Since $\bar{t}\geqslant 3$, it follows that $\SSTR''(d-1,c)=\SSTR''(d,c)=d$ for some column $c$. Moreover, since $2m\leqslant k-d+3$, there are $k+l-d-2m+3\geqslant l+1\geqslant 3$ entries equal to $d+1$ in row $d$ of $\SSTR''$. Hence $\widehat{\Theta}_{\SSTR''}=0$ by \cref{lem5}.
		
	\end{enumerate}
	Hence the sum in \cref{ssalgorithm} is $\widehat{\Theta}_{\SSTS^l}+\widehat{\Theta}_{\SSTT^l}=0$ for $k-d$ odd.
	
	\smallskip
	\noindent{\bf Case 2:}  Suppose that $k-d$ is even. Arguing as in the first case, we also deduce that  $\widehat{\Theta}_{\SSTS^l}+\widehat{\Theta}_{\SSTT^l}=0$.
	We shall now apply \cref{ssalgorithm} to $\widehat{\Theta}_{\SSTT^l}$ with $i=d$ and $k-d$ even in order to  show that $\widehat{\Theta}_{\SSTT^l}=0$ and thus obtain the required result.
	We have
	$$A=\emptyset, \quad  B=\{ d^{k+l-d+3 },(d+2)^{d } \},
	\quad C=\{ (d+1)^{k-d-l+3 },d+2,(d+3)^{d+1 } \}$$
	and $|D|=k+4$, $|E|=l-1$. It follows that $\alpha^l_{\SSTT_{D,E}}$ is zero if and only if $E_{d+2}$ is odd. We thus assume that $E_{d+2}$ is even.
	We know that $\SSTT_{D,E}^l=\SSTT^l$ when $E_{d+2}=0$, so it suffices to only consider the pairs $(D,E)$ with
	$$
	D = \{ d^{k-d+2m+4}, (d+2)^{d-2m} \}, \quad
	E=\{ d^{l-2m-1 },(d+2)^{2m } \},
	$$
	where $m\geqslant 1$.
	Now observe that the tableau $\SSTT^l_{D,E}$ with $E=\{ d^{l-2m-1}, (d+2)^{2m} \}$ is similar to $\SSTS^l_{D,E}$ with $E=\{ d^{l-r}, (d+2)^{r} \}$ such that $r$ is even.
	Thus following the argument as in part (2) of the first case with $k-d$ odd, we find that $\widehat{\Theta}_{\SSTT^l_{D,E}}=0$ for all possible pairs $(D,E)$ as above. Hence $\widehat{\Theta}_{\SSTT}=0$  and thus $\widehat{\Theta}_{\SSTS}=0$.
\end{proof}

\begin{rmk}
	We note that neither $\SSTS^{k-d+3}$ nor $\SSTT^0$ are considered in the previous result due to the restriction on $l$. In fact, we consider these two cases to be \emph{exceptional}; we show that these tableaux are well defined for certain $\bar{t}$ and that their corresponding tableau homomorphisms are zero.
\end{rmk}

\begin{lem}\label{lem4}
	Let $\lambda=((k+4)^k)$ and $\mu=((k+2)^k,k^2)$ for some $k>1$, $d\in\{2,\dots,k-1\}$ and $\bar{t}\in\{1,\dots,k+2\}$.
	\begin{enumerate}[leftmargin=*]
		
		\item
		
		Suppose that $d>2$ and $1<\bar{t}\leqslant d-1$.  We have that  $\widehat{\Theta}_{\SSTS}=0$ for $\SSTS\in\RStd(\la,\mu)$   the  tableau   obtained from $\SSTA$ by replacing $\bar{t}$ entries equal to $d+1$ in the $(d-1)$th row with the entry $d$. The tableau $\SSTS$ is identical to $\SSTA$ except for row $d-1$, which is of the form
		\[\scalefont{0.8}
		\begin{tikzpicture}[scale=1.6]
		 \tgyoung(0cm,0cm,;<d{-}1>;<d{-}1>;\hdts;<d{-}1>;<d{-}1>;<d{-}1>;<d>!\dgr<d>;<d>;\hdts;<d>!\gr<d{+}1>;<d{+}1>;\hdts;<d{+}1>)
		\draw[decoration={brace,raise=5pt},decorate]	(0.05,0.43) -- node[above=6pt] {$k-d+4$ times} (2.7,0.43);
		\draw[decoration={brace,raise=5pt},decorate]	(3.2,0.43) -- node[above=6pt] {$\bar{t}$ times} (4.95,0.43);
		\end{tikzpicture}
		\]

		\item
		
		We suppose that $k-d+3\leqslant\bar{t}$.  We have that $\widehat{\Theta}_{\SSTT}=0$ for $\SSTT\in\RStd(\la,\mu)$   the  tableau  obtained from $\SSTA$ by replacing the leftmost $\bar{t}$ entries equal to $d+1$ with  entries equal to $d$.
		The tableau $\SSTT$ is identical to $\SSTA$ except for in the $(d-1)$th, $d$th and $(d+1)$th rows  which are of the form

		\[\scalefont{0.8}
		\begin{tikzpicture}[scale=1.6]
		 \tgyoung(0cm,0cm,;<d{-}1>;<d{-}1>;\hdts;<d{-}1>;<d{-}1>;<d{-}1>;<d>!\dgr<d>;<d>;\hdts;<d>!\gr<d{+}1>;<d{+}1>;\hdts;<d{+}1>,!\wh<d>;<d>;\hdts;<d>;<d>!\dgr<d>!\wh<d{+}2>;<d{+}2>;<d{+}2>;\hdts;<d{+}2>;<d{+}2>;<d{+}2>;\hdts;<d{+}2>,!\dgr<d>;<d>;\hdts;<d>!\wh<d{+}2>;<d{+}3>;<d{+}3>;<d{+}3>;<d{+}3>;\hdts;<d{+}3>;<d{+3}>;<d{+3}>;\hdts;<d{+3}>)
		\draw[decoration={brace,mirror,raise=5pt},decorate]	(0.05,-0.9) -- node[below=6pt] {$k-d+2$ times} (1.8,-0.9);
		\draw[decoration={brace,mirror,raise=5pt},decorate]	(2.3,-0.9) -- node[below=6pt] {$d+1$ times} (6.85,-0.9);
		\draw[decoration={brace,raise=5pt},decorate]	(0.05,0.43) -- node[above=6pt] {$k-d+4$ times} (2.7,0.43);
		\draw[decoration={brace,raise=5pt},decorate]	(3.2,0.43) -- node[above=6pt] {$\bar{t}+d-k-3$ times} (4.95,0.43);
		\end{tikzpicture}
		\]

	\end{enumerate}
\end{lem}

\begin{proof}
	We consider these two exceptional cases separately.
	
	Since $d>2$ and $\bar{t}\geqslant 2$, we have $k\geqslant 4$. Thus $\SSTS$ contains the rows $d-2$, $d-1$ and $d$ for any $d\in\{3,\dots,k-1\}$.	
	
	\smallskip
	\noindent{\bf Case 1:}
	We fix $i=d-2$ and apply \cref{ssalgorithm} to $\widehat{\Theta}_{\SSTS}$.
	We set
	$$A=\{ (d-2)^{k-d+5}, d-1 \},\quad B=\{ (d-1)^{k-d+4}, d^{d+\bar{t}-1} \},\quad C=\{ (d+1)^{d-\bar{t}-1} \},$$ and $|D|=d-2$, $|E|=k+\bar{t}-d+5$. Thus every homomorphism $\widehat{\Theta}_{\SSTS_{D,E}}$ in \cref{eqnssalgorithm} appears with some scalar multiple of $[D_{d-1}+1]$, which is zero if and only if $D_{d-1}$ is odd. So suppose that $D_{d-1}$ is even.
	If $D_{d-1}=0$ then $\SSTS=\SSTS_{D,E}$. Now let $\SSTR^{r}=\SSTS_{D,E}$ with $D_{d-1}=2r$ for some $r\geqslant 1$, so that the pair $(D,E)$ are of the form
	\[
	D=\{ (d-1)^{2r}, d^{d-2r-2} \}, \quad
	E= \{ (d-1)^{k-d-2r+4}, d^{\bar{t}+2r+1} \}.
	\]
	We observe that rows $d-1$ and $d$ of $\SSTR^{r}_{D,E}$ are of the form
	\[\scalefont{0.8}
	\begin{tikzpicture}[scale=1.6]
	 \tgyoung(0cm,0cm,<d{-}1>;\hdts;<d{-}1>;<d>;<d>;<d>;\hdts;<d>!\dgr<d>;\hdts;<d>!\gr<d{+}1>;\hdts;<d{+}1>,!\wh<d>;\hdts;<d>;<d>!\gr<d{+}1>!\wh<d{+}2>;\hdts;<d{+}2>;<d{+}2>;\hdts;<d{+}2>;<d{+}2>;\hdts;<d{+}2>)
	\draw[decoration={brace,raise=5pt},decorate]	(5.1,0.43) -- node[above=6pt] {$d-\bar{t}-1$ times} (6.4,0.43);
	\draw[decoration={brace,raise=5pt},decorate]	(1.4,0.43) -- node[above=6pt] {$\bar{t}+2r+1$ times} (5,0.43);
	\draw[decoration={brace,mirror,raise=5pt},decorate]	(0.05,-0.43) -- node[below=6pt] {$k-d+3$ times} (1.8,-0.43);
	\draw[decoration={brace,mirror,raise=5pt},decorate]	(2.3,-0.43) -- node[below=6pt] {$d$ times} (6.4,-0.43);
	\end{tikzpicture}
	\]
	We can thus write
	$$\widehat{\Theta}_{\SSTS}+\sum_{r\geqslant 1 }\alpha_{\SSTR^{r}}\widehat{\Theta}_{\SSTR^{r}}=0,$$
	where $\alpha_{\SSTR^{r}}\in\{-1,1\}$.
	
	We now apply \cref{ssalgorithm} to $\widehat{\Theta}_{\SSTR^{r}}$ with $i=d-1$. We set
	$$ A= \{ (d-1)^{k-d-2r+4} \},\quad B = \{ d^{k-d+\bar{t}+2r+4}, (d+1)^{d-\bar{t}-1} \},\quad C = \{ d+1, (d+2)^d \},$$ and $|D|=d+2r$, $|E|=k-d+3$. Thus every homomorphism $\widehat{\Theta}_{\SSTR^r_{D,E}}$ in \cref{eqnssalgorithm} appears with some scalar multiple of $[E_{d+1}+1]$, which is zero if and only if $E_{d+1}$ is odd. So suppose that $E_{d+1}$ is even.
	If $E_{d+1}=0$ then $\SSTR^{r}=\SSTR^{r}_{D,E}$. Now let $\SSTU^{r,s}=\SSTR^{r}_{D,E}$ with $E_{d+1}=2s$ for some $s\geqslant 1$, so that the pair $(D,E)$ are of the form
	$$ D = \{ d^{\bar{t}+2r+2s+1}, (d+1)^{d-\bar{t}-2s-1} \},\quad
	E = \{ d^{k-d-2s+3}, (d+1)^{2s}\}.$$
	We can thus write
	$$\widehat{\Theta}_{\SSTR^r}+\sum_{s\geqslant 1 }\alpha_{\SSTU^{r,s}}\widehat{\Theta}_{\SSTU^{r,s}}=0,$$
	where $\alpha_{\SSTU^{r,s}}\in\{-1,1\}$.
	
	We now observe that we lie in one of the following two cases.
	\begin{enumerate}
		\item Suppose that $r\geqslant s+1$. Then $\SSTU^{r,s}(d-1,c)=\SSTU^{r,s}(d,c)=d$ for some column $c$. We can thus apply part one of \cref{lem5} to $\SSTU^{r,s}$ with $v=2s+1$, and hence $\widehat{\Theta}_{\SSTU^{r,s}}=0$.
		
		\item Suppose that $s\geqslant r$. Then rows $d-1$ and $d$ in $\SSTU^{r,s}$ are standard. Now observe that the $(d-2)$th and $(d-1)$th rows of $\SSTU^{r,s}$ are of the form
		\[\scalefont{0.8}
		\begin{tikzpicture}[scale=1.6]
		 \tgyoung(0cm,0cm,<d{-}2>;\hdts;<d{-}2>;<d{-}2>;<d{-}1>;\hdts;<d{-}1><d>;<d>;\hdts;<d>,;<d{-}1>;\hdts;<d{-}1>;<d>;<d>;\hdts;<d>;<d>!\gr<d{+}1>;\hdts;<d{+}1>)
		\draw[decoration={brace,raise=5pt},decorate]	(1.9,0.43) -- node[above=6pt] {$2r+1$ times} (3.15,0.43);
		\draw[decoration={brace,raise=5pt},decorate]	(0.05,0.43) -- node[above=6pt] {$k-d+5$ times} (1.8,0.43);
		\draw[decoration={brace,mirror,raise=5pt},decorate]	(0.05,-0.43) -- node[below=6pt] {$k-d+4-2r$ times} (1.35,-0.43);
		\draw[decoration={brace,mirror,raise=5pt},decorate]	(3.7,-0.43) -- node[below=6pt] {$d-\bar{t}-2s-1$ times} (5,-0.43);
		\end{tikzpicture}
		\]
		Observe that $\SSTU^{r,s}(d-2,c)=\SSTU^{r,s}(d-1,c)=d$ for some column $c$. We now apply \cref{ssalgorithm} to $\widehat{\Theta}_{\SSTU^{r,s}}$ with $i=d-2$. We have
		$$ A = \{ (d-2)^{k-d+5}, (d-1)^{2r+1}\},\quad
		B = \{ (d-1)^{k-d+4-2r} , d^{\bar{t}+d+2s-1}\}\quad
		C = \{ (d+1)^{d-\bar{t}-2s-1} \}.$$
		Thus every homomorphism $\widehat{\Theta}_{\SSTU^{r,s}_{D,E}}$ in  \cref{eqnssalgorithm} appears with some scalar multiple of
		\[
		\stirling{A_{d-1}+D_{d-1}}{A_{d-1}}
		=\stirling{D_{d-1} + 2r +1}{2r+1}
		=\begin{cases}
		1&\text{if $D_{d-1}=0$,}\\
		0&\text{otherwise},
		\end{cases}
		\]
		and hence $\widehat{\Theta}_{\SSTU^{r,s}}=0$ for all $s\geqslant 1$.
		
	\end{enumerate}
	
	Thus $\widehat{\Theta}_{\SSTR^{r}}=0$ for all $r\geqslant 1$, and moreover $\widehat{\Theta}_{\SSTS}=0$.

	\smallskip
	\noindent{\bf Case 2:}
	We fix $i=d$ and apply \cref{ssalgorithm} to $\widehat{\Theta}_{\SSTT}$. We set
	$$A=\emptyset,\quad  B=\{ d^{2k-2d+6 },(d+2)^{d } \},\quad C=\{ d+2,(d+3)^{d+1} \},$$
	and $|D|=k+4$, $|E|=k-d+2$.
	Thus every homomorphism $\widehat{\Theta}_{\SSTT_{D,E}}$ in \cref{eqnssalgorithm} appears with some scalar multiple of $[E_{d+2}+1]$, which is zero if and only if $E_{d+2}$ is odd.
	
	Now assume that $E_{d+2}$ is even. Let $\SSTT_{D,E}$ be the tableau that is identical to $\SSTT$ except for its $d$th and $(d+1)$th rows, which have the following form
	\[\scalefont{0.8}
	\begin{tikzpicture}[scale=1.6]
	 \tgyoung(0cm,0cm,;<d>;<d>;\hdts;<d>;<d>;<d{+}2>;\hdts;<d{+}2>;<d{+}2>;\hdts;<d{+}2>,;<d>;<d>;\hdts;<d>;<d{+}2>;<d{+}2>;\hdts;<d{+}2>;<d{+}3>;\hdts;<d{+}3>)
	\draw[decoration={brace,raise=5pt},decorate]	(2.3,0.43) -- node[above=6pt] {$d-E_{d+2}$ times} (5,0.43);
	\draw[decoration={brace,mirror,raise=5pt},decorate]	(1.9,-0.43) -- node[below=6pt] {$E_{d+2}+1$ times} (3.65,-0.43);
	\draw[decoration={brace,mirror,raise=5pt},decorate]	(3.7,-0.43) -- node[below=6pt] {$d+1$ times} (5,-0.43);
	\end{tikzpicture}
	\]
	Observe that $\SSTT=\SSTT_{D,E}$ when $E_{d+2}=0$, and we can thus write
	$$\widehat{\Theta}_{\SSTT}+\sum_{(D,E)\mid  E_{d+2}>0\}}\alpha_{\SSTT_{D,E}}\widehat{\Theta}_{\SSTT_{D,E} },$$
	where $\alpha_{\SSTT_{D,E}}\in\{-1,1\}$.
	If $E_{d+2}>0$, then $\widehat{\Theta}_{\SSTT_{D,E}}=0$ by setting v=0 in part (1) of \cref{lem2}. Hence the sum in \cref{eqnssalgorithm} contains only the tableau homomorphism $\widehat{\Theta}_{\SSTT}$, which must be zero.
\end{proof}

Thanks to Lyle, we know how to compute the composition of any tableau homomorphism with the map $\varphi_{d,t}$.

\begin{prop}\cite[Proposition 2.14]{MR2339470}
	\label{prop:homcomp}
	Let $\mu,\alpha\vdash n$, $1\leqslant d$, $0\leqslant t<\alpha_{d+1}$ and $\bar{t}=\alpha_{d+1}-t$. Suppose $\SSTS\in\RStd(\mu,\alpha)$ and let $\mathscr{S}\subseteq\RStd(\mu,\nu^{d,t})$ be the set of row standard tableaux obtained by replacing $\bar{t}$ entries of $d+1$ in $\SSTS$ with $d$. For each $\SSTT\in\mathscr{S}$, set
	\begin{itemize}
		\item $\beta_i=|\{j\mid \SSTS(i,j)=d+1\}|-|\{j\mid \SSTT(i,j)=d+1\}|$;
 		\item $x_i=|\{(k,j)\mid k>i\text{ and }\SSTS(k,j)=d\}|$;
		\item $y_i=|\{j\mid \SSTT(i,j)=d\}|$,
	\end{itemize}
	for all $i\geqslant 1$. Then
	\[
	\varphi_{d,t}\circ\Theta_{\SSTS}
	=\sum_{\SSTT\in\mathscr{S}}
 	\alpha_{\SSTT}
	\Theta_{\SSTT},
	\quad
	\text{	where }
	\quad
	\alpha_{\SSTT}=\prod_{i\geqslant 1}q^{x_i\beta_i}\stirling{y_i}{\beta_i}.
	\]
\end{prop}

In our case, we note that $\alpha_{\SSTT}=\pm \stirling{y_i}{\beta_i}$ since $q=-1$.

We are now ready to determine that the semistandard tableau homomorphism $\widehat{\Theta}_{\SSTA}$ belongs to the homomorphism space $\operatorname{Hom}_{H^\Bbbk_q(n)}\left({\bf S}^\Bbbk_q(((k+4)^k)),{\bf S}^\Bbbk_q((k+2)^k,k^2)\right)$.

\begin{thm}
	Let $k> 1$. Then $\varphi_{d,t}\circ \widehat{\Theta}_{\SSTA}=0$ for all $d\geqslant 1$ and $0\leqslant t<\lambda_{d+1}$.
\end{thm}
\begin{proof}
	We apply \cref{prop:homcomp} to $\widehat{\Theta}_{\SSTA}$ for all possible $d$ and $t$.
	
	\smallskip
	\noindent{\bf Case 1:}  Suppose that $d=1$. Then we have $\varphi_{1,t}\circ \widehat{\Theta}_{\SSTA}=\alpha_{\SSTS}\widehat{\Theta}_{\SSTS}+\alpha_{\SSTT}\widehat{\Theta}_{\SSTT}$, where $\SSTS$ is identical to $\SSTA$ except for its first two rows, which are of the form
	\[\scalefont{0.8}
	\begin{tikzpicture}[scale=1.6]
	\tgyoung(0cm,0cm,;1;\hdts;1;1;1;\hdts;1;1!\dgr1!\wh3,!\dgr1;\hdts;1!\gr2;2;\hdts;2!\wh3;4;4).
	\draw[decoration={brace,raise=5pt},decorate]	(0.05,0.43) -- node[above=6pt] {$k+2$ times} (3.65,0.43);
	\draw[decoration={brace,mirror,raise=5pt},decorate]	(0.05,-0.43) -- node[below=6pt] {$\bar{t}-1$ times} (1.35,-0.43);
	\end{tikzpicture}
	\]	
	and where $\SSTT$ is identical to $\SSTA$ except for its first two rows, which are of the form
	\[\scalefont{0.8}
	\begin{tikzpicture}[scale=1.6]
	\tgyoung(0cm,0cm,;1;1;\hdts;1;1;\hdts;1;1!\gr2!\wh3,!\dgr1;1;\hdts;1!\gr2;\hdts;2!\wh3;4;4)
	\draw[decoration={brace,raise=5pt},decorate]	(0.05,0.43) -- node[above=6pt] {$k+2$ times} (3.65,0.43);
	\draw[decoration={brace,mirror,raise=5pt},decorate]	(0.05,-0.43) -- node[below=6pt] {$\bar{t}$ times} (1.8,-0.43);
	\end{tikzpicture}
	\]
	We observe that $\alpha_{\SSTT}=1$ and $\alpha_{\SSTS}=[k+3]=\begin{cases}
	0&\text{if $k$ is odd,}\\
	1&\text{if $k$ is even.}
	\end{cases}$
	
	Hence
	\[
	\varphi_{1,t}\circ\widehat{\Theta}_{\SSTA}
	=\begin{cases}
	\widehat{\Theta}_{\SSTT}
	&\text{if $k$ is odd,}\\
	\widehat{\Theta}_{\SSTS}+\widehat{\Theta}_{\SSTT}
	&\text{if $k$ is even.}
	\end{cases}
	\]
	
	We first apply \cref{ssalgorithm} to $\widehat{\Theta}_{\SSTT}$ with $i=1$.
	We have \[
	A=\emptyset,\quad B=\{1^{k+\bar{t}+2 },2,3\},
	\quad C=\{2^{k-\bar{t}+1 },3,4,4\},
	\]
	and we thus have the following four possible pairs of $(D,E)$:
	$$
	( \{1^{k+2},2,3\}, \{1^{\bar{t}}\})\quad
	( \{1^{k+3},3\} ,  \{1^{\bar{t}-1},2\}) \quad
	( \{1^{k+3},2\} ,  \{1^{\bar{t}-1},3\}) \quad
	( \{1^{k+4}\} ,  \{1^{\bar{t}-2},2,3\}).$$
	Observe that $\SSTT_{D,E}=\SSTT$ where $(D,E)$ is the first pair above.
	We note that $\widehat{\Theta}_{\SSTT_{D,E}}$ appears in \cref{eqnssalgorithm} with some scalar multiple of
	\[
	\stirling{E_2+k-\bar{t}+1}{k-\bar{t}+1}
	\cdot[E_3+1]
	\cdot\stirling{E_4+2}{2}.
	\]
	Hence the homomorphism $\widehat{\Theta}_{\SSTT_{D,E}}$, where $(D,E)$ is one of the last two pairs above, appears with zero coefficient in \cref{eqnssalgorithm} since $[E_3+1]=[2]=0$.
	Also, $\SSTT_{D,E}=\SSTS$ where $(D,E)$ is the second pair above.
	We can thus write $\widehat{\Theta}_{\SSTS}+\widehat{\Theta}_{\SSTT}=0$, and hence
	\[
	\varphi_{1,t}\circ\widehat{\Theta}_{\SSTA}
	=\begin{cases}
	\widehat{\Theta}_{\SSTT}=-\widehat{\Theta}_{\SSTS}
	&\text{if $k$ is odd,}\\
	0
	&\text{if $k$ is even.}
	\end{cases}
	\]
	
	We now suppose that $k$ is odd. If $\bar{t}=1$ and we apply \cref{ssalgorithm} to $\widehat{\Theta}_{\SSTT}$ with $i=1$ as above, we observe that $\widehat{\Theta}_{\SSTT_{D,E}}$ with $(D,E)$ as the second pair above appears in \cref{eqnssalgorithm} with zero coefficient since
	$$
	\stirling{E_2+k-\bar{t}+1}{k-\bar{t}+1}=\stirling{k+1}{k}=[k+1]=0.
	$$
	Hence $\widehat{\Theta}_{\SSTT}=0$ when $\bar{t}=1$.
	
	We now let $\bar{t}>1$ and apply \cref{ssalgorithm} to $\widehat{\Theta}_{\SSTS}$ with $i=1$ as follows.
	We have
	\[
	A=\emptyset,\quad B=\{1^{k+\bar{t}+2 },3\},
	\quad C=\{2^{k-\bar{t}+2 },3,4,4\},
	\]
	and we thus have the following two possible pairs of $(D,E)$:
	$$D=\{1^{k+3},3\},\quad E=\{1^{\bar{t}-1}\}\quad \text{ or }\quad D=\{1^{k+4}\}, \quad E=\{1^{\bar{t}-2},3\}.$$
	Observe that $\SSTS_{D,E}=\SSTS$ where $(D,E)$ is the first pair above. By applying \cref{ssalgorithm} to $\SSTS_{D,E}$ where $(D,E)$ is the second pair above, we see that the corresponding sum in \cref{eqnssalgorithm} contains only a single homomorphism, and hence $\widehat{\Theta}_{\SSTS_{D,E}}=0$.
	It thus follows that $\widehat{\Theta}_{\SSTS}=0$, and hence $\varphi_{1,t}\circ\widehat{\Theta}_{\SSTA}=0$ when $\bar{t}>1$.
	
	\smallskip
	\noindent{\bf Case 2:}  Suppose that $d\in\{2,3,\dots,k-1\}$, so that $\bar{t}\in\{1,2,\dots,k+2\}$. Note that the entries equal to $d+1$ lie in the $(d-1)$th, $d$th and $(d+1)$th rows of $\SSTA$, which have the following form
	\[\scalefont{0.8}
	\begin{tikzpicture}[scale=1.6]
	 \tgyoung(0cm,0cm,;<d{-}1>;<d{-}1>;\hdts;<d{-}1>;<d{-}1>;<d{-}1>;<d>!\gr<d{+}1>;<d{+}1>;\hdts;<d{+}1>,!\wh<d>;<d>;\hdts;<d>;<d>!\gr<d{+}1>!\wh<d{+}2>;<d{+}2>;<d{+}2>;\hdts;<d{+}2>,!\gr<d{+}1>;<d{+}1>;\hdts;<d{+}1>!\wh<d{+}2>!\wh<d{+}3>;<d{+3}>;<d{+3}>;<d{+3}>;\hdts;<d{+3}>)
	\draw[decoration={brace,mirror,raise=5pt},decorate]	(0.05,-0.9) -- node[below=6pt] {$k-d+2$ times} (1.8,-0.9);
	\draw[decoration={brace,mirror,raise=5pt},decorate]	(2.3,-0.9) -- node[below=6pt] {$d+1$ times} (5,-0.9);
	\end{tikzpicture}
	\]
	
	The row-standard tableaux appearing in \cref{eqnssalgorithm} split into the following three cases, depending on $\bar{t}$.
	\begin{enumerate}[leftmargin=*]
		\item Suppose that $\bar{t}\leqslant d-1$. Then there are $2\bar{t}+1$ such tableaux, namely $\SSTR$, $\SSTS^1,\SSTS^2,\dots,\SSTS^{\bar{t}}$ and $\SSTT^1,\SSTT^2,\dots,\SSTT^{\bar{t}}$, where $\SSTR,\SSTS^l,\SSTT^l$ are row-standard $(k^4)$-tableaux of type $((k+2)^k,k^2)$ that are obtained from $\SSTA$ by replacing $\bar{t}$ entries equal to $d+1$ with $d$s. We let $\SSTR$ be identical to $\SSTA$ except for its $(d-1)$th row, which is drawn as in part (1) of \cref{lem4}; let $\SSTS^l$ and $\SSTT^l$ be identical to $\SSTA$ except for their $(d-1)$th, $d$th and $(d+1)$th rows, which are drawn as in parts (1) and (2) of \cref{lem1}, respectively. Then
		\[
		\varphi_{d,t}\circ \widehat{\Theta}_{\SSTA}
		=\alpha_{\SSTR}\widehat{\Theta}_{\SSTR}
		+\sum_{l=1}^{\bar{t}}
		\left(\alpha_{\SSTS^l}\widehat{\Theta}_{\SSTS^l}
		+\alpha_{\SSTT^l}\widehat{\Theta}_{\SSTT^l}\right),
		\]
		where $\alpha_{\SSTR},\alpha_{\SSTS^l},\alpha_{\SSTT^l}\in\{-1,0,1\}$.
		
		\item Suppose that $d\leqslant \bar{t}\leqslant k-d+2$. Then there are $2\bar{t}$ such tableaux, namely $\SSTS^1,\SSTS^2,\dots,\SSTS^{\bar{t}}$ and $\SSTT^1,\SSTT^2,\dots,\SSTT^{\bar{t}}$, where $\SSTS^l$ and $\SSTT^l$ are defined as in \cref{lem1}. Then
		\[
		\varphi_{d,t}\circ \widehat{\Theta}_{\SSTA}
		=
		\sum_{l=1}^{\bar{t}}\left(\alpha_{\SSTS^l}\widehat{\Theta}_{\SSTS^l}
		+\alpha_{\SSTT^l}\widehat{\Theta}_{\SSTT^l}\right),
		\]
		where $\alpha_{\SSTS^l},\alpha_{\SSTT^l}\in\{-1,0,1\}$.
		
		\item Suppose that $\bar{t}\geqslant k-d+3$.  Then there are $2(k-\bar{t})+5$ such tableaux, namely $\SSTU$, $\SSTS^{\bar{t}-d+1}$, $\SSTS^{\bar{t}-d+2},\dots,\SSTS^{k-d+2}$ and $\SSTT^{\bar{t}-d+1},\SSTT^{\bar{t}-d+2},\dots,\SSTT^{k-d+2}$. We let $\SSTU$ be defined to be the tableau that is identical to $\SSTA$ except for its $(d-1)$th, $d$th and $(d+1)$th rows, which are drawn as in the second part of \cref{lem4}; we let $\SSTS^l$ and $\SSTT^l$ be defined as in \cref{lem1}. Then
		\[
		\varphi_{d,t}\circ \widehat{\Theta}_{\SSTA}
		=\alpha_{\SSTU}\widehat{\Theta}_{\SSTU}
		+\sum_{l=\bar{t}-d+1}^{k-d+2}\left(\alpha_{\SSTS^l}\widehat{\Theta}_{\SSTS^l}
		+\alpha_{\SSTT^l}\widehat{\Theta}_{\SSTT^l}\right),
		\]
		where $\alpha_{\SSTU},\alpha_{\SSTS^l},\alpha_{\SSTT^l}\in\{-1,0,1\}$.
		
	\end{enumerate}
	
	We now observe the scalars of the tableau homomorphisms of $\SSTR$, $\SSTS^l$, $\SSTT^l$ and $\SSTU$ appearing in \cref{prop:homcomp} are as follows:
	$$
	\textstyle {  \alpha_{\SSTR}\	 =\pm [\bar{t}+1], 	\  	
		\alpha_{\SSTS^l}	 =[\bar{t}-l+1], \
		\alpha_{\SSTT^l}	 =[\bar{t}-l+1]    [k-d+4], \
		\alpha_{\SSTU}\ =\pm[\bar{t}-k+d-2]    [k-d+4].
	}
	$$
	
	Observe that
	\[
	[k-d+4]=\begin{cases}
	0&\text{if $k-d$ is even},\\
	1&\text{if $k-d$ is odd};
	\end{cases}\quad\quad
	[\bar{t}+1]=\begin{cases}
	0&\text{if $\bar{t}$ is odd},\\
	1&\text{if $\bar{t}$ is even}.
	\end{cases}
	\]
	
	We thus suppose that $\bar{t}$ is even. Then we know from part (1) of \cref{lem4} that $\widehat{\Theta}_{\SSTR}=0$ when $\bar{t}\leqslant d-1$.
	
	Suppose that $k-d$ is odd. Then $\alpha_{\SSTS^l}=[\bar{t}-l+1]=\alpha_{\SSTT^l}$ for all $l$. Moreover, we know from part (1) of \cref{lem1} that $\widehat{\Theta}_{\SSTS^l}=-\widehat{\Theta}_{\SSTT^l}$ for all $l$. Also, $\alpha_{\SSTU}=[\bar{t}-k+d-2]$, and when this coefficient is non-zero, it follows from part (2) of \cref{lem4} that $\widehat{\Theta}_{\SSTU}=0$.
	Hence, for all $\bar{t}$,
	\[
	\varphi_{d,t}\circ \widehat{\Theta}_{\SSTA}
	 =\sum_l\left(\alpha_{\SSTS^l}\widehat{\Theta}_{\SSTS^l}+\alpha_{\SSTT^l}\widehat{\Theta}_{\SSTT^l}\right)
	=\sum_l\alpha_{\SSTS^l}\left(\widehat{\Theta}_{\SSTS^l}-\widehat{\Theta}_{\SSTS^l}\right)
	=0.
	\]
	
	We now suppose that $k-d$ is even, in which case both $\alpha_{\SSTU}=0$ and $\alpha_{\SSTT^l}=0$. Moreover, we know from part (2) of \cref{lem1} that $\widehat{\Theta}_{\SSTS^l}=0$ for all $l$, and hence $\varphi_{d,t}\circ\widehat{\Theta}_{\SSTA}=0$ for all $\bar{t}$.

	\smallskip
	\noindent{\bf Case 3:}   Suppose that $d=k$. Then we have $\varphi_{k,t}\circ \widehat{\Theta}_{\SSTA}=\alpha_{\SSTS}\widehat{\Theta}_{\SSTS}+\alpha_{\SSTT}\widehat{\Theta}_{\SSTT}$ with $\alpha_{\SSTS},\alpha_{\SSTT}\in\{-1,0,1\}$, where $\SSTS$ and $\SSTT$ are as follows: $\SSTS$ is identical to $\SSTA$ except for its last two rows, which are of the form
	\[\scalefont{0.8}
	\begin{tikzpicture}[scale=1.6]
	 \tgyoung(0cm,0cm,;<k{-}1>;<k{-}1>;<k{-}1>;<k{-}1>;<k>!\dgr<k>;\hdts;<k>!\gr;<k{+}1>;<k{+}1>;\hdts;<k{+}1>,!\wh<k>;<k>;<k>!\dgr<k>!\wh<k{+}2>;<k{+}2>;\hdts;<k{+}2>;<k{+}2>;<k{+}2>;\hdts;<k{+}2>)
	\draw[decoration={brace,raise=5pt},decorate]	(2.3,0.43) -- node[above=6pt] {$\bar{t}-1$ times} (3.6,0.43);
	\draw[decoration={brace,mirror,raise=5pt},decorate]	(1.85,-0.43) -- node[below=6pt] {$k$ times} (5.45,-0.43);
	\end{tikzpicture}
	\]
	and $\SSTT$ is identical to $\SSTA$ except for its last two rows, which are of the form
	\[\scalefont{0.8}
	\begin{tikzpicture}[scale=1.6]
	 \tgyoung(0cm,0cm,;<k{-}1>;<k{-}1>;<k{-}1>;<k{-}1>;<k>!\dgr<k>;<k>;\hdts;<k>!\gr<k{+}1>;\hdts;<k{+}1>,!\wh<k>;<k>;<k>!\gr<k{+}1>!\wh<k{+}2>;<k{+}2>;<k{+}2>;\hdts;<k{+}2>;<k{+}2>;\hdts;<k{+}2>)
	\draw[decoration={brace,raise=5pt},decorate]	(2.3,0.43) -- node[above=6pt] {$\bar{t}$ times} (4.1,0.43);
	\draw[decoration={brace,mirror,raise=5pt},decorate]	(1.85,-0.43) -- node[below=6pt] {$k$ times} (5.45,-0.43);
	\end{tikzpicture}
	\]
	We observe that $\alpha_{\SSTS}=[\bar{t}]  \cdot  [4]=0$ and $\alpha_{\SSTT}=[\bar{t}+1]=\begin{cases}
	0&\text{if $\bar{t}$ is odd},\\
	1&\text{if $\bar{t}$ is even}.
	\end{cases}$
	
	We thus suppose that $\bar{t}$ is even. We now apply \cref{ssalgorithm} to $\widehat{\Theta}_{\SSTT}$ with $i=k-2$.
	We set
	$$A=\{(k-2)^5,k-1 \},\quad B=\{k-1,k-1,k-1,k-1,k^{k+\bar{t}-1 }\},\quad C=\{ (k+1)^{k-\bar{t}-1 } \},$$	and we thus have the following five possible pairs of $(D,E)$:
	$$
	(\{k^{k-2 }\} , \{(k-1)^4,k^{\bar{t}+1 } \}) \quad
	(\{k-1,k^{k-3 }\} , \{(k-1)^3,k^{\bar{t}+2 } \})  \quad
	(\{(k-1)^2,k^{k-4 }\} , \{(k-1)^2,k^{\bar{t}+3 } \}) $$
	$$		(\{(k-1)^3,k^{k-5 }\} , \{k-1,k^{\bar{t}+4 } \}) \quad
	(\{(k-1)^4,k^{k-6 }\} , \{k^{\bar{t}+5 } \}.$$
	Observe that $\SSTT_{D,E}=\SSTT$ where $(D,E)$ is the first pair above. Each homomorphism $\widehat{\Theta}_{\SSTT_{D,E}}$ appears in \cref{eqnssalgorithm} with coefficient $\pm[D_{k-1}+1]$, which is zero when $(D,E)$ is either the second or the fourth pair above.
	We let $\SSTR:=\SSTT_{D,E}$ where $(D,E)$ is the third pair above, and $\SSTU:=\SSTT_{D,E}$ where $(D,E)$ is the fifth pair above. We can thus write $\widehat{\Theta}_{\SSTT}+\alpha_{\SSTR}\widehat{\Theta}_{\SSTR}+\alpha_{\SSTU}\widehat{\Theta}_{\SSTU}=0$, where $\alpha_{\SSTR},\alpha_{\SSTU}\in\{-1,1\}$.
	
	We first apply \cref{ssalgorithm} to $\widehat{\Theta}_{\SSTR}$ with $i=k-2$. We set
	$$
	A = \{ (k-2)^5, (k-1)^3 \},\quad
	B = \{ (k-1)^2, k^{k+\bar{t}-1} \},\quad
	C = \{ (k+1)^{k-\bar{t}-1} \},
	$$
	and we thus have the following three possible pairs of $(D,E)$:
	$$
	( \{ k^{k-4} \}, \{ (k-1)^2, k^{\bar{t}+3} \} )\quad
	( \{ k-1, k^{k-5} \}, \{ k-1, k^{\bar{t}+4} \} )\quad
	( \{ (k-1)^2, k^{k-6} \}, \{ k^{\bar{t}+5} \} ).
	$$
	Observe that $\SSTR_{D,E}=\SSTT$ where $(D,E)$ is the first pair above. The other two homomorphisms $\widehat{\Theta}_{\SSTT_{D,E}}$ appear in \cref{eqnssalgorithm} with zero coefficient since $\stirling{D_{k-1}+3}{3}$ is zero if and only if $D_{k-1}>0$. It thus follows that $\widehat{\Theta}_{\SSTR}=0$.
	
	We now apply \cref{ssalgorithm} to $\widehat{\Theta}_{\SSTU}$ with $i=k-2$. The only homomorphism that appears in \cref{eqnssalgorithm} is $\widehat{\Theta}_{\SSTU}$ since $B=\{k^{\bar{t}+k-1}\}$, and hence $\widehat{\Theta}_{\SSTU}$ must be zero.
	
	Finally, since both $\widehat{\Theta}_{\SSTR}$ and $\widehat{\Theta}_{\SSTU}$ are zero homomorphisms, so is $\widehat{\Theta}_{\SSTT}$.

	\smallskip
	\noindent{\bf Case 4:}  Suppose that $d=k+1$. Then we have $\varphi_{k+1,t}\circ\widehat{\Theta}_{\SSTA}=\alpha_{\SSTT}\widehat{\Theta}_{\SSTT}$ with $\alpha_{\SSTT}\in\{-1,0,1\}$, and where $\SSTT$ is identical to $\SSTA$ except for its last row. Observe that the last two rows of $\SSTT$ are of the form
	\[\scalefont{0.8}
	\begin{tikzpicture}[scale=1.6]
	 \tgyoung(0cm,0cm,;<k{-}1>;<k{-}1>;<k{-}1>;<k{-}1>!\gr<k>!\wh<k{+}1>;\hdts;<k{+}1>;<k{+}1>;<k{+}1>;\hdts;<k{+}1>,!\gr<k>;<k>;<k>!\wh<k{+}1>!\dgr<k{+}1>;<k{+}1>;\hdts;<k{+}1>!\wh<k{+}2>;<k{+}2>;\hdts;<k{+}2>)
	\draw[decoration={brace,mirror,raise=5pt},decorate]	(1.85,-0.43) -- node[below=6pt] {$\bar{t}$ times} (3.65,-0.43);
	\end{tikzpicture}
	\]
	We have $\alpha_{\SSTT}=[\bar{t}+1]=\begin{cases}
	0&\text{if $\bar{t}$ is odd},\\
	1&\text{if $\bar{t}$ is even}.
	\end{cases}$
	
	So suppose that $\bar{t}$ is even and apply \cref{ssalgorithm} to $\widehat{\Theta}_{\SSTT}$ with $i=k-1$.
	We set
	$$A=\{k-1^4,k \},\quad B=\{k^3,(k+1)^{k+\bar{t} }\},\quad C=\{ (k+2)^{k-\bar{t} } \},$$
	and we thus have the following four possible pairs of $(D,E)$:
	$$		(\{ (k+1)^{k-1 } \}, \{ k^3,(k+1)^{\bar{t}+1 } \})\quad
	(\{ k,(k+1)^{k-2 } \}), \{ k^2,(k+1)^{\bar{t}+2 } \})$$
	$$
	(\{ k^2,(k+1\})^{k-3 } \}), \{ k,(k+1)^{\bar{t}+3 } \})\quad
	(\{ k^3,(k+1)^{k-4 } \}, \{ (k+1)^{\bar{t}+4 } \}).$$
	Observe that $\SSTT=\SSTT_{D,E}$ when $(D,E)$ is the first pair above. Each homomorphism $\widehat{\Theta}_{\SSTT_{D,E}}$ appears in \cref{eqnssalgorithm} with coefficient $\pm[D_k+1]$, which is zero when $(D,E)$ is either the second or the fourth pair above. We now let $\SSTS=\SSTT_{(D,E)}$, where $(D,E)$ is the third pair above, whose last two rows are of the form
	\[\scalefont{0.8}
	\begin{tikzpicture}[scale=1.6]
	 \tgyoung(0cm,0cm,;<k{-}1>;<k{-}1>;<k{-}1>;<k{-}1>!\gr<k>;<k>;<k>!\wh<k{+}1>;<k{+}1>;<k{+}1>;\hdts;<k{+}1>,!\gr<k>!\wh<k{+}1>;<k{+}1>;<k{+}1>!\dgr<k{+}1>;<k{+}1>;\hdts;<k{+}1>!\wh<k{+}2>;<k{+}2>;\hdts;<k{+}2>)
	\draw[decoration={brace,mirror,raise=5pt},decorate]	(1.85,-0.43) -- node[below=6pt] {$\bar{t}$ times} (3.65,-0.43);
	\end{tikzpicture}
	\]
	We can thus write $\alpha_{\SSTS}\widehat{\Theta}_{\SSTS}+\widehat{\Theta}_{\SSTT}=0$, where $\alpha_{\SSTS}\in\{-1,1\}$.
	
	\noindent{\bf Case 4a:}
	Assume that $\bar{t}=2$.
	Let $r\geqslant k$, and suppose that $\SSTT'$ (respectively, $\SSTS'$) is the $((k+4)^r)$-tableau whose first $k$ rows agree with $\SSTT$ (respectively, $\SSTS$) and the remaining rows are arbitrary.
	Now observe that the $(k-1)$th and $k$th rows of both $\SSTT$ and $\SSTT'$ in this case are of the form
	\[\scalefont{0.8}
	\begin{tikzpicture}[scale=1.6]
	 \tgyoung(0cm,0cm,;<k{-}1>;<k{-}1>;<k{-}1>;<k{-}1>!\gr<k>!\wh<k{+}1>;<k{+}1>;\hdts;<k{+}1>,!\gr<k>;<k>;<k>!\wh<k{+}1>!\dgr<k{+}1>;<k{+}1>!\wh<k{+}2>;\hdts;<k{+}2>)
	\draw[decoration={brace,mirror,raise=5pt},decorate]	(2.75,-0.43) -- node[below=6pt] {$k-2$ times} (4.1,-0.43);
	\draw[decoration={brace,raise=5pt},decorate]	(2.35,0.43) -- node[above=6pt] {$k-1$ times} (4.1,0.43);
	\end{tikzpicture}
	\]
	We proceed to show that $\widehat{\Theta}_{\SSTT'}=0$ by induction on $k$. Note that $\widehat{\Theta}_{\SSTT'}$ is the only homomorphism appearing in \cref{eqnssalgorithm} when $k=2$, and hence $\widehat{\Theta}_{\SSTT'}=0$. Now suppose that $k>2$, and observe that the $(k-2)$th, $(k-1)$th and $k$th rows of $\SSTS'$ are of the form
	\[\scalefont{0.8}
	\begin{tikzpicture}[scale=1.6]
	 \tgyoung(0cm,0cm,;<k{-}2>;<k{-}2>;<k{-}2>;<k{-}2>;<k{-}2>;<k{-}1>;<k>;<k>;\hdts;<k>,<k{-}1>;<k{-}1>;<k{-}1>;<k{-}1>!\gr<k>;<k>;<k>!\wh<k{+}1>;\hdts;<k{+}1>,!\gr<k>!\wh<k{+}1>;<k{+}1>;<k{+}1>!\dgr<k{+}1>;<k{+}1>!\wh<k{+}2>;<k{+}2>;\hdts;<k{+}2>)
	\draw[decoration={brace,raise=5pt},decorate]	(2.8,0.43) -- node[above=6pt] {$k-2$ times} (4.5,0.43);
	\end{tikzpicture}
	\]
	By the inductive hypothesis, $\widehat{\Theta}_{\SSTS'}=0$ and hence $\widehat{\Theta}_{\SSTT}=0$. It thus follows that $\widehat{\Theta}_{\SSTT}$ is also zero.
	
	\noindent{\bf Case 4b:}
	We now assume that $\bar{t}\geqslant 4$. We apply \cref{ssalgorithm} to $\widehat{\Theta}_{\SSTS}$ with $i=k-1$ as follows. We set
	\[
	A=\{(k-1)^4,k^3\},\quad
	B=\{k,(k+1)^{k+\bar{t}}\},\quad
	C=\{(k+2)^{k-\bar{t}}\},
	\]
	and we thus have the following two possible pairs of $(D,E)$:
	$$ (\{(k+1)^{k-3}\}, \{k,(k+1)^{\bar{t}+3}\})\qquad
	(\{k,(k+1)^{k-4}\},\{(k+1)^{\bar{t}+4}\}).$$
	Observe that $\SSTS=\SSTS_{D,E}$ where $(D,E)$ is the first pair. Also, the homomorphism $\widehat{\Theta}_{\SSTS_{D,E}}$, where $(D,E)$ is the second pair above, appears in the sum of \cref{ssalgorithm} with zero coefficient since $\stirling{D_k+3}{3}$ is zero if and only if $D_k>0$. Hence $\widehat{\Theta}_{\SSTS}$ is zero, and moreover, so is $\widehat{\Theta}_{\SSTT}$.
\end{proof}

\subsection{Proof part $(ii)$: the decomposition numbers}

We now calculate the graded decomposition multiplicity $[{\bf S}^\CC_{-1}(\alpha_k):
{\bf D}^\CC_{-1}(\alpha_k^C)]$  and hence prove that ${\bf S}^\CC_{-1}(\alpha_k)$ is decomposable.

\begin{prop}\label{max2}
For $k\geq 1$, we have that
$
\sharp\{{\rm CStd}(\alpha_k,\alpha_k^C)\}=2
$.    The first tableau is simply given by adding the nodes  in weakly increasing order in the most dominant possible position in $\alpha_k$.  The second tableau is the conjugate of the first.
\end{prop}
\begin{proof}
The entries from the set $\{2,\dots,k+1\}$ are all forced and the only tableau is that of shape $(k,k-1,\dots,2,1)$ and weight  $(k,k-1,\dots,2,1)$.   There are now $k$ possible choices for how to add the nodes  with entry $k+2$ (this is because we have $k$ such nodes  and $k+1$ addable nodes all of the correct residue).
 Regardless of how these nodes  are added, there are no choices for how to add the nodes  with entry $k+3$.
Having added these nodes , we now have a tableau   which is one of   $k$ distinct possible shapes: namely those obtained from $(k,k-1,\dots,1)$ by adding $k$ $2$-dominoes  either of shape $(2)$ or of shape $(1^2)$ (this is precisely the set of  2-separated partitions from this block).  Of these partitions, only two of them have a total of $k$ distinct addable nodes of the form $\{(r,c)\mid (r,c)\in \alpha_k\}$ of the required residue (namely those  of the form $(  k,k-1,\dots,1)+ 2(1^k)$ and its transpose).
We now proceed to add the nodes  with entry $k+5$ to these two possible partitions.
There are now {\em no} choices for where to add these nodes; continuing in this fashion there are no more choices to be made and the result follows.
 \end{proof}

\begin{eg}
Let $k=3$.  The two elements of ${\rm CStd} ((5^3,3^2),(9,7,5))$ are as follows
$$\SSTS_1=\gyoung(2;3;4;5;6,3;4;5;6;7,4;5;6;7;8,7;8;9,8;9;\ten)
\qquad
\SSTS_2=\gyoung(2;3;4;7;8,3;4;5;8;9,4;5;6;9;\ten,5;6;7,6;7;8).$$
 \end{eg}

\begin{prop}\label{hardtab}
Let $\SSTT \in {\rm CStd}(\alpha_k^R,\alpha_k^C)$  be the tableau constructed as follows:
\begin{itemize}[leftmargin=*]
\item[$(i)$] Add the nodes  with ladder number $\ell\in\{2,\dots,k+3\}$ in weakly increasing order in the most dominant possible position.
\item[$(ii)$] Add the first $k-1$ of the nodes  of ladder number $k+4$ in the most dominant possible position.
\item[$(iii)$] Add the final node   of ladder number $k+4$ in the least dominant  position: namely $(k+1,1)$.

 \item[$(iv)$] At this point, there remain $k$ distinct ladders numbers $\{k+4 + i\mid  1\leq i \leq k\}$ to add; which we add in weakly increasing order.
  For $1\leq i\leq k-2$, we add
 the first $k-1-i$ nodes  in the most dominant positions possible and the final
  two nodes    in the least dominant positions  possible: namely  $(k+2-i,2i-1)$ and $(k+1-i,2i)$.

\item[$(v)$] For ladder numbers $\ell=2k+3$ (respectively $\ell=2k+4$) there are precisely 2 nodes  (respectively 1 node) of this ladder number, which we add in the only available positions:  namely $(4,2k-2)$ and  $(4,2k-3)$ (respectively $(2,2k)$).
\end{itemize}
For $\SSTT\in  {\rm CStd}(\alpha_k^C,\alpha_k^R)$ constructed as above, we have that $\deg(\SSTT)=0$.
 \end{prop}
\renewcommand{\sts}{{\sf s}}
\renewcommand{\stt}{{\sf t}}
\begin{proof}
 For $(a,b)$ a node added in steps $(i)$ and $(ii)$, we have that $\mathcal{A}_\SSTS(a,b)=\varnothing = \mathcal{R}_\SSTS(a,b)$.

We have that the one node added in step $(iii)$ has
$\mathcal{A}_\SSTS(a,b)=\{(k,k)\}$  and $ \mathcal{R}_\SSTS(a,b)=\varnothing$.

Fix $1\leq i \leq k$ in step $(iv)$.
The most dominant nodes of ladder number $k+4+i$ each have $\mathcal{A}_\SSTS(a,b)=\varnothing = \mathcal{R}_\SSTS(a,b)$.
The two least dominant nodes added in step $(iv)$ have
$\mathcal{A}_\SSTS(a,b) = \{(k-i-1,2i+4)\}$	and 	$	\varnothing = \mathcal{R}_\SSTS(a,b)=\{(k-i,2i+1)\}$.

In step $(v)$, for the two nodes $(a,b)$  of ladder number $2k+3$ we have that
$\mathcal{A}_\SSTS(a,b) = \{(2,2k)\}$	and 	$	\varnothing = \mathcal{R}_\SSTS(a,b)=\{(1,2k+1)\}$.
The node $(a,b)$ of ladder number $2k+4$ has
$\mathcal{A}_\SSTS(a,b) = \varnothing$	and $\mathcal{R}_\SSTS(a,b) = \{(1,2k+1)\}$.

The only nodes with non-zero degree contribution are the nodes $(k+1,1)$ (added in step $(iii)$) and
the node $(2,2k)$ (added in step $(v)$).  The degree contributions of these nodes are $+1$ and $-1$ respectively.  The result follows.
 \end{proof}

\begin{eg}Let $k=2,3,4$. The semistandard coloured tableau   described in the \cref{hardtab} is as follows.  We have coloured the nodes  for which the degree contribution is $1$ and $-1$, all other nodes  contribute nothing to the relative degree.    $$\newcommand{\seven}{{\bf7}}
\newcommand{\eight}{{\bf8}}
\newcommand{\six}{{\bf6}}
\newcommand{\tten}{{\bf10}}
\newcommand{\ttwelve}{{\bf12}}
 \gyoung(2;3;4;5;6,3;4;5!\dgr8,;6!\wh7,7)\quad
\gyoung(2;3;4;5;6;7;8,3;4;5;6;7!\dgr<10>,!\wh4;5;6;9,!\dgr7!\wh8;9,8)
\quad
\gyoung(2;3;4;5;6;7;8;9;\ten,3;4;5;6;7;8;9!\dgr<12>,!\wh4;5;6;7;8;\eleven,5;6;7;\ten;\eleven,!\dgr8!\wh9;\ten,9)
.$$
\end{eg}
\begin{prop}\label{regular}
For $k\geq 1$ we have that $
\sharp\{{\rm CStd}(\alpha_k,\alpha_k^R)\}=1
$ and we let $\SSTT$ denote the unique tableau in this set.  We have that
$
\deg(\SSTT)= w(\alpha_k)/2.
$

\end{prop}
\begin{proof}
 This follows simply by counting    nodes moved under James' regularisation map.
For a proof (for all $e$-singular partitions) in the language of Fock-spaces, see \cite[Theorem 2.2]{Fay07}.
 \end{proof}

Now, with the combinatorics in place we are ready to calculate decomposition numbers.
 By \cref{max2}, we know that
 $$
[{\bf S}^\CC_{-1}(\alpha_k):{\bf D}^\CC_{-1}(\alpha_k^C)]=a t^{ w(\alpha_k)/2}
 $$
 for $a =0,1$ or $2$.  By \cref{louise}, we know that $a \neq 0$.
Now, we know that
$$
[{\bf S}^\CC_{-1}(\alpha_k):{\bf D}^\CC_{-1}(\alpha_k^R)]=t^{ w(\alpha_k)/2}
$$
by \cref{regular} and that
$$e(\alpha^C)( {\bf D}^\CC_{-1}(\alpha_k^R))\neq0$$
by \cref{hardtab} and the fact that  $d_{\alpha_k^C,\alpha_k^R}\in t\NN_0[t]$.
 Therefore $a=1$ (as the other vector of weight $\alpha_k^R$ is in the span of  the composition factor  ${\bf D}^\CC_{-1}(\alpha_k^R)$ of ${\bf S}^\CC_{-1}(\alpha_k)$).
We thus obtain:

 \begin{thm}\label{thm:decompalpha}
Let $k\in\mathbb{N}$. Then the Specht module ${\bf S}^\CC_{-1}(\alpha_k)$ is decomposable with ${\bf D}^\CC_{-1}(\alpha_k^C)$ appearing as a  direct summand of ${\bf S}^\CC_{-1}(\alpha_k)$.
  \end{thm}

\begin{proof}
Since ${\bf D}^\CC_{-1}(\alpha_k^C)$ appears in the socle of ${\bf S}^\CC_{-1}(\alpha_k)$ and as a decomposition number with multiplicity 1, the result follows as   ${\bf S}^\CC_{-1}(\alpha_k)$ is self-dual (as  the partition $\alpha_k$ is self-conjugate).
\end{proof}

\subsection{The decomposability of Specht modules labelled by $\beta_k$ for $k\in 2\NN$}
We now proceed to prove that the second exceptional family of partitions $\left\{\beta_k\ |\ k\in 2\mathbb{N}+1 \right\}$ also label decomposable Specht modules. For $k\in 2\mathbb{N}+1$, we let $\beta_k$ denote the partition $$\beta_k^C=(2k+3,2k+1,2k-1,\dots,9,7,6,1).$$

\begin{eg}\label{eg:summands}
The Young diagrams of the partitions $\alpha_5^C$ and $\beta_5^C$, along with their residues, are drawn as follows:
\[
[\alpha_5^C]=\gyoung(;0;1;0;1;0;1;0;1;0;1;0;1;0,;1;0;1;0;1;0;1;0;1;0;1,;0;1;0;1;0;1;0;1;0,;1;0;1;0;1;0;1,;0;1;0;1;0)
\qquad
[\beta_5^C]=\gyoung(;0;1;0;1;0;1;0;1;0;1;0;1;0,;1;0;1;0;1;0;1;0;1;0;1,;0;1;0;1;0;1;0;1;0,;1;0;1;0;1;0;1,;0;1;0;1;0!\dgr1,;1)
\]
\end{eg}

Notice that $\beta_5^C$ is obtained from $\alpha_5^C$ by adding certain nodes of residue $1$ to $[\alpha_5^C]$; we will see that these nodes are predetermined by the \emph{Modular Branching Rule}.
There exist \emph{cyclotomic divided power $i$-restriction and $i$-induction functors}, $e_i^{(r)}$ and $f_i^{(r)}$ respectively (we refer the reader to \cite{bk09} for details), that act on $H^\CC_q(n)$-modules and enable us to move between different blocks of $H^\CC_q(n)$. We will describe the action of $f_i^{(r)}$ on both Specht modules and irreducible $H^\CC_q(n)$-modules using the following combinatorics (the action of $e_i^{(r)}$ can be analogously described).

Let $i\in\ZZ/e\ZZ$ and $\la\vdash n$.
We define the \emph{$i$-signature of $\la$} by reading the Young digram $[\la]$ from the top of the first component down to the bottom of the last component, writing a $+$ for each addable node and writing a $-$ for each removable node, where the leftmost $+$ corresponds to the highest addable node of $\la$. We obtain the \emph{reduced $i$-signature of $\la$} by successively deleting all adjacent pairs $+-$ from the $i$-signature of $\la$, always of the form $-\dots-+\dots+$. The addable nodes of residue $i$ corresponding to the $+$ signs in the reduced $i$-signature of $\la$ are called the \emph{conormal} nodes of residue $i$ of $\la$.
 We now define
\begin{itemize}
	\item $\la^{\vartriangle_i}$ to be the partition obtained from $\la$ by adding all addable nodes of residue $i$,
	\item $\la^{\blacktriangle_i}$ to be the partition obtained from $\la$ by adding all conormal nodes of residue $i$.
\end{itemize}
Denote the total number of addable \emph{(respectively conormal)} nodes of residue $i$ of $\la$ by $\text{add}_i(\la)$ \emph{(respectively $\text{\rm conor}_i(\la)$)}.

\begin{eg}
The $1$-signature of $\alpha_5^C$ in \cref{eg:summands} is $+-+-++$, and removing all pairs $+-$, we obtain the reduced $1$-signature $++$. Observe that $(6,1)$ and $(5,6)$ are the conormal nodes of residue $1$ of $\alpha_5^C$, and upon adding these to $\alpha_5^C$ we have $(\alpha_5^C)^{\blacktriangle_1}=\beta_5^C$.
\end{eg}

We now introduce versions of the \emph{Branching Rules}, which will allow us to determine decomposability of ${\bf S}^\CC_{-1}(\beta_k)$ from ${\bf S}^\CC_{-1}(\alpha_k)$; we first observe how the functor $f_i^{(r)}$ acts on Specht modules.
The following  is a generalisation of \cite[Theorem 3.2]{MR3250450}.

\begin{prop}\label{branch1}
Let $i\in\ZZ/e\ZZ$ and $\la\vdash n$ and let $r=|{\text{\rm add}_i(\la)}|$ and suppose that ${\rm Rem}_i(\lambda)=\emptyset$ and $f_i^r {\bf S}^\CC_q(\la) \cong {\bf S}^\CC_q(\la^{\vartriangle_i})$.
 Then ${\bf S}^\CC_q(\la)$ is indecomposable (respectively simple) if and only if ${\bf S}^\CC_q(\la^{\vartriangle_i})$ is indecomposable (respectively simple).  \end{prop}
\begin{proof}
This follows by exactness of the functors $e_i^{(r)}$ and $f_i^{(r)}$.
\end{proof}

The following result is a natural corollary of \cref{thm:decompalpha}.

\begin{thm}
Let $k\in 2\mathbb{N}+1$. Then the Specht module ${\bf S}^\CC_{-1}(\beta_k)$
is decomposable with ${\bf D}^\CC_{-1}(\beta_k^C)$ appearing as a  direct summand of ${\bf S}^\CC_{-1}(\beta_k)$.
\end{thm}

\begin{proof}
First observe that $(\alpha_k)^{\vartriangle_1}=\beta_k$, and we thus have $f_1^{|\text{add}_i(\la)|} {\bf S}^\CC_{-1}(\alpha_k)\cong{\bf S}^\CC_{-1}(\beta_k)$ by applying \cref{branch1}. It follows from \cref{thm:decompalpha} that ${\bf S}^\CC_{-1}(\beta_k)$ is decomposable.

Moreover, we know from \cref{thm:decompalpha} that ${\bf D}^\CC_{-1}(\alpha_k^C)$ is a direct summand of ${\bf S}^\CC_{-1}(\alpha_k)$. Now observe that $\alpha_k^C$ has $k$ (odd) non-zero parts, where the $r$th row of $\alpha_k^C$ has:
\begin{itemize}
	\item an addable node of residue $1$ when $r\in\{1,3,5,\dots,k\}\cup\{k+1\}$,
	\item a removable node of residue $1$ when $r\in\{2,4,6,\dots,k-1\}$.
\end{itemize}
Hence $\alpha_k^C$ has $1$-signature
\[
\underbrace{(+-)(+-)\dots(+-)}_{(k-1)/2\text{ times}}++,
\]
and thus reduced $1$-signature $++$, corresponding to the conormal nodes of residue $1$ $(k+1,1)$ and $(k,6)$. Hence $(\alpha_k^C)^{\blacktriangle_1}=\beta_k^C$.  Now, we observe that
$$  f_i^{\text{conor}_i(\la)}{\bf D}^\CC_q(\la) \cong {\bf D}^\CC_q(\la^{\blacktriangle_i})
$$
follows trivially from Kleshchev's branching rule;  therefore $f_1^{\text{conor}_1(\alpha_k^C)}{\bf D}^\CC_{-1}(\alpha_k^C) \cong {\bf D}^\CC_{-1}(\beta_k^C)$ is a direct summand of ${\bf S}^\CC_{-1}(\beta_k)$.
\end{proof}

  Analogous to \cref{conj:exceptional}, we predict that the Specht modules ${\bf S}_{-1}^\CC (\beta_k)$ are  semisimple.

\begin{conj} For all $k\in 2\mathbb{N}+1$, we expect that
	$${\bf S}_{-1}^\CC (\beta_k) =
	{\bf D}_{-1}^\CC (\beta_k^R)\langle w(\beta_k)/2 \rangle \oplus  {\bf D}_{-1}^\CC (\beta_k^C)\langle w(\beta_k)/2 \rangle.$$
\end{conj}

\bibliographystyle{amsalpha}
% \bibliography{master}
%%\bibliography{dominance}

\bibliographystyle{amsalpha}

\end{document}